\documentclass{amsart}
\usepackage{amsfonts}
\usepackage{amsmath}
\usepackage{amssymb}
\usepackage{graphicx}

\usepackage[bindingoffset=0.2in,%
            left=1.25in,right=1.25in,top=1.25in,bottom=1.375in,%
            footskip=.25in]{geometry}

\usepackage{amsthm}
\usepackage{latexsym}
\usepackage{fancyhdr}
\usepackage{amscd}
\usepackage{lscape}
\usepackage{tikz}
\usepackage{float}
\usepackage{mathtools}
\usepackage{lipsum}
\usepackage{bm}
\usepackage{subfigure}
\usepackage{booktabs}
\usepackage{grffile}
\usepackage{array}
\usepackage{caption}
\usepackage{longtable}
\newcolumntype{C}[1]{>{\centering\arraybackslash}p{#1}}
\definecolor{navy}{HTML}{2F729C} 
\usepackage[hyperfootnotes=false, colorlinks, linkcolor={blue}, citecolor={red}, filecolor={blue}, urlcolor={blue}, plainpages=false, pdfpagelabels]{hyperref}
\newtheorem{theorem}{Theorem}[section]
\newtheorem{lemma}[theorem]{Lemma}
\newtheorem{proposition}[theorem]{Proposition}
\newtheorem{corollary}[theorem]{Corollary}
\theoremstyle{definition}

\newtheorem{example}[theorem]{Example}

\newenvironment{claim}[1][Claim]{\begin{trivlist}
\item[\hskip \labelsep {\bfseries #1}]}{\end{trivlist}}

\theoremstyle{remark}
\newtheorem{remark}[theorem]{Remark}

\numberwithin{equation}{section}

\begin{document}

\title[Minimal models of rational elliptic curves with
non-trivial torsion]{Minimal models of rational elliptic curves \\ with
non-trivial torsion}

\author{Alexander J. Barrios}

\address{Department of Mathematics and Statistics, Carleton College, Northfield, Minnesota 55057}

\email{abarrios@carleton.edu}

%    General info
\subjclass{Primary 11G05, 11G07}


%\keywords{Arithmetic Geometry, Elliptic Curves, Number Theory}

\begin{abstract}
In this paper, we explicitly classify the minimal discriminants of all elliptic curves $E/\mathbb{Q}$ with a non-trivial torsion subgroup. This is done by considering various parameterized families of elliptic curves with the property that they parameterize all elliptic curves $E/\mathbb{Q}$ with a non-trivial torsion point. We follow this by giving admissible change of variables, which give a global minimal model for $E$. We also provide necessary and sufficient conditions on the parameters of these families to determine the primes at which $E$ has additive reduction. In addition, we use these parameterized families to give new proofs of results due to Frey and Flexor-Oesterl\'{e} pertaining to the primes at which an elliptic curve over a number field $K$ with a non-trivial $K$-torsion point can have additive reduction.
\end{abstract}
\maketitle
\setcounter{tocdepth}{1}
\tableofcontents

\section{Introduction\label{chSS}}
The two-parameter family of elliptic curves $F\!\left(  a,b\right)
:y^{2}=x\left(  x+a\right)  \left(  x+b\right)  $, known as the Frey curve,
has the property that if $a$ and $b$ are relatively prime integers, then the
minimal discriminant $\Delta_{F}^{\text{min}}$ of the rational elliptic curve
$F=F\!\left(  a,b\right)  $ is easily computable \cite[Exercise 8.23]{MR2514094}. In fact,
\begin{equation}
\Delta_{F}^{\text{min}}=u^{-12}\left(  4ab\left(  a-b\right)  \right)
^{2}\qquad\text{where }\qquad u=\left\{
\begin{array}
[c]{cl}%
2 & \text{if }a\equiv0\ \operatorname{mod}16\ \text{and\ }b\equiv
1\ \operatorname{mod}4,\\
1 & \text{otherwise.}%
\end{array}
\right.  \label{Freycurve1}%
\end{equation}
Moreover, $F(a,b)$ is a global minimal model for $F$ if $u=1$ and
\begin{equation}\label{Freycurve3}
F^{\prime}:y^{2}+xy=x^{3}+\frac{a+b-1}{4}x^{2}+\frac{ab}{16}x
\end{equation}
is a global minimal model if $u=2$. Having explicit knowledge of a global minimal model for the Frey curve also allows us to easily determine the primes at which $F$ has additive reduction. In fact, $F$ is
semistable at all odd primes and has additive reduction at $2$ if and only if
$u=1$.

In addition to its use in the proof of Fermat's Last Theorem \cite{MR853387}, \cite{MR1333035}, the Frey curve has been a well-studied family due to its easily computable global minimal model. Knowing such a model allows us to bypass the standard means by which a global minimal model is computed for a rational elliptic curve $E$. For instance, if $E$ is given by an integral Weierstrass model, then a global minimal model for $E$ can be computed through the algorithms of Tate~\cite{MR0393039} or Laska~\cite{MR637305}. Laska's algorithm was later improved by Connell to make use of a theorem of Kraus~\cite{MR1024419}. This algorithm is known today as the Laska-Kraus-Connell Algorithm~\cite[Section~3.2]{MR1628193}, and it is the algorithm commonly used for computing a global minimal model. Moreover, knowing a global minimal model for $E$ allows for faster implementation of Tate's algorithm. Knowledge of a global minimal model also aids one in the study of Szpiro's conjecture \cite{MR642675}, which relates the minimal discriminant of an elliptic curve to its conductor. In fact, Oesterl\'{e} made use of~\eqref{Freycurve1} in his proof that the $abc$ Conjecture is equivalent to the modified Szpiro conjecture~\cite{MR992208}. On a related note, Wong \cite{MR1832985} showed that the $abc$ Conjecture holds for almost all Frey curves.

In this paper, we find global minimal models for all rational elliptic curves with a non-trivial torsion point. Specifically, we study fifteen parameterized families of elliptic
curves with the property that if $E$ is a rational elliptic curve with a
non-trivial torsion subgroup, then $E$ is $\mathbb{Q} $-isomorphic to an
elliptic curve occurring in one of these families (see Proposition
\ref{rationalmodels}). For each of these families, we prove results analogous
to those mentioned above for the Frey curve, which, in turn, allows us to
classify the minimal discriminants of all rational elliptic curves with a
non-trivial torsion subgroup.

More precisely, let $E$ be an elliptic curve over a number field $K$. The
Mordell-Weil Theorem states that the group $E\!\left(  K\right)  $ of
$K$-rational points on $E$ is finitely generated. Thus $E\!\left(  K\right)
\cong E\!\left(  K\right)  _{\text{tors}}\times\mathbb{Z}^{r}$ where
$E\!\left(  K\right)  _{\text{tors}}$ is the torsion subgroup of $E\!\left(
K\right)  $ and $r$ is a non-negative integer. Over $\mathbb{Q}$, Mazur's
Torsion Theorem \cite{MR488287} states that there are exactly $15$ possible
torsion subgroups $E\!\left(  \mathbb{Q}\right)  _{\text{tors}}$, and
parameterizations of the corresponding modular curves $X_{1}\!\left(
N\right)  $ and $X_{1}\!\left(  2,N\right)  $ are well known \cite[Table~3]%
{MR0434947}. We use these parameterizations in Section \ref{SecParamCurv} to construct families of elliptic curves $E_{T}$ where $T$ is one of the fourteen
non-trivial torsion subgroups allowed by Mazur's Torsion Theorem. 

The families $E_T$ (see Table \ref{ta:ETmodel}) have the property that
they parameterize all elliptic curves $E/K$ such that $T\hookrightarrow E\!\left(  K\right)$. When $T$ has a point of order greater than $3$, $E_T$ is a two-parameter family attained from the universal elliptic curve over the corresponding modular curve $X_1(N)$ or $X_1(2,N)$ (see Theorem~\ref{LemmaunivincT} and Lemma~\ref{ss:Lem1}). When $T=C_3$, where $C_{N}$ denotes the cyclic group of order $N$, we must consider two different
families of elliptic curves, $E_{C_3}$ and $E_{C_3}^0$. Indeed, we must take special care with the fiber of $X_1(3)\rightarrow X(1)$ corresponding to elliptic curves with $j$-invariant~$0$. In fact, $E_{C_3}^0$ is a one-parameter family of $j$-invariant~$0$ elliptic curves corresponding to the stacky point of $X_1(3)$. Next, let $Y$ denote the complement of this stacky point in $X_1(3)$. Then $E_{C_3}$ is a two-parameter family, which is derived from the universal elliptic curve over $Y$ \cite[Section~3.3]{MR3680373}.
For $T=C_2,C_2\times C_2$, $E_T$ is a three-parameter family with the property that $E_{C_2}$ parameterizes all elliptic curves with a~$2$-torsion point, but no full~$2$-torsion. Similarly, $E_{C_2 \times C_2}$ parameterizes all elliptic curves with full~$2$-torsion. In particular, the Frey curve is a special case of our
three-parameter family of elliptic curves $E_{C_2 \times C_2}$. We note that the families $E_{C_2}$ and $E_{{C_2}\times C_2}$ were also studied in \cite{MR1929230} and \cite{MR1424534}, respectively. 

We denote the invariants~$c_4$ and $c_6$ associated to $E_T$ by $\alpha_T$ and $\beta_T$, respectively, and their expressions can be found in Tables~\ref{ta:alpT}~and~\ref{ta:betT}, respectively. Similarly, the discriminant of $E_T$ is denoted by $\gamma_T$, and these quantities are found in Table~\ref{ta:gamT}. In
Section~\ref{SecFreyFlex}, we study the quantities $\alpha_T$, $\beta_T$, and $\gamma_T$ to deduce the following theorem:

\begin{claim}
[Theorem 1.]\textit{Let $E$ be an elliptic curve over a number field $K$ and let $R_{K}$ denote its ring of integers. Suppose further that $T\hookrightarrow
E\!\left(  K\right)  $ for one of the $T$ listed below. If $E$ has additive
reduction at a prime $\mathfrak{p}$ of $K$, then the residue characteristic of $\mathfrak{p}$ is a
rational prime in the set $S$ given below:}
\[%
\begin{tabular}
[c]{ccccccccccc}\hline
$T$ & $C_{5}$ &$C_{6}$& $C_{7}$ & $C_{8}$ & $C_{9}$ & $C_{10}$ & $C_{12}$ &
$C_{2}\times C_{4}$ & $C_{2}\times C_{6}$ & $C_{2}\times C_{8}$\\\hline
$S$ & $\left\{  5\right\}$& $\left\{  2,3\right\}  $ & $\left\{  7\right\}  $ & $\left\{  2\right\}  $
& $\left\{  3\right\}  $ & $\left\{  5\right\}  $ & $\left\{  2,3\right\}  $ &
$\left\{  2\right\}  $ & $\left\{  2,3\right\}  $ & $\left\{  2\right\}
$\\\hline
\end{tabular}
\]
\end{claim}

\noindent This is Theorem \ref{Thm1} and as a consequence, we attain a special case of Frey's Theorem \cite{MR0457444}. Specifically, Frey proved that if $E\!\left(  K\right)  $ contains a point of
order $\ell$ for $\ell$ a prime greater than~$3$, then $E$ is semistable at
all primes $\mathfrak{p}$ of $K$ whose residue field has a characteristic different from~$\ell$. We note that Theorem 1 extends Frey's Theorem to elliptic curves that contain $C_8,C_9,$ or $C_2\times C_4$ in their torsion subgroup.
Using Theorem 1 and Frey's Theorem, we provide a new proof of the following result of Flexor and Oesterl\'{e} \cite{MR1065153}:
\begin{claim}
[Theorem 2.]\textit{Let $E$ be an elliptic curve over a number field $K$. If $E\!\left(  K\right)
$ contains a point of order $N$ and $E$ has additive reduction at a prime
$\mathfrak{p}$ of $K$ whose residue characteristic does not divide $N$, then
$N\leq4$. Moreover, if $E$ has additive reduction at at least two primes of
$K$ with different residue characteristics, then $N$ divides $12$.}
\end{claim}
This is Theorem \ref{FOes}. The current proofs in the literature of Frey's Theorem, as well as  Flexor and Oesterl\'{e}'s Theorem, rely on knowledge of the N\'{e}ron
model of $E$. The same holds true for their generalizations \cite{MR1343556}. Our proofs of Theorems~1~and~2 differ in that they rely on the properties of the parameterized families $E_T$. In the case when $K=\mathbb{Q}$, we consider the converse to Theorem~1 and prove:
\begin{claim}
[Theorem 3.]\textit{There are necessary and sufficient conditions on the parameters of
$E_{T}$ to determine the primes at which $E_{T}$ has additive reduction.}
\end{claim}

\noindent This is Theorem \ref{semis}, where it is given in its explicit form. As a consequence, we deduce that if $E/\mathbb{Q}$ has additive reduction at at least two primes, then $N$ divides $4$ or $6$. In particular, the condition that $N$ divides $12$ in Theorem~2 is not sharp when $K=\mathbb{Q}$. Another consequence of Theorem~3 is Corollary~\ref{semistableconditions}, which gives necessary and sufficient conditions on the parameters of $E_{T}$ to determine when $E_{T}$ is semistable. In particular, we have that if $E/\mathbb{Q}$ is an elliptic curve with $E\!\left(  \mathbb{Q}\right)  _{\text{tors}} \cong C_2\times C_8 $, then $E$ is semistable. This phenomenon is not realized over general number fields. Indeed, by Theorem 1, if $E$ is defined over a number field $K$ with $C_2 \times C_8 \hookrightarrow E(K) $ and $E$ has additive reduction at a prime $\mathfrak{p}$, then $\mathfrak{p}|2$. 
We see this is the case when $K=\mathbb{Q}(\sqrt{-2})$ and $E:y^2+\sqrt{-2}xy+\sqrt{-2}y=x^3+5x-22$ (\href{https://www.lmfdb.org/EllipticCurve/2.0.8.1/144.2/a/1}{LMFDB label 2.0.8.1-144.2-a1} \cite{lmfdb}) since $E/K$ has additive reduction at the prime $(\sqrt{-2})$ and has $E(K)_{\text{tors}}\cong C_2\times C_8$.

The proof of Theorem~3 relies on an understanding of the minimal discriminant of $E_T$. This investigation begins in Sections~\ref{sec:class}, where we restrict our attention to rational elliptic
curves with a non-trivial torsion point. Understanding the minimal
discriminants of these elliptic curves is equivalent to understanding the
minimal discriminants of the parameterized families $E_{T}$ (see Proposition~\ref{rationalmodels}). Our main result
is a classification of the minimal discriminants of these families~$E_{T}$,
which in turn, generalizes the aforementioned result of the Frey curve
(\ref{Freycurve1}).

\begin{claim}
[Theorem 4.]\textit{There are necessary and sufficient conditions on the parameters of
$E_{T}$ which uniquely determine the minimal discriminant of $E_{T}$.}
\end{claim}
This is Theorem \ref{semistablecondthm} and Section \ref{pfmainthm} is devoted
to its proof. To illustrate Theorems~3~and~4, let us suppose that $E$ is a rational elliptic curve with full $2$-torsion. Then, by Proposition~\ref{rationalmodels},
there are integers $a,b,d$ with $\gcd\!\left(  a,b\right)  =1$, $a$ even, and
$d>0$ squarefree such that $E$ is $\mathbb{Q}$-isomorphic to
% \[
$E_{C_{2}\times C_{2}}=E_{C_{2}\times C_{2}}\!\left(  a,b,d\right)
:y^{2}=x^{3}+\left(  ad+bd\right)  x^{2}+abd^{2}x.$
% \]
Observe that if $d=1$, then we are in the setting of the Frey curve. By
Theorem 4, the minimal discriminant of $E_{C_{2}\times C_{2}}$ is
$u^{-12}d^{6}\left(  4ab\left(  a-b\right)  \right)  ^{2}$ where%
\[
u=\left\{
\begin{array}
[c]{cl}%
2 & \text{if }v_{2}\!\left(  a\right)  \geq4\text{ and }bd\equiv
1\ \operatorname{mod}4,\\
1 & \text{if }v_{2}\!\left(  a\right)  \leq3\text{ or }bd\not \equiv
1\ \operatorname{mod}4.
\end{array}
\right.
\]
By Theorem 3, $E_{C_{2}\times C_{2}}$ has additive reduction at a prime $p$
if and only if $\left(  i\right)  $ $p|d$ or $\left(  ii\right)  $ $p=2$ and
$u=1$.
Consequently, $E_{C_{2}\times C_{2}}$ is semistable if and only if $d=1,v_2(a)\geq 4,$ and $b\equiv 1\operatorname{mod}4$.
Moreover, $E_{C_{2}\times C_{2}}$ is a global minimal model if $u=1$. If
$u=2$, then
\[
y^{2}+xy=x^{3}+\frac{ad+bd-1}{4}x^{2}+\frac{abd^{2}}{16}x
\]
is a global minimal model for $E_{C_{2}\times C_{2}}$. More generally, we
prove the following result in Section~\ref{sec:GMM}:
\begin{claim}
[Theorem 5.]\textit{There are sufficient conditions on the parameters of $E_{T}$ to
determine a global minimal model for $E_{T}$.}
\end{claim}

This is Theorem \ref{GlobalMinModel}. Section \ref{secaddred} is devoted to the proof of Theorem 3. Section \ref{sec:corexa} concludes this paper by
considering some corollaries of Theorems~3~and~4 for rational elliptic curves
with a non-trivial torsion subgroup as well as some examples.

Lastly, we note that if $E$ is a rational elliptic curve with a $3$-torsion point, then $E$ is $\mathbb{Q}$-isomorphic to $E^{\prime}:y^2+axy+by=x^3$ for some integers $a,b$ with the property that for each prime $p$, $p\nmid a$ and $p^3 \nmid b$. Under these assumptions, $E^{\prime}$ is a global minimal model for $E$ \cite{MR491734}. Moreover, in an unpublished work of Koike \cite[Lemma 1.3]{MR491734}, necessary and sufficient conditions on $a$ and $b$ were found to determine the primes at which $E$ has additive reduction. In particular, Theorems 3, 4, and 5 have already been established for $T=C_3$. We note that our two-parameter family $E_{C_3}(a,b)$ differs from $E^{\prime}$, and in addition, the parameters of $E_{C_3}(a,b)$ are only required to be relatively prime.

\subsection{Notation and Terminology}

We start by recalling some basic facts about elliptic curves. See
\cite{MR2514094} for details. We say $E$ is defined over a field $K$ if $E$ is
given by an (affine) Weierstrass model%
\begin{equation}
E:y^{2}+a_{1}xy+a_{3}y=x^{3}+a_{2}x^{2}+a_{4}x+a_{6}\label{ch:inintroweier}%
\end{equation}
where each $a_{j}\in K$. From the Weierstrass coefficients, one defines the
quantities%
\begin{equation}%
\begin{array}
[c]{l}%
c_{4}=a_{1}^{4}+8a_{1}^{2}a_{2}-24a_{1}a_{3}+16a_{2}^{2}-48a_{4},\\
c_{6}=-\left(  a_{1}^{2}+4a_{2}\right)  ^{3}+36\left(  a_{1}^{2}%
+4a_{2}\right)  \left(  2a_{4}+a_{1}a_{3}\right)  -216\left(  a_{3}^{2}%
+4a_{6}\right)  ,\\
\Delta=\frac{c_{4}^{3}-c_{6}^{2}}{1728},\qquad\text{and}\qquad j=\frac
{c_{4}^{3}}{\Delta}.
\end{array}
\label{basicformulas}%
\end{equation}
We call $\Delta$ the \textit{discriminant} of $E$, and the assumption that $E$
is an elliptic curve is equivalent to $\Delta\neq0$. The quantities $c_{4}$
and $c_{6}$ are the \textit{invariants associated to the Weierstrass model} of
$E$. In particular, we have the identity $1728\Delta=c_{4}^{3}-c_{6}^{2}$. An
elliptic curve $E^{\prime}$ is $K$\textit{-isomorphic} to $E$ if $E^{\prime}$
arises from $E$ via an \textit{admissible change of variables} $x\longmapsto
u^{2}x+r$ and $y\longmapsto u^{3}y+u^{2}sx+w$ for $u,r,s,w\in K$ and $u\neq0$.
If $\Delta^{\prime},j^{\prime},c_{4}^{\prime},$ and $c_{6}^{\prime}$ denote
the quantities associated to the Weierstrass model of $E^{\prime}$, then%
\[
\Delta^{\prime}=u^{-12}\Delta,\qquad j^{\prime}=j,\qquad c_{4}^{\prime}%
=u^{-4}c_{4},\qquad c_{6}^{\prime}=u^{-6}c_{6}.
\]

Now suppose $K$ is a number field or a local field with ring of integers
$R_{K}$. We say $E$ is given by an \textit{integral Weierstrass model} if each
$a_{i}\in R_{K}$. Suppose further that $K$ is a number field and let
$\mathfrak{p}$ be a prime of $K$. Then $E$ is $K_{\mathfrak{p}}$-isomorphic to
an elliptic curve given by a \textit{minimal Weierstrass model at
}$\mathfrak{p}$ of the form \eqref{ch:inintroweier} with each $a_{j}\in
R_{K_{\mathfrak{p}}}$ and $v_{\mathfrak{p}}\!\left(  \Delta\right)  $ is
minimal over all elliptic curves $K_{\mathfrak{p}}$-isomorphic to $E$. Here
$v_{\mathfrak{p}}$ denotes the $\mathfrak{p}$-adic valuation of $K_{\mathfrak{p}}$. We call the discriminant associated to a minimal
Weierstrass model of $E$ at $\mathfrak{p}$ the \textit{minimal discriminant of
$E$ at $\mathfrak{p}$}, and denote it by $\Delta_{E/K_{\mathfrak{p}}%
}^{\text{min}}$. If $v_{\mathfrak{p}}(  \Delta_{E/K_{\mathfrak{p}}
}^{\text{min}})  ,\ v_{\mathfrak{p}}\!\left(  c_{4}\right)  >0$ where
$c_{4}$ is the invariant associated to a minimal Weierstrass model of $E$ at $\mathfrak{p}$,
then $E$ is said to have \textit{additive reduction at $\mathfrak{p}$}. If
this is not the case, then $E$ is \textit{semistable at }$\mathfrak{p}$, and
we say $E$ is \textit{semistable} if it is semistable at all primes
$\mathfrak{p}$ of $K$.

Now suppose $K$ is a number field with class number one. Then $E$ is
$K$-isomorphic to an elliptic curve given by a \textit{global minimal model}
of the form \eqref{ch:inintroweier} with each $a_{j}\in R_{K}$ and for each
prime $p$ of $K$, $v_{p}\!\left(  \Delta\right)  =v_{p}(
\Delta_{E/K_{p}}^{\text{min}})  $. We call the discriminant associated
to a global minimal model of $E$ the \textit{minimal discriminant of }$E$ and
denote it by $\Delta_{E/K}^{\text{min}}$. Next, let $c_{4}$ denote the invariant
associated to a global minimal model of $E$. Then $E$ has additive reduction
at a prime $p$ if and only if $p$ divides $\gcd\!\left(  c_{4},\Delta
_{E/K}^{\text{min}}\right)  $. Lastly, we say $E$ is a \textit{rational elliptic
curve} if $E$ is defined over~$\mathbb{Q}$.

\section{Parameterization of Certain Elliptic Curves with non-Trivial
Torsion\label{SecParamCurv}}

Let $K$ be a number field with ring of integers $R_{K}$ and let $E$ be the
elliptic curve given by the Weierstrass model%
\begin{equation}
E:y^{2}+a_{1}xy+a_{3}y=x^{3}+a_{2}x^{2}+a_{4}x+a_{6} \label{ch:SSintwei}%
\end{equation}
where each $a_{i}\in K$. Suppose further that $P=\left(  a,b\right)  \in
E\!\left(  K\right)  $ is a torsion point of order $N$. Then the admissible
change of variables $x\longmapsto x-a$ and $y\longmapsto y-b$ results in a
$K$-isomorphic elliptic curve with $P$ translated to the origin. In
particular, we may assume that $a_{6}=0$ in (\ref{ch:SSintwei}) and that
$P=\left(  0,0\right)  $.

\subsection{Point of Order \bm{$N=2$}}

First, suppose $N=2$, so that $P=-P$. By \cite[III.2.3]{MR2514094},
$-P=\left(  0,-a_{3}\right)  $ and so $a_{3}=0$. The admissible change of
variables $x\longmapsto u^{2}x$ and $y\longmapsto u^{3}y+u^{2}sx$ with
$u=\left(  2a_{1}\right)  ^{-1}$ and $s=-\frac{a_{1}}{2}$ results in a
$K$-isomorphic elliptic curve given by the Weierstrass model%
\[
y^{2}=x^{3}+\left(  a_{1}^{4}+4a_{1}^{2}a_{2}\right)  x^{2}+16a_{1}^{4}%
a_{4}x.
\]
As a result, if $E$ is an elliptic curve over $K$ with a torsion point of
order $2$, we may assume that $E$ is given by the Weierstrass model%
\[
E:y^{2}=x^{3}+a_{2}x^{2}+a_{4}x
\]
where each $a_{i}\in K$. In fact, we may assume that $a_{2},a_{4}\in R_{K}$
since the admissible change of variables $x\longmapsto w^{-2}x$ and
$y\longmapsto w^{-3}y$ results in the Weierstrass model $y^{2}=x^{3}%
+a_{2}w^{2}x^{2}+a_{4}w^{4}x$. Note that if $a_{2}^{2}-4a_{4}$ is a square in
$R_{K}$, then $x^{2}+a_{2}x+a_{4}=\left(  x+\alpha\right)  \left(
x+\beta\right)  $ for some $\alpha,\beta\in R_{K}$. In particular, we observe
that $C_{2}\times C_{2}\hookrightarrow E\!\left(  K\right)  $ if and only if
$a_{2}^{2}-4a_{4}$ is a square since $\left(  -\alpha,0\right)  $ is a torsion
point of order $2$.

\begin{lemma}
\label{ch:ss:lemc2}Let $K$ be a number field with class number one. Let $E$ be
an elliptic curve over $K$ with a $K$-rational torsion point $P$ of order $2$.
Suppose further that $C_{2}\times C_{2}\not \hookrightarrow E\!\left(
K\right)  $. Then $E$ is $K$-isomorphic to the elliptic curve%
\[
E_{C_{2}}\!\left(  a,b,d\right)  :y^{2}=x^{3}+2ax^{2}+\left(  a^{2}%
-b^{2}d\right)  x
\]
for some $a,b,d\in R_{K}$ with $d\neq1,b\neq0$ such that $\gcd\!\left(
a,b\right)  $ and $d$ are squarefree.
\end{lemma}

\begin{proof}
By the above discussion, we may assume that $E$ is given by the Weierstrass
model $E:y^{2}=x^{3}+a_{2}x^{2}+a_{4}x$ with $a_{2},a_{4}\in R_{K}$ and
$P=\left(  0,0\right)  $.

Since $C_{2}\times C_{2}\not \hookrightarrow E\!\left(  K\right)  $, we have
that $a_{2}^{2}-4a_{4}$ is not a square in $R_{K}$. Next, we proceed as in~\cite{MR1929230} and factor%
\[
x^{3}+a_{2}x^{2}+a_{4}x=x\left(  x-\theta_{1}\right)  \left(  x-\theta
_{2}\right)
\]
with $\theta_{1}=a+b\sqrt{d}$ and $\theta_{2}=a-b\sqrt{d}$ for some $a,b\in K$
with $b\neq0$ and $d\in R_{K}$ squarefree. Therefore $E:y^{2}=x^{3}%
+2ax^{2}+\left(  a^{2}-b^{2}d\right)  x$. We may further assume that $a,b\in
R_{K}$ since the admissible change of variables $x\longmapsto u^{-2}x$ and
$y\longmapsto u^{-3}y$ results in the $K$-isomorphic elliptic curve%
\[
y^{2}=x^{3}+2au^{2}x^{2}+u^{4}\left(  a^{2}-b^{2}d\right)  x.
\]
Now suppose $\gcd\!\left(  a,b\right)  =g^{2}h$ with $h$ squarefree. Note that this is valid since $K$ has class number one. Then the
admissible change of variables $x\longmapsto g^{2}x$ and $y\longmapsto g^{3}y$
results in the $K$-isomorphic elliptic curve%
\[
y^{2}=x^{3}+\frac{2a}{g^{2}}x^{2}+\frac{\left(  a^{2}-b^{2}d\right)  }{g^{4}%
}x.
\]
In particular, we may assume that $\gcd\!\left(  a,b\right)  $ and $d$ are
squarefree, which completes the proof.
\end{proof}

\begin{lemma}
\label{ch:ss:lemc2c2}Let $K$ be a number field with class number one. Let $E$
be an elliptic curve over $K$ with a $K$-rational torsion point $P$ of order
$2$. Suppose further that $C_{2}\times C_{2}\hookrightarrow E\!\left(
K\right)  $. Then $E$ is $K$-isomorphic to the elliptic curve%
\[
E_{C_{2}\times C_{2}}\!\left(  a,b,d\right)  :y^{2}=x^{3}+\left(
ad+bd\right)  x^{2}+abd^{2}x
\]
for some non-zero $a,b,d\in R_{K}$ such that $\gcd\!\left(  a,b\right)  =1$
and $d$ is squarefree.
\end{lemma}

\begin{proof}
By the discussion before Lemma \ref{ch:ss:lemc2c2}, we may assume that $E$ is given by the Weierstrass
model $E:y^{2}=x^{3}+a_{2}x^{2}+a_{4}x$ with $a_{2},a_{4}\in R_{K}$ and
$P=\left(  0,0\right)  $.

Since $C_{2}\times C_{2}\hookrightarrow E\!\left(  K\right)  $, we proceed as
in \cite{MR1424534} and write%
\[
x^{3}+a_{2}x^{2}+a_{4}x=x\left(  x+A\right)  \left(  x+B\right)
\]
for non-zero $A,B\in R_{K}$. Since $K$ has class number one, we can write $\gcd\!\left(  A,B\right)
=g^{2}d$ with $d$ squarefree. Then the admissible change of variables
$x\longmapsto g^{2}x$ and $y\longmapsto g^{3}y$ results in the $K$-isomorphic
elliptic curve%
\[
y^{2}=x^{3}+\frac{\left(  A+B\right)  }{g^2}x^2+\frac{AB}{g^{4}}x
\]
and so we may assume that $\gcd\!\left(  A,B\right)  =d$. Then $A=ad$ and
$B=bd$ for some $a,b\in R_{K}$ with $\gcd\!\left(  a,b\right)  =1$ which gives
the lemma.
\end{proof}

\begin{remark}
If we omit the condition that $K$ has class number equal to one, then Lemma
\ref{ch:ss:lemc2} (resp. Lemma \ref{ch:ss:lemc2c2}) still holds with the
omission that $\gcd\!\left(  a,b\right)  $ is squarefree (resp. $\gcd\!\left(
a,b\right)  =1$).
\end{remark}

\subsection{Point of Order \bm{$N=3$}}

Now suppose $E$ is an elliptic curve over $K$ with a $K$-rational torsion
point $P$ of order $N\geq3$. By the discussion following (\ref{ch:SSintwei}),
we may assume that $E$ is given by the Weierstrass model%
\begin{equation}
E:y^{2}+a_{1}xy+a_{3}y=x^{3}+a_{2}x^{2}+a_{4}x \label{ch:ssn3wei}%
\end{equation}
with $P=\left(  0,0\right)  $ the point of order $N$.

\begin{lemma}
\label{BazLem1.1}Let $E$ be given by the Weierstrass model (\ref{ch:ssn3wei})
and $P=\left(  0,0\right)  $ a torsion point of order $N$.

$\left(  1\right)  $ If $N\geq3$, then $a_{3}\neq0$ and, after a change of
coordinates, we can suppose $a_{4}=0$.

$\left(  2\right)  $ If $a_{3}\neq0$ and $a_{4}=0$, then $P$ is of order $3$
if and only if $a_{2}=0$.
\end{lemma}

\begin{proof}
See \cite[Lemma 1.1]{MR2684370}.
\end{proof}

\begin{corollary}
\label{sscorC3}Let $E$ be an elliptic curve over a field $K$ with a $K$-rational
torsion point of order~$3$. If the $j$-invariant of $E$ is non-zero, then there is a $t\in K^{\times}$ such that $E$ is $K$-isomorphic to the elliptic curve%
\[
\mathcal{X}_{t}\!\left(  C_{3}\right)  :y^{2}+xy+ty=x^{3}.
\]
\end{corollary}

\begin{proof}
By Lemma \ref{BazLem1.1}, we may assume that $E$ is given by the Weierstrass
model $y^{2}+a_{1}xy+a_{3}y=x^{3}$. Since $c_{4}=a_{1}\left(  a_{1}%
^{3}-24a_{3}\right)  $ and the $j$-invariant of $E$ is $0$ if and only if the
invariant $c_{4}=0$, we may assume that $a_{1}\neq0$ and $a_{1}^{3}%
-24a_{3}\neq0$. The admissible change of variables $x\longmapsto a_{1}^{2}x$
and $y\longmapsto a_{1}^{3}y$ results in the $K$-isomorphic elliptic curve%
\[
y^{2}+xy+\frac{a_{3}}{a_{1}^{3}}y=x^{3}.
\]
Taking $t=\frac{a_{3}}{a_{1}^{3}}$ completes the proof.
\end{proof}

\subsection{Point of Order \bm{$N\geq4$} and Modular Curves}

For an integer $N\geq2$, the modular curve $X_{1}\!\left(  N\right)  $ (with
cusps removed) parameterizes isomorphism classes of pairs $\left(  E,P\right)
$ where $E$ is an elliptic curve and $P$ is a torsion point of order $N$ on
$E$. Two isomorphism classes of pairs $\left(  E,P\right)  $ and $\left(
E^{\prime},P^{\prime}\right)  $ are isomorphic if there exists an isomorphism
$\varphi:E\rightarrow E^{\prime}$ such that $\varphi\!\left(  P\right)
=P^{\prime}$ \cite{MR772569}. If $\varphi$ is a $K$-isomorphism for some field
$K$, we say that the pairs $\left(  E,P\right)  $ and $\left(  E^{\prime
},P^{\prime}\right)  $ are $K$-isomorphic and the $K$-isomorphism class of the
pair $\left(  E,P\right)  $ is a $K$-rational point of $X_{1}\!\left(
N\right)  $. We denote the set of $K$-rational points of $X_{1}\!\left(
N\right)  $ by $X_{1}\!\left(  N\right)  \!\left(  K\right)  $.

\begin{lemma}
[Tate Normal Form]\label{LemmaTNF}Let $K$ be a field and let $E$ be an
elliptic curve over $K$ with a $K$-rational torsion point of order $N\geq4$.
Then every $K$-isomorphism class of pairs $\left(  E,P\right)  $ with $E$ an
elliptic curve over $K$ and $P\in E\!\left(  K\right)  $ a torsion point of
order $N$ contains a unique Weierstrass model%
\begin{equation}
E_{\text{TNF}}:y^{2}+\left(  1-g\right)  xy-fy=x^{3}-fx^{2} \label{bazlem1.4}%
\end{equation}
with $f\in K^{\times}$ and $g\in K$. We call $E_{\text{TNF}}$ the \textit{Tate
Normal Form} of $E$.
\end{lemma}

\begin{proof}
See \cite[Proposition 1.3]{MR2684370}.
\end{proof}

A consequence of Lemma \ref{LemmaTNF} is that the non-cuspidal points of
$X_{1}\!\left(  N\right)  \!\left(  K\right)  $ for $N\geq4$ are in one-to-one
correspondence with the set of $K$-isomorphism classes of pairs $\left(
E_{\text{TNF}},\left(  0,0\right)  \right)  $ where $E_{\text{TNF}}$ is an
elliptic curve over $K$ given by a Tate normal form and $\left(  0,0\right)  $
is its point of order $N$.

{\renewcommand*{\arraystretch}{1.35} \begin{table}[h]
\caption{The Universal Elliptic Curve $\mathcal{X}_{t}\!\left(  T\right)
:y^{2}+\left(  1-g\right)  xy-fy=x^{3}-fx^{2}$}%
\label{ta:universalell}
\begin{center}%
\begin{tabular}
[c]{ccc}\hline
$T$ & $f$ & $g$\\\hline
$C_{4}$ & $t$ & $0$\\\hline
$C_{5}$ & $t$ & $t$\\\hline
$C_{6}$ & $t^{2}+t$ & $t$\\\hline
$C_{7}$ & $t^{3}-t^{2}$ & $t^{2}-t$\\\hline
$C_{8}$ & $2t^{3}-3t+1$ & $\frac{2t^{2}-3t+1}{t}$\\\hline
$C_{9}$ & $t^{5}-2t^{4}+2t^{3}-t^{2}$ & $t^{3}-t^{2}$\\\hline
$C_{10}$ & $\frac{2t^{5}-3t^{4}+t^{3}}{\left(  t^{2}-3t+1\right)  ^{2}}$ &
$\frac{-2t^{3}+3t^{2}-t}{t^{2}-3t+1}$\\\hline
$C_{12}$ & $\frac{12t^{6}-30t^{5}+34t^{4}-21t^{3}+7t^{2}-t}{\left(
t-1\right)  ^{4}}$ & $\frac{-6t^{4}+9t^{3}-5t^{2}+t}{\left(  t-1\right)  ^{3}%
}$\\\hline
$C_{2}\times C_{4}$ & $4t^{2}+t$ & $0$\\\hline
$C_{2}\times C_{6}$ & $\frac{-2t^{3}+14t^{2}-22t+10}{\left(  t+3\right)
^{2}\left(  t-3\right)  ^{2}}$ & $\frac{-2t+10}{\left(  t+3\right)  \left(
t-3\right)  }$\\\hline
$C_{2}\times C_{8}$ & $\frac{16t^{3}+16t^{2}+6t+1}{\left(  8t^{2}-1\right)
^{2}}$ & $\frac{16t^{3}+16t^{2}+6t+1}{2t\left(  4t+1\right)  \left(
8t^{2}-1\right)  }$\\\hline
\end{tabular}
\end{center}
\end{table}}

For an integer $M\geq1$, the modular curve $X_{1}\!\left(  2,2M\right)  $
parameterizes isomorphism classes of triples $\left(  E,P,Q\right)  $ where
$E$ is an elliptic curve and $\left\langle P,Q\right\rangle \cong C_{2}\times
C_{2M}$ with $e_2\!\left(  P,M\cdot Q\right)  =-1$ where~$e_{2}$ is the Weil
pairing \cite[III.8]{MR2514094}. It is well known that the modular curves
$X_{1}\!\left(  N\right)  $ and $X_{1}\!\left(  2,2M\right)  $ have genus $0$
if and only if $N=1,2,\ldots,10,12$ or $M=1,2,3,4$ \cite[Proposition~3.7]{MR589086}. 
These modular curves are fine moduli spaces for $N=5,\ldots,10,12$ or $M=2,3,4$. Consequently, they are parameterizable by a single parameter $t$ \cite[Table 3]{MR0434947}. The same holds for $X_1(4)$ upon restricting to its non-cuspidal points.
More precisely, for these values of $N$ and $M$, we consider the abelian
groups $T=C_{N}$ and $T=C_{2}\times C_{2M}$. For $t\in\mathbb{P}^{1}$, define
$\mathcal{X}_{t}$ as the mapping which takes $T$ to the elliptic curve
$\mathcal{X}_{t}\!\left(  T\right)  $ where the Weierstrass model of
$\mathcal{X}_{t}\!\left(  T\right)  $ is given in\ Table \ref{ta:universalell}%
\footnote{Our parameterizations differ slightly from \cite[Table 3]%
{MR0434947}. We instead use \cite[Table 3]{MR1748483} which expands the
implicit expressions for the parameters $b$ and $c$ in \cite[Table
3]{MR0434947} to express the universal elliptic curves in terms of a single
parameter $t$.}. Then $\mathcal{X}_{t}\!\left(  T\right)  $ is a one-parameter
family of elliptic curves with the property that if $t\in K$ for some field
$K$, then $\mathcal{X}_{t}\!\left(  T\right)  $ is an elliptic curve over $K$
and $T\hookrightarrow\mathcal{X}_{t}\!\left(  T\right)  \!\left(  K\right)
_{\text{tors}}$. For $T\neq C_4$, the Weierstrass model of $\mathcal{X}_{t}\!\left(  T\right)
$ is known as the \textit{universal elliptic curve%
\index{Elliptic Curve!Universal Elliptic Curve}%
} over $X_{1}\!\left(  N\right)  $ (resp. $X_{1}\!\left(  2,2M\right)  $) if
$T=C_{N}$ (resp. $T=C_{2}\times C_{2M}$). We note that when $T=C_4$, $\mathcal{X}_{t}\!\left(  T\right)$ is the universal elliptic curve over the non-cuspidal points of $X_1(4)$.  The following Proposition summarizes
this discussion for the eleven groups $T$ considered above. In particular, it is a consequence of Lemma \ref{LemmaTNF} and \cite[Table~3]{MR0434947}.

\begin{proposition}
\label{LemmaunivincT}\label{PropSSuniv}Let $K$ be a field. If $t\in K$ such
that $\mathcal{X}_{t}\!\left(  T\right)  $ is an elliptic curve, then
$T\hookrightarrow\mathcal{X}_{t}\!\left(  T\right)  \!\left(  K\right)  $.
Moreover, if $E$ is an elliptic curve over $K$ with $T\hookrightarrow
E\!\left(  K\right)  $, then there is a $t\in K$ such that $E$ is
$K$-isomorphic to $\mathcal{X}_{t}\!\left(  T\right)  $.
\end{proposition}

\subsection[The Elliptic Curves $E_{T}\!\left(  a,b\right)  $ and
$E_{T}\!\left(  a,b,d\right)  $]{The Elliptic Curves
\bm{$E_{T}\!\left(  a,b\right)  $} and \bm{$E_{T}\!\left(  a,b,d\right)  $}}

Mazur's Torsion Theorem \cite{MR488287} states that there are fifteen
possibilities for the torsion subgroup of a rational elliptic curve.

\begin{theorem}
[Mazur's Torsion Theorem]\label{MazurTorThm}Let $E$ be a rational elliptic
curve. Then%
\[
E\!\left(
%TCIMACRO{\U{211a} }%
%BeginExpansion
\mathbb{Q}
%EndExpansion
\right)  _{\text{tors}}\cong\left\{
\begin{array}
[c]{ll}%
C_{N} & \text{for }N=1,2,\ldots,10,12,\\
C_{2}\times C_{2N} & \text{for }N=1,2,3,4.
\end{array}
\right.
\]

\end{theorem}

For the fourteen non-trivial torsion subgroups $T$ allowed by Theorem \ref{MazurTorThm}, let $E_T$ be as defined in Table~\ref{ta:ETmodel}. We note that the table makes reference to an additional family, namely $E_{C_3^{0}}$, which corresponds to elliptic curves with a $3$-torsion point and $j$-invariant equal to $0$ (see Proposition~\ref{rationalmodels}). Now observe that when $T=C_{2},C_{2}\times C_{2}$, $E_{T}=E_{T}\!\left(
a,b,d\right)  $ is the three parameter family of elliptic curves which was the
subject of Lemmas \ref{ch:ss:lemc2} and \ref{ch:ss:lemc2c2}. For $T=C_{3}$,
let $\mathcal{X}_{t}\!\left(  T\right)  $ denote the elliptic curve defined in
Corollary \ref{sscorC3} and for the remaining $T$'s, let $\mathcal{X}%
_{t}\!\left(  T\right)  $ be as defined in Table \ref{ta:universalell}. For
$T\neq C_{2},C_{2}\times C_{2}$, we show that $E_{T}=E_{T}\!\left(
a,b\right)  $ is $K$-isomorphic to $\mathcal{X}_{b/a}\!\left(  T\right)  $.
For the following lemma, let $\alpha_{T},\beta_{T},$ and $\gamma_{T}$ be as
defined in Tables \ref{ta:alpT}, \ref{ta:betT}, and \ref{ta:gamT}, respectively.

\begin{lemma}
\label{ss:Lem1}Let $K$ be a number field or a local field and let $R_{K}$ denote its ring of
integers. Then for $T\neq C_{2},C_{2}\times C_{2}$, the elliptic curves
$\mathcal{X}_{b/a}\!\left(  T\right)  $ and $E_{T}=E_{T}\!\left(  a,b\right)
$ are $K$-isomorphic for $a,b\in R_{K}$. Moreover, $E_{T}$ is given by an
integral Weierstrass model and its discriminant is $\gamma_{T}$ and the
invariants $c_{4}$ and $c_{6}$ of $E_{T}$ are $\alpha_{T}$ and $\beta_{T}$, respectively.
\end{lemma}

\begin{proof}
Let%
\[
w_{T}=\left\{
\begin{array}
[c]{ll}%
a & \text{if }T=C_{3},C_{4},C_{5},C_{6},C_{2}\times C_{4},\\
a^{2} & \text{if }T=C_{7},\\
ab & \text{if }T=C_{8},\\
a^{3} & \text{if }T=C_{9},\\
a\left(  a^{2}-3ab+b^{2}\right)  & \text{if }T=C_{10},\\
a\left(  -a+b\right)  ^{3} & \text{if }T=C_{12},\\
-9a^{2}+b^{2} & \text{if }T=C_{2}\times C_{6},\\
2b\left(  a+4b\right)  \left(  -a^{2}+8b^{2}\right)  & \text{if }T=C_{2}\times
C_{8}.
\end{array}
\right.
\]
Then the admissible change of variables $x\longmapsto w_{T}^{-2}x$ and
$y\longmapsto w_{T}^{-3}y$ gives a $K$-isomorphism from $\mathcal{X}%
_{b/a}\!\left(  T\right)  $ onto $E_{T}=E_{T}\!\left(  a,b\right)  $. It is
now verified via the formulas in $\left(  \ref{basicformulas}\right)  $ that
the discriminant of $E_{T}$ is $\gamma_{T}$ and that the invariants $c_{4}$
and $c_{6}$ of $E_{T}$ are $\alpha_{T}$ and $\beta_{T}$, respectively.
SageMath~\cite{sagemath} worksheets containing these verifications are also found in \cite{GitHubMinimalModels}.
\end{proof}

\section{Explicit Frey and Flexor-Oesterl\'{e}\label{SecFreyFlex}}

Let $E_{T}$ be the elliptic curve defined in Table \ref{ta:ETmodel}. In the
previous section, we saw that these families of elliptic curves parameterize
all elliptic curves over $K$ with $T\hookrightarrow E\!\left(  K\right)  $
where $T\neq C_{2},C_{3},C_{2}\times C_{2}$ is one of the remaining eleven
non-trivial torsion subgroups allowed by Theorem~\ref{MazurTorThm}. This also
holds for $T=C_{3}$ under the assumption that $E$ has non-zero $j$-invariant.
We also saw that if $K$ has class number one, then for $T=C_{2},C_{2}\times
C_{2}$, the family $E_{T}=E_{T}\!\left(  a,b,d\right)  $ parameterizes all
elliptic curves over $K$ with $T\hookrightarrow E\!\left(  K\right)  $. Note
that for $T=C_{2}$, we must assume that $C_{2}\times C_{2}%
\not \hookrightarrow E\!\left(  K\right)  $. Moreover, the discriminant of
$E_{T}$ is given by $\gamma_{T}$ and the invariants $c_{4}$ and $c_{6}$ are
$\alpha_{T}$ and $\beta_{T}$, respectively. In the following lemma, we
consider $\alpha_{T},\beta_{T}$, and $\gamma_{T}$ as polynomials in $%
%TCIMACRO{\U{2124} }%
%BeginExpansion
\mathbb{Z}
%EndExpansion
\left[  a,b,d,r,s\right]  $. In particular, we view $a,b,d,r,s$ as indeterminates.

\begin{lemma}
\label{polynomials}Let $\alpha_{T},\beta_{T},\gamma_{T}$ be as given in Tables
\ref{ta:alpT}, \ref{ta:betT}, and \ref{ta:gamT}, respectively. Then for each
$T$, the identity $\alpha_{T}^{3}-\beta_{T}^{2}=1728\gamma_{T}$ holds in $%
%TCIMACRO{\U{2124} }%
%BeginExpansion
\mathbb{Z}
%EndExpansion
\left[  a,b,d,r,s\right]  $ and there exist $\mu_{T}^{\left(  j\right)
},\nu_{T}^{\left(  j\right)  }\in%
%TCIMACRO{\U{2124} }%
%BeginExpansion
\mathbb{Z}
%EndExpansion
\left[  a,b,d,r,s\right]  $ for $j=1,2,3$ such that the following identities
hold in $%
%TCIMACRO{\U{2124} }%
%BeginExpansion
\mathbb{Z}
%EndExpansion
\left[  a,b,d,r,s\right]  $:
\end{lemma}
{\renewcommand*{\arraystretch}{1.15} \begin{longtable}{cccc}
\hline
$T$ & $\mu_{T}^{\left(  1\right)  }\alpha_{T}+\nu_{T}^{\left(  1\right)
}\beta_{T}$ & $\mu_{T}^{\left(  2\right)  }\alpha_{T}+\nu_{T}^{\left(
2\right)  }\gamma_{T}$ & $\mu_{T}^{\left(  3\right)  }\beta_{T}+\nu
_{T}^{\left(  3\right)  }\gamma_{T}$\\
\hline
\endfirsthead
%\caption[]{\emph{continued}}\\
\hline
$T$ & $\mu_{T}^{\left(  1\right)  }\alpha_{T}+\nu_{T}^{\left(  1\right)
}\beta_{T}$ & $\mu_{T}^{\left(  2\right)  }\alpha_{T}+\nu_{T}^{\left(
2\right)  }\gamma_{T}$ & $\mu_{T}^{\left(  3\right)  }\beta_{T}+\nu
_{T}^{\left(  3\right)  }\gamma_{T}$\\
\hline
\endhead
\hline
\multicolumn{2}{r}{\emph{continued on next page}}
\endfoot
\hline
\endlastfoot
	
$C_{2}$ & \multicolumn{1}{l}{$2^{8}3^{2}\left(  rb^{4}d^{2}+sa^{3}\right)  $}
& \multicolumn{1}{l}{$2^{10}\left(  rb^{6}d^{3}+sa^{6}\right)  $} &
\multicolumn{1}{l}{$2^{12}\left(  rb^{8}d^{4}+sa^{7}\right)  $}\\\hline
$C_{3}$ & \multicolumn{1}{l}{$2^{6}3^{3}a^{3}\left(  ra^{3}+sb^{3}\right)  $}
& \multicolumn{1}{l}{$2^{15}3^{6}a^{3}\left(  ra^{9}+sb^{9}\right)  $} &
\multicolumn{1}{l}{$2^{6}3^{9}a^{4}\left(  ra^{9}+sb^{9}\right)  $}\\\hline
$C_{4}$ & \multicolumn{1}{l}{$2^{8}3^{2}a^{2}\left(  ra^{5}+sb^{5}\right)  $}
& \multicolumn{1}{l}{$2^{12}a^{2}\left(  ra^{12}+sb^{11}\right)  $} &
\multicolumn{1}{l}{$2^{18}a^{3}\left(  ra^{11}+sb^{11}\right)  $}\\\hline
$C_{5}$ & \multicolumn{1}{l}{$2^{4}3^{2}5\left(  ra^{9}+sb^{9}\right)  $} &
\multicolumn{1}{l}{$5\left(  ra^{15}+sb^{15}\right)  $} &
\multicolumn{1}{l}{$5^{3}\left(  ra^{17}+sb^{17}\right)  $}\\\hline
$C_{6}$ & \multicolumn{1}{l}{$2^{7}3^{4}\left(  ra^{9}+sb^{9}\right)  $} &
\multicolumn{1}{l}{$2^{4}3^{4}\left(  ra^{15}+sb^{5}\right)  $} &
\multicolumn{1}{l}{$2^{9}3^{3}\left(  ra^{17}+sb^{17}\right)  $}\\\hline
$C_{7}$ & \multicolumn{1}{l}{$2^{4}3^{2}7\left(  ra^{19}+sb^{19}\right)  $} &
\multicolumn{1}{l}{$7^{2}\left(  ra^{31}+sb^{31}\right)  $} &
\multicolumn{1}{l}{$7\left(  ra^{35}+sb^{35}\right)  $}\\\hline
$C_{8}$ & \multicolumn{1}{l}{$2^{7}3^{2}\left(  ra^{18}+sb^{19}\right)  $} &
\multicolumn{1}{l}{$2^{4}\left(  ra^{30}+sb^{31}\right)  $} &
\multicolumn{1}{l}{$2^{9}\left(  ra^{34}+sb^{35}\right)  $}\\\hline
$C_{9}$ & \multicolumn{1}{l}{$2^{4}3^{4}\left(  ra^{29}+sb^{29}\right)  $} &
\multicolumn{1}{l}{$3^{4}\left(  ra^{47}+sb^{47}\right)  $} &
\multicolumn{1}{l}{$3^{3}\left(  ra^{53}+sb^{53}\right)  $}\\\hline
$C_{10}$ & \multicolumn{1}{l}{$2^{7}3^{2}5\left(  ra^{29}+sb^{29}\right)  $} &
\multicolumn{1}{l}{$2^{4}5\left(  ra^{47}+sb^{47}\right)  $} &
\multicolumn{1}{l}{$2^{9}5^3\left(  ra^{53}+sb^{47}\right)  $}\\\hline
$C_{12}$ & \multicolumn{1}{l}{$2^{7}3^{4}\left(  ra^{38}+sb^{39}\right)  $} &
\multicolumn{1}{l}{$2^{4}3^{4}\left(  ra^{62}+sb^{63}\right)  $} &
\multicolumn{1}{l}{$2^{9}3^{3}\left(  ra^{70}+sb^{71}\right)  $}\\\hline
$C_{2}\times C_{2}$ & \multicolumn{1}{l}{$2^{5}3^{2}d^{4}\left(  ra^{4}%
+sb^{4}\right)  $} & \multicolumn{1}{l}{$2^{4}d^{6}\left(  ra^{6}%
+sb^{6}\right)  $} & \multicolumn{1}{l}{$2^{7}d^{8}\left(  ra^{8}%
+sb^{8}\right)  $}\\\hline
$C_{2}\times C_{4}$ & \multicolumn{1}{l}{$2^{14}3^{2}\left(  ra^{8}%
+sb^{8}\right)  $} & \multicolumn{1}{l}{$2^{16}\left(  ra^{14}+sb^{12}\right)
$} & \multicolumn{1}{l}{$2^{24}\!\left(  ra^{16}+sb^{16}\right)  $}\\\hline
$C_{2}\times C_{6}$ & \multicolumn{1}{l}{$2^{31}3^{4}\left(  ra^{18}%
+sb^{19}\right)  $} & \multicolumn{1}{l}{$2^{45}3^{4}\left(  ra^{30}%
+sb^{31}\right)  $} & \multicolumn{1}{l}{$2^{56}3^{3}\left(  ra^{34}%
+sb^{35}\right)  $}\\\hline
$C_{2}\times C_{8}$ & \multicolumn{1}{l}{$2^{49}3^{2}\left(  ra^{38}%
+sb^{38}\right)  $} & \multicolumn{1}{l}{$2^{74}\left(  ra^{62}+sb^{62}%
\right)  $} & \multicolumn{1}{l}{$2^{90}\left(  ra^{70}+sb^{70}\right)  $%
}
\end{longtable}}\addtocounter{table}{-1}
\textit{In particular, if $K$ is a number field or a local field with ring of integers $R_{K}$ and $a,b\in R_K$ are coprime elements, then the ideals
$\left(  \alpha_{T}\!\left(  a,b\right)  \right)  +\left(  \gamma_{T}\!\left(
a,b\right)  \right)  $ and $  \left(\beta_{T}\!\left(  a,b\right)\right)  +\left(\gamma
_{T}\!\left(  a,b\right)  \right)   $ are contained in the principal ideal
$\delta_{T}R_{K}$ where $\delta_{T}$ is as given below:}
\[%
\begin{tabular}
[c]{cccccccccccccc}\hline
$T$ & $C_{3}$ & $C_{4}$ & $C_{5}$ & $C_{6}$ & $C_{7}$ & $C_{8}$ & $C_{9}$ &
$C_{10}$ & $C_{12}$& $C_{2}\times C_{2}$ & $C_{2}\times C_{4}$ & $C_{2}\times C_{6}$ & $C_{2}\times
C_{8}$\\\hline
$\delta_{T}$ & $6a$ & $2a$ & $5$ & $6$ & $7$ & $2$ & $3$ & $10$ & $6$ & $2d$ & $2$ &
$6$ & $2$\\\hline
\end{tabular}
\]

\begin{proof}
Let $a,b,d,r,s$ be indeterminates. The polynomials $\mu_{T}^{\left(  j\right)  },\nu_{T}^{\left(  j\right)  }$
for $j=1,2,3$ are constructed by applying the Euclidean Algorithm to pairs of
$\alpha_{T},\beta_{T},\ $and $\gamma_{T}$ over the rings $%
%TCIMACRO{\U{211a} }%
%BeginExpansion
\mathbb{Q}
%EndExpansion
\!\left[  a\right]  \left(  b\right)  $ and $%
%TCIMACRO{\U{211a} }%
%BeginExpansion
\mathbb{Q}
%EndExpansion
\!\left[  b\right]  \left(  a\right)  $. Explicit equations for these
polynomials are found in \cite[Appendix D]{MR3828473} or \cite{GitHubMinimalModels}. The latter has the verification of these identities. We note that the identities were verified using SageMath~\cite{sagemath}.

For the second statement, note that the ideals generated by $a^{n}$ and
$b^{m}$ are coprime for any positive integers $n$ and $m$ since the ideals
generated by $a$ and $b$ are coprime. In particular, there exist $r,s\in
R_{K}$ such that $ra^{n}+sb^{m}=1$. The second statement now follows.
\end{proof}

\begin{lemma}
\label{silverconverse}\label{intweier}\label{intweierpara}Let $K$ be a number
field and let $\mathfrak{p}$ be a prime of $K$. If $E$ is an elliptic curve
over $K_{\mathfrak{p}}$ given by an integral Weierstrass model, then any
admissible change of variables $x\longmapsto u^{2}x+r$ and $y\longmapsto
u^{3}y+u^{2}sx+w$ used to produce a minimal model at $\mathfrak{p}$ satisfies $u,r,s,w\in
R_{K_{\mathfrak{p}}}$.

In particular, if $E$ is a rational elliptic curve given by an integral
Weierstrass model with invariants $c_{4}$ and $c_{6}$ and discriminant
$\Delta$, then there is a unique positive integer $u$ such that
\[
c_{4}^{\prime}=u^{-4}c_{4},\qquad c_{6}^{\prime}=u^{-6}c_{6},\qquad
\text{and}\qquad\Delta_{E/\mathbb{Q}}^{\text{min}}=u^{-12}\Delta
\]
where $\Delta_{E/\mathbb{Q}}^{\text{min}}$ is the minimal discriminant of $E$ and
$c_{4}^{\prime}$ and $c_{6}^{\prime}$ are the invariants associated to a
global minimal model of $E$.
\end{lemma}

\begin{proof}
See \cite[Proposition VII.1.3 and Corollary VIII.8.3]{MR2514094}.
\end{proof}

\begin{theorem}
\label{Thm1}Let $E$ be an elliptic curve over a number field $K$ with ring of
integers $R_{K}$. Suppose further that $T\hookrightarrow E\!\left(  K\right)
$ where $T$ is one of the torsion subgroups appearing in Theorem~\ref{MazurTorThm} with $\left\vert T\right\vert >4$. If $E$ has additive
reduction at a prime $\mathfrak{p}$ of $K$, then the residue characteristic of
$\mathfrak{p}$ is a rational prime in the set $S$ given below:
\[%
\begin{tabular}
[c]{ccccccccccc}\hline
$T$ & $C_{5}$ &$C_{6}$& $C_{7}$ & $C_{8}$ & $C_{9}$ & $C_{10}$ & $C_{12}$ &
$C_{2}\times C_{4}$ & $C_{2}\times C_{6}$ & $C_{2}\times C_{8}$\\\hline
$S$ & $\left\{  5\right\}$& $\left\{  2,3\right\}  $ & $\left\{  7\right\}  $ & $\left\{  2\right\}  $
& $\left\{  3\right\}  $ & $\left\{  5\right\}  $ & $\left\{  2,3\right\}  $ &
$\left\{  2\right\}  $ & $\left\{  2,3\right\}  $ & $\left\{  2\right\}
$\\\hline
\end{tabular}
\]
% \[
% S=\left\{
% \begin{array}
% [c]{cl}
% \left\{  2\right\}  & \text{if }T=C_{8},C_{2}\times C_{4},C_{2}\times C_{8},\\
% \left\{  3\right\}  & \text{if }T=C_{9},\\
% \left\{  5\right\}  & \text{if }T=C_{5},C_{10},\\
% \left\{  7\right\}  & \text{if }T=C_{7},\\
% \left\{  2,3\right\}  & \text{if }T=C_{6},C_{12},C_{2}\times C_{6},
% \end{array}
% \right.
% \]
\end{theorem}

\begin{proof}
Let $E$ be an elliptic curve with $T\hookrightarrow E\!\left(  K\right)  $ and
assume that $E$ has additive reduction at a prime $\mathfrak{p}$ of $K$. Let
$R_{K_{\mathfrak{p}}}$ denote the ring of integers of $K_{\mathfrak{p}}$. By
Proposition~\ref{LemmaunivincT}, there is a $t\in K$ such that $E$ is
$K$-isomorphic to $\mathcal{X}_{t}\!\left(  T\right)  $. In particular, $t\in
K_{\mathfrak{p}}$ and there are $a,b\in R_{K_{\mathfrak{p}}}$ such that
$t=\frac{b}{a}$ and coprime $a,b$. By Lemma~\ref{ss:Lem1},
$E$ is $K_{\mathfrak{p}}$-isomorphic to $E_{T}=E_{T}\!\left(  a,b\right)  $.
Next, let $x\longmapsto u_{T}^{2}x+r_{T}$ and $y\longmapsto u_{T}^{3}%
y+u_{T}^{2}s_{T}x+w_{T}$ be an admissible change of variables resulting in a
minimal model $E_{T}^{\prime}$ at $\mathfrak{p}$ for $E_{T}$. Since $E_{T}$ is given by an
integral Weierstrass model, we have by Lemma \ref{silverconverse} that
$u_{T},r_{T},s_{T},w_{T}\in R_{K_{\mathfrak{p}}}$. In particular, the minimal
discriminant of $E_{T}$ at $\mathfrak{p}$ is $\Delta_{E/K_{\mathfrak{p}}}^{\text{min}}=u_{T}%
^{-12}\gamma_{T}$. Moreover, the invariants $c_{4}$ and $c_{6}$ associated to
$E_{T}^{\prime}$ satisfy $c_{4}=u_{T}^{-4}\alpha_{T}$ and $c_{6}=u_{T}%
^{-6}\beta_{T}$. By Lemma \ref{polynomials}, $c_{4}R_{K_{\mathfrak{p}}}+\Delta_{E/K_{\mathfrak{p}}}^{\text{min}}R_{K_{\mathfrak{p}}}%
\subseteq\alpha_{T}R_{K_{\mathfrak{p}}}+\gamma_{T}R_{K_{\mathfrak{p}}%
}\subseteq\delta_{T}R_{K_{\mathfrak{p}}}$
% \[
% c_{4}R_{K_{\mathfrak{p}}}+\Delta_{E/K_{\mathfrak{p}}}^{\text{min}}R_{K_{\mathfrak{p}}}%
% \subseteq\alpha_{T}R_{K_{\mathfrak{p}}}+\gamma_{T}R_{K_{\mathfrak{p}}%
% }\subseteq\delta_{T}R_{K_{\mathfrak{p}}}%
% \]
where $\delta_{T}$ is as given in Lemma \ref{polynomials}. 
Since $E_T$ has additive reduction at $\mathfrak{p}$, it is the case that
$v_{\mathfrak{p}}(  \Delta_{E/K_{\mathfrak{p}}}^{\text{min}})  ,v_{\mathfrak{p}%
}\!\left(  c_{4}\right)  >0$. Hence $v_{\mathfrak{p}}\!\left(  \delta
_{T}\right)  >0$ and thus the residue characteristic of $\mathfrak{p}$ divides
$\delta_{T}$. This shows all cases claimed except for $T=C_{10}$. Indeed, by
Lemma \ref{polynomials}, $E_{C_{10}}$ may have additive reduction at $2$. Now
observe that if $C_{10}\hookrightarrow E\!\left(  K_{\mathfrak{p}}\right)  $,
then $C_{5}\hookrightarrow E\!\left(  K_{\mathfrak{p}}\right)  $ and thus
$E_{C_{10}}\!\left(  a,b\right)  $ is $K_{\mathfrak{p}}$-isomorphic to
$E_{C_{5}}\!\left(  a^{\prime},b^{\prime}\right)  $ for two coprime elements
$a^{\prime},b^{\prime}\in R_{K_{\mathfrak{p}}}$. In particular, if $E_{C_{10}%
}$ has additive reduction at a prime $\mathfrak{p}$ of $R_{K}$, then the
residue characteristic of $\mathfrak{p}$ is $5$.
\end{proof}

As a consequence, we get Frey's Theorem \cite{MR0457444} in the case when
$\ell=5,7$. Next, we prove Flexor and Oesterl\'{e}'s Theorem \cite{MR1065153}
by assuming Frey's Theorem.

\begin{theorem}\label{FOes}
Let $E$ be an elliptic curve over a number field $K$. If $E\!\left(  K\right)
$ contains a point of order $N$ and $E$ has additive reduction at a prime
$\mathfrak{p}$ of $K$ whose residue characteristic does not divide $N$, then
$N\leq4$. Moreover, if $E$ has additive reduction at at least two primes of
$K$ with different residue characteristics, then $N$ divides $12$.
\end{theorem}

\begin{proof}
Let $E$ be an elliptic curve over $K$ with a $K$-rational torsion point of
order $N$. First, suppose $E$ has additive reduction at a prime $\mathfrak{p}$
of $K$. If $\ell>3$ is a rational prime which divides~$N$, then by Frey's
Theorem the residue characteristic of $\mathfrak{p}$ must divide~$N$. If$\ 8$
(resp. $9$) divides $N$, then the residue characteristic of $\mathfrak{p}$ is
$2$ (resp. $3$) by Theorem \ref{Thm1}.

Next, suppose $E$ has additive reduction at at least two primes of $K$ with
different residue characteristics. By Frey's Theorem, the only primes dividing
$N$ are $2$ and $3$. By Theorem \ref{Thm1}, neither $8$ nor $9$ divide $N$ and so
$N=1,2,3,4,6$, or $12$. In particular, $N$ divides $12$. Lastly, if $E$ has additive
reduction at a prime $\mathfrak{p}$ of $K$ whose residue characteristic does
not divide $N$, then by Theorem~\ref{Thm1} we conclude that $N\leq4$.
\end{proof}

\section{Classification of Minimal Discriminants\label{sec:class}}

In this section, we restrict our attention to rational elliptic curves. In
Proposition \ref{ssremarkapositive}, we show that the study of rational
elliptic curves with a non-trivial torsion point is equivalent to studying the
parameterized families $E_{T}$ where $E_{T}$ is as given in Table
\ref{ta:ETmodel}. Consequently, an explicit classification of the minimal
discriminant of rational elliptic curves with a non-trivial torsion subgroup
is reduced to finding necessary and sufficient conditions on the parameters of~$E_{T}$ to determine the minimal discriminant. This is the context of our main
theorem, Theorem~\ref{ch:ss:mainthm}.

\begin{lemma}
\label{ch:sslemaeven}Let $T=C_{2}\times C_{2}$ and suppose $E$ is a rational
elliptic curve with $T\hookrightarrow E\!\left(
%TCIMACRO{\U{211a} }%
%BeginExpansion
\mathbb{Q}
%EndExpansion
\right)  $. Then there are integers $a,b,d$ with $a$ and $b$ relatively prime
integers, $a$ even, and $d$ a positive squarefree integer such that $E$ is $%
%TCIMACRO{\U{211a} }%
%BeginExpansion
\mathbb{Q}
%EndExpansion
$-isomorphic to $E_{T}=E_{C_{2}\times C_{2}}\!\left(  a,b,d\right)  $.
\end{lemma}

\begin{proof}
By Lemma \ref{ch:ss:lemc2c2}, $E$ is $%
%TCIMACRO{\U{211a} }%
%BeginExpansion
\mathbb{Q}
%EndExpansion
$-isomorphic to $E_{T}:y^{2}=x^{3}+\left(  ad+bd\right)  x^{2}+abd^{2}$ where
$a,b,d$ are integers such that $a$ and $b$ are relatively prime and $d$ is a
squarefree integer. By the proof of Lemma \ref{ch:ss:lemc2c2}, $d$ may be
assumed to be positive. It remains to show that $a$ may be assumed to be even.
Observe that if $b$ were even, then we can interchange $a$ and $b$. So suppose
$a$ and $b$ are odd. Then $c=b-a$ is even, and the admissible change of
variables $x\longmapsto x-ad$ gives a $%
%TCIMACRO{\U{211a} }%
%BeginExpansion
\mathbb{Q}
%EndExpansion
$-isomorphism from $E_{T}$ onto the elliptic curve given by the Weierstrass
model%
\[
E_{T}\!\left(  c,-a,d\right)  :y^{2}=x^{3}+\left(  cd-ad\right)  x^{2}%
-acd^{2}x.
\]
Thus, after relabeling, we may assume that $a$ is even.
\end{proof}

\begin{lemma}
\label{Lemposnega}For $T\neq C_{2},C_{2}\times C_{2}$, we have that
$E_{T}\!\left(  -a,b\right)  $ is $%
%TCIMACRO{\U{211a} }%
%BeginExpansion
\mathbb{Q}
%EndExpansion
$-isomorphic to $E_{T}\!\left(  a,-b\right)  $.
\end{lemma}

\begin{proof}
Let $E$ and $E^{\prime}$ be rational elliptic curves. Suppose further that the
invariants~$c_{4}$ and~$c_{6}$ of their Weierstrass models coincide. Then $E$
and $E^{\prime}$ are $%
%TCIMACRO{\U{211a} }%
%BeginExpansion
\mathbb{Q}
%EndExpansion
$-isomorphic since they are both $%
%TCIMACRO{\U{211a} }%
%BeginExpansion
\mathbb{Q}
%EndExpansion
$-isomorphic to the elliptic curve $y^{2}=x^{3}-27c_{4}x-54c_{6}$. In
particular, the invariants~$c_{4}$ and~$c_{6}$ of a Weierstrass model
determine an elliptic curve up to $%
%TCIMACRO{\U{211a} }%
%BeginExpansion
\mathbb{Q}
%EndExpansion
$-isomorphism. The result now follows since $\alpha_{T}\!\left(  a,b\right)  $
and $\beta_{T}\!\left(  a,b\right)  $ are the invariants $c_{4}$ and $c_{6}$,
respectively, of the Weierstrass model of $E_{T}\!\left(  a,b\right)  $ and
the following equality holds for each $T$: $\alpha_{T}\!\left(  -a,b\right)
=\alpha_{T}\!\left(  a,-b\right)  $ and $\beta_{T}\!\left(  -a,b\right)
=\beta_{T}\!\left(  a,-b\right)  $.
\end{proof}

\begin{proposition}
\label{rationalmodels}\label{ssremarkapositive}\label{sscorc3j0}Let $E$ be a
rational elliptic curve and suppose further that $T\hookrightarrow E\!\left(
%TCIMACRO{\U{211a} }%
%BeginExpansion
\mathbb{Q}
%EndExpansion
\right)  _{\text{tors}}$ where $T$ is one of the fourteen non-trivial torsion
subgroups allowed by Theorem~\ref{MazurTorThm}. Then there are integers
$a,b,d$ such that

$\left(  1\right)  $ If $T\neq C_{2},C_{3},C_{2}\times C_{2}$, then $E$ is
$\mathbb{Q}$-isomorphic to $E_{T}\!\left(  a,b\right)  $ with $\gcd\!\left(
a,b\right)  =1$ and $a$ is positive.

$\left(  2\right)  $ If $T=C_{2}$ and $C_{2}\times C_{2}\not \hookrightarrow
E(\mathbb{Q})$, then $E$ is $\mathbb{Q}$-isomorphic to $E_{T}\!\left(
a,b,d\right)  $ with $d\neq1,b\neq0$ such that $d$ and $\gcd\!\left(
a,b\right)  $ are positive squarefree integers.

$\left(  3\right)  $ If $T=C_{3}$ and the $j$-invariant of $E$ is not $0$,
then $E$ is $\mathbb{Q}$-isomorphic to $E_{T}\!\left(  a,b\right)  $ with
$\gcd\!\left(  a,b\right)  =1$ and $a$ is positive.

$\left(  4\right)  $ If $T=C_{3}$ and the $j$-invariant of $E$ is $0$, then
$E$ is either $\mathbb{Q}$-isomorphic to $E_{T}\!\left(  24,1\right)  $ or to
the curve $E_{C_{3}^{0}}\!\left(  a\right)  :y^{2}+ay=x^{3}$ for some positive
cubefree integer $a$.

$\left(  5\right)  $ If $T=C_{2}\times C_{2}$, then $E$ is $\mathbb{Q}%
$-isomorphic to $E_{T}\!\left(  a,b,d\right)  $ with $\gcd\!\left(
a,b\right)  =1$, $d$ positive squarefree, and $a$ is even.
\end{proposition}

\begin{proof}
Observe that $\left(  1\right)  $ is a direct consequence of Proposition
\ref{PropSSuniv}, Lemma \ref{ss:Lem1}, and Lemma \ref{Lemposnega}. Statements
$\left(  2\right)  $ and $\left(  5\right)  $ are restatements of Lemmas
\ref{ch:ss:lemc2} and \ref{ch:sslemaeven}, respectively.

Now suppose $T=C_{3}$ and $T\hookrightarrow E\!\left(
%TCIMACRO{\U{211a} }%
%BeginExpansion
\mathbb{Q}
%EndExpansion
\right)  $. By Lemma \ref{BazLem1.1}, we may assume that $E$ is given by the
Weierstrass model $E:y^{2}+a_{1}xy+a_{3}y=x^{3}$ for some integers $a_{1}$ and
$a_{2}$.

First, suppose $a_{1}\neq0$, then there are relatively prime integers $a$ and
$b$ with $a$ positive such that $E$ is $%
%TCIMACRO{\U{211a} }%
%BeginExpansion
\mathbb{Q}
%EndExpansion
$-isomorphic to $E_{T}\!\left(  a,b\right)  $ by Corollary \ref{sscorC3},
Lemma \ref{ss:Lem1}, and Lemma \ref{Lemposnega}. Next, observe that
$E_{T}\!\left(  a,b\right)  $ has $j$-invariant $0$ if and only if $\alpha
_{T}=a^{3}\left(  a-24b\right)  =0$. This is equivalent to $a=24b$. Since $a$
and $b$ are relatively prime, we conclude that $E$ is $%
%TCIMACRO{\U{211a} }%
%BeginExpansion
\mathbb{Q}
%EndExpansion
$-isomorphic to $E_{C_{3}}\!\left(  24,1\right)  $.

Next, suppose $a_{1}=0$ so that $E$ is given by the Weierstrass model
$E:y^{2}+a_{3}y=x^{3}$. In particular, $E$ has $j$-invariant $0$. Write
$a_{3}=c^{3}a$ where $a$ is a cubefree positive integer and observe that the
admissible change of variables $x\longmapsto c^{2}x$ and $y\longmapsto c^{3}y$
gives a $%
%TCIMACRO{\U{211a} }%
%BeginExpansion
\mathbb{Q}
%EndExpansion
$-isomorphism from $E$ onto $E_{C_{3}^{0}}\!\left(  a\right)  :y^{2}+ay=x^{3}$.
\end{proof}

\begin{theorem}
\label{propQ}\label{semistablecondthm}\label{ch:ss:mainthm}Let $E_{T}$ be as
given in Table~\ref{ta:ETmodel} and let
\begin{equation}
a=\left\{
\begin{array}
[c]{ll}%
c^{3}d^{2}e\text{ with }d,\text{ }e\text{ positive squarefree integers such that }\gcd\!\left(
d,e\right)  =1 & \text{if }T=C_{3},\\
c^{2}d\text{ with }d\text{ a squarefree integer} & \text{if }T=C_{4}.
\end{array}
\right.  \label{Cdef}%
\end{equation}
If the parameters of $E_{T}$ satisfy the
conclusion of Proposition~\ref{rationalmodels}, then the minimal discriminant
of $E_{T}$ is $u_{T}^{-12}\gamma_{T}$ where
\[
u_{T}=\left\{
\begin{array}
[c]{cl}%
1 & \text{if }T=C_{5},C_{7},C_{9},C_{3}^{0},\\
1\ \text{or }2 & \text{if }T=C_{6},C_{8},C_{10},C_{12},C_{2}\times C_{2},\\
1,2,\ \text{or\ }4 & \text{if }T=C_{2},C_{2}\times C_{4},\\
1,4,\ \text{or\ }16 & \text{if }T=C_{2}\times C_{6},\\
1,16,\ \text{or\ }64 & \text{if }T=C_{2}\times C_{8},\\
c^{2}d & \text{if }T=C_{3},\\
c\ \text{or\ }2c & \text{if }T=C_{4}.
\end{array}
\right.
\]
Moreover, when $T\neq C_{5},C_{7},C_{9},C_{3}^{0}$, there are necessary and sufficient conditions on the parameters of
$E_{T}$ to determine the exact value of $u_{T}$. Table \ref{ta:uTmin}
summarizes these necessary and sufficient conditions.
\end{theorem}
{\renewcommand{\arraystretch}{1.2} \begin{longtable}{ccl}
\caption{Necessary and sufficient conditions on $a,b,d$ to determine $u_{T}$}
\label{ta:uTmin}
\\
	
\hline
$T$  & $u_{T}$ & Conditions on $a,b,d$ \\
\hline
\endfirsthead
\caption[]{\emph{continued}}\\
\hline
$T$  & $u_{T}$ & Conditions on $a,b,d$\\
\hline
\endhead
\hline
\multicolumn{3}{r}{\emph{continued on next page}}
\endfoot
\hline
\endlastfoot
	
$C_{2}$ & $4$ & $v_{2}\left(  b^{2}d-a^{2}\right)  \geq8$ with
$v_{2}\!\left(  a\right)  =v_{2}\!\left(  b\right)  =1$ and $a\equiv2\ \operatorname{mod}8$ \\\cmidrule(lr){2-3}
& $2$  &  $v_{2}\left(  b^{2}d-a^{2}\right)  \geq8$ with $v_{2}\!\left(
a\right)  =v_{2}\!\left(  b\right)  =1$ and $a\equiv6\ \operatorname{mod}8$ \\\cmidrule(lr){3-3}
& & $4 \leq v_{2}\left(  b^{2}d-a^{2}\right)  \leq7$ with $v_{2}\!\left(
a\right)  =v_{2}\!\left(  b\right)  =1$ \\\cmidrule(lr){3-3}
&&   $v_{2}\!\left(  b\right)  \geq3$ and $a\equiv
3\ \operatorname{mod}4$ \\\cmidrule(lr){2-3}
& $1$ & The previous conditions are not satisfied \\\hline
$C_{4}$ & $2c$  & $v_{2}\!\left(  a\right)  \geq8$ is even with $bd\equiv3\ \operatorname{mod}4$\\\cmidrule(lr){2-3}
& $c$  & $v_{2}\!\left(  a\right) \leq7$ or $bd\not \equiv3\ \operatorname{mod}4$ \\\hline
$C_{6}$ & $2$  & $v_{2}\!\left(  a+b\right) \geq3$ \\\cmidrule(lr){2-3}
& $1$ & $v_{2}\!\left(  a+b\right) \leq 2$\\\hline
$C_{8}$ & $2$  & $v_{2}\!\left(a\right)  = 1$ \\\cmidrule(lr){2-3}
& $1$ & $v_{2}\!\left(a\right)  \neq 1$ \\\hline
$C_{10}$ & $2$  & $v_{2}\!\left(a\right)\geq 1$ \\\cmidrule(lr){2-3}
& $1$ & $v_{2}\!\left(a\right)=0$ \\\hline
$C_{12}$ & $2$  & $v_{2}\!\left(a\right)\geq 1$ \\\cmidrule(lr){2-3}
& $1$ & $v_{2}\!\left(a\right)=0$ \\\hline
$C_{2} \times C_{2}$ & $2$  & $v_{2}\!\left(a\right) \geq 4 $ and $bd\equiv 1\ \operatorname{mod}4$ \\\cmidrule(lr){2-3}
& $1$ & $v_{2}\!\left(  a\right) \leq3$ or $bd\not \equiv1\ \operatorname{mod}4$ \\\hline
$C_{2} \times C_{4}$ & $4$  & $v_{2}\!\left(  a\right) =2$ and $v_{2}\!\left(  a+4b\right) \geq 4 $ \\\cmidrule(lr){2-3}
& $2$ & $v_{2}\!\left(  a\right) \geq 2 $ and $v_{2}\!\left(  a+4b\right) \leq 3$ \\\cmidrule(lr){2-3}
& $1$ & $v_{2}\!\left(  a\right) \leq 1 $ \\\hline
$C_{2} \times C_{6}$ & $16$  & $v_{2}\!\left(  a+b\right)  =1$ \\\cmidrule(lr){2-3}
& $4$ & $v_{2}\!\left(  a+b\right)  \geq2$ \\\cmidrule(lr){2-3}
& $1$ & $v_{2}\!\left(  a+b\right) =0$ \\\hline
$C_{2} \times C_{8}$ & $64$  & $v_{2}\!\left(  a\right)  \geq2$ \\\cmidrule(lr){2-3}
& $16$ & $v_{2}\!\left(  a\right)  =1$ \\\cmidrule(lr){2-3}
& $1$ & $v_{2}\!\left(  a\right)  =0$ \\\hline
\end{longtable}}

In the proof of Theorem \ref{propQ}, we will consider the following two cases separately: $\alpha_T\beta_T=0$ and $\alpha_T\beta_T\neq 0$. Observe that the former occurs if and only if the $j$-invariant of $E_T$ is $0$ or $1728$. Consequently, we need knowledge of when $E_{T}$ has $j$-invariant~$0$~or~$1728$ under
the assumptions that the parameters of $E_{T}$ satisfy the conclusion of
Proposition~\ref{rationalmodels}. Below we prove a series of lemmas so that we can
distinguish those $E_{T}$'s whose $j$-invariant is $0$ or $1728$.

\begin{lemma}
\label{sslemc4j0}Let $E$ be a rational elliptic curve with a rational torsion
point of order $N\geq4$. If $E$ has $j$-invariant $0$, then $E$ is $%
%TCIMACRO{\U{211a} }%
%BeginExpansion
\mathbb{Q}
%EndExpansion
$-isomorphic to $E_{C_{6}}\!\left(  3,-1\right)  $. If $E$ has $j$-invariant
$1728$, then $E$ is $%
%TCIMACRO{\U{211a} }%
%BeginExpansion
\mathbb{Q}
%EndExpansion
$-isomorphic to $E_{C_{4}}\!\left(  8,-1\right)  $.
\end{lemma}

\begin{proof}
From (\ref{basicformulas}) it is checked that $j=0$ if and only if $c_{4}=0$.
Similarly, $j=1728$ if and only if $c_{6}=0$. By Proposition \ref{rationalmodels}, $E$ is $\mathbb{Q}$-isomorphic to $E_{T}=E_{T}\!\left(  a,b\right)  $ for some $T$ with $a$
positive. We now consider the cases when $j=0$ and $j=1728$.

\textbf{Case 1.} Suppose $j=0$. Then the invariant $c_{4}$ of $E_{T}$ is $0$
and thus $\alpha_{T}=0$. It suffices to check when there are integer solutions
to the equations $\alpha_{T}=0$. By inspection, this only occurs for $T=C_{6}$
with $\left(  a,b\right)  =\left(  3,-1\right)  $.

\textbf{Case 2.} Suppose $j=1728$. Then the invariant $c_{6}$ of $E$ is $0$
and by inspection, the equation $\beta_{T}=0$ has integer solutions only when
$T=C_{4}$ with $\left(  a,b\right)  =\left(  8,-1\right)  $.
\end{proof}

\begin{lemma}
\label{sselmmac2result}Let $E_{T}=E_{C_{2}}\!\left(  a,b,d\right)  $ with
$d\neq1$, $b\neq0$ such that $d$ and $\gcd\!\left(  a,b\right)  $ are squarefree.

$\left(  1\right)  $ If $E_{T}$ has $j$-invariant $0$, then $\left(
a,b,d\right)  =\left(  3b,b,-3\right)  $ for $b$ a squarefree integer.

$\left(  2\right)  $ If $E_{T}$ has $j$-invariant $1728$, then $E_{T}$ is $%
%TCIMACRO{\U{211a} }%
%BeginExpansion
\mathbb{Q}
%EndExpansion
$-isomorphic to $E_{T}\!\left(  0,m,d\right)  $ where $m$ is the squarefree
part of $b$.
\end{lemma}

\begin{proof}
$\left(  1\right)  $ If $E_{T}$ has $j$-invariant $0$, then $\alpha
_{T}=16\left(  3b^{2}d+a^{2}\right)  =0$. In particular, $a^{2}=-3b^{2}d$.
Since $\gcd\!\left(  a,b\right)  $ and $d$ are squarefree, it follows that
$\left(  a,d\right)  =\left(  3b,-3\right)  $ with $b$ a squarefree integer.

$\left(  2\right)  $ If $E_{T}$ has $j$-invariant $1728$, then $\beta
_{T}=-64a\left(  9b^{2}d-a^{2}\right)  =0$. Therefore $a=0$ or $a^{2}=9b^{2}%
d$. Since $d\neq1$, it follows that the latter cannot occur. Consequently,
$a=0$. Now suppose $b=k^{2}m$ for $m$ a squarefree integer. Then the
admissible change of variables $x\longmapsto k^{2}x$ and $y\longmapsto k^{3}y$
gives a $%
%TCIMACRO{\U{211a} }%
%BeginExpansion
\mathbb{Q}
%EndExpansion
$-isomorphism from $E_{T}$ onto $E_{T}\left(  0,m,d\right)  :y^{2}=x^{3}%
-m^{2}dx$.
\end{proof}

\begin{lemma}
\label{sselmmac2c2result}Let $E_{T}=E_{C_{2}\times C_{2}}\!\left(
a,b,d\right)  $ with $\gcd\!\left(  a,b\right)  =1,$ $d$ squarefree, and $a$
even. Then the $j$-invariant of $E_{T}$ is non-zero. Moreover, if $E_{T}$ has
$j$-invariant $1728$, then $E_{T}$ is $%
%TCIMACRO{\U{211a} }%
%BeginExpansion
\mathbb{Q}
%EndExpansion
$-isomorphic to $E_{T}\!\left(  2,1,d\right)  $ for some positive squarefree integer
$d$.
\end{lemma}

\begin{proof}
First, observe that $a^{2}-ab+b^{2}=0$ has no non-zero integer solutions. Then
$\alpha_{T}\neq0$ and therefore the $j$-invariant of $E_{T}$ is non-zero.
Next, suppose the $j$-invariant of $E_{T}$ is $1728$. Then%
\[
\beta_{T}=-32d^{3}\left(  a+b\right)  \left(  a-2b\right)  \left(
2a-b\right)  =0.
\]
Since $a$ and $b$ are relatively prime with $a$ even, we have that $\beta
_{T}=0$ if and only if $\left(  a,b\right)  =\left(  2,1\right)  $ or $\left(
a,b\right)  =\left(  -2,-1\right)  $. The admissible change of variables
$x\longmapsto x-2d$ gives a $%
%TCIMACRO{\U{211a} }%
%BeginExpansion
\mathbb{Q}
%EndExpansion
$-isomorphism from $E_{T}\!\left(  2,1,d\right)  $ onto $E_{T}\!\left(
-2,-1,d\right)  $, which concludes the proof.
\end{proof}

\section{Proof of Theorem \ref{semistablecondthm}\label{pfmainthm}}

The proof of Theorem \ref{propQ} will rely on the following theorem of Kraus \cite{MR1024419}.

\begin{theorem}
\label{kraus}Let $\alpha,\beta,$ and $\gamma$ be integers such that
$\alpha^{3}-\beta^{2}=1728\gamma$ with $\gamma\neq0$. Then there exists a
rational elliptic curve $E$ given by an integral Weierstrass equation having
invariants $c_{4}=\alpha$ and $c_{6}=\beta$ if and only if the following
conditions hold:

$\left(  1\right)  \ v_{3}\!\left(  \beta\right)  \neq2$ and

$\left(  2\right)  \ $either $\beta\equiv3\ \operatorname{mod}4$ or both
$v_{2}\!\left(  \alpha\right)  \geq4$ and $\beta\equiv0$ or $8$
$\operatorname{mod}32$.
\end{theorem}

The following corollary is automatic by Lemma \ref{silverconverse} and the
definition of an integral Weierstrass model.

\begin{corollary}
\label{CorKraus}Let $E$ be a rational elliptic with discriminant $\Delta$ and let $c_{4}$ and $c_{6}$ be the invariants associated to $E$. Suppose further that $\Delta,c_4,c_6\in \mathbb{Z}$. If $c_6 \not \equiv 3 \operatorname{mod}4$ and either $v_2(c_4)<4$ or $c_6 \not \equiv 0,8 \operatorname{mod}32$, then $\Delta$ is not the minimal discriminant of $E$. 
\end{corollary}

\begin{lemma}
\label{ch:ss:elemlemm}Let $\alpha,\beta,$ and $\gamma$ be integers such that
$\alpha^{3}-\beta^{2}=1728\gamma$ with $\gamma\neq0$. If $v_{2}\!\left(
\alpha\right)  =4k$ and $v_{2}\!\left(  \gamma\right)  \geq12k$ for some
nonnegative integer $k$, then $v_{2}\!\left(  \beta\right)  =6k$.
\end{lemma}

\begin{proof}
The assumption implies that $v_{2}(  \alpha^{3})  =12k$. Then
$v_{2}(  \alpha^{3}-1728\gamma)  =12k$. Thus $v_{2}(
\beta)=6k$.
\end{proof}

By Lemma \ref{sslemc4j0}, we have that if a rational elliptic curve with a non-trivial torsion subgroup $T$ has $j$-invariant $0$ or $1728$, then $T= C_{2},C_{3},C_{4},C_{6},$ and $C_{2}\times C_{2}$. We will implicitly assume this result in the proof of Theorem \ref{semistablecondthm}.

\subsection[Proof of Theorem \ref{semistablecondthm} for $T=C_{3}^{0}%
,C_{5},C_{7},C_{9}$]{Proof of Theorem \ref{semistablecondthm} for
\bm{$T=C_{3}^{0},C_{5},C_{7},C_{9}$}.}

\begin{claim}
[Theorem \ref{semistablecondthm} for \bm{$T=C_{3}^{0},C_{5},C_{7},C_{9}$}.]If
the parameters of $E_{T}$ satisfy the conclusion of
Proposition~\ref{rationalmodels} for $T=C_{3}^{0},C_{5},C_{7},C_{9}$, then the
minimal discriminant of $E_{T}$ is~$\gamma_{T}$.
\end{claim}

\begin{proof}
First, suppose $T=C_{3}^{0}$ so that $E_{T}=E_{C_{3}^{0}}\!\left(  a\right)  $
with $a$ cubefree. Then $\gamma_{T}=-27a^{4}$ and thus $v_{p}\!\left(
\gamma_{T}\right)  \leq11$ for each prime $p$. It now follows that the minimal
discriminant of $E_{T}$ is $\gamma_{T}$ since $E_{T}$ is given by an integral
Weierstrass model.

Next, suppose $T=C_{5},C_{7},$ or $C_{9}$ and let $E_{T}=E_{T}\!\left(
a,b\right)  $ for relatively prime integers $a$ and $b$. Since $E_{T}$ is
given by an integral Weierstrass model, Lemma \ref{silverconverse} implies
that there is a unique positive integer $u_{T}$ such that $\Delta_{E_{T}%
}^{\text{min}}=u_{T}^{-12}\gamma_{T}$. In particular, $u_{T}^{6}$ divides
$\gcd\!\left(  \beta_{T},\gamma_{T}\right)  $. Since $a$ and $b$ are
relatively prime, we have by Lemma \ref{polynomials}, that $\gcd\!\left(
\beta_{T},\gamma_{T}\right)  $ divides $d_{T}$ where%
\[
d_{T}=\left\{
\begin{array}
[c]{ll}%
5^{3} & \text{if }T=C_{5},\\
7 & \text{if }T=C_{7},\\
3^{3} & \text{if }T=C_{9}.
\end{array}
\right.
\]
In particular, $u_{T}=1$ which shows that $\gamma_{T}$ is the minimal
discriminant of $E_{T}$.
\end{proof}

\subsection[Proof of Theorem \ref{semistablecondthm} for $T=C_{2}$]{Proof of
Theorem \ref{semistablecondthm} for \bm{$T=C_{2}$}}

\begin{claim}
[Theorem \ref{semistablecondthm} for \bm{$T=C_{2}$}.]Let $E_{T}=E_{T}\!\left(
a,b,d\right)  $ where $a,b,d$ are integers with $d\neq1,b\neq0$ such that
$\gcd\!\left(  a,b\right)  $ and $d$ are squarefree. Then the minimal
discriminant of $E_{T}$ is $u_{T}^{-12}\gamma_{T}$ with $u_{T}\in\left\{
1,2,4\right\}  $. Moreover,

$\left(  1\right)  $ $u_{T}=4$ if and only if $v_{2}\left(  b^{2}%
d-a^{2}\right)  \geq8$ with $v_{2}\!\left(  a\right)  =v_{2}\!\left(
b\right)  =1$ and $a\equiv2\ \operatorname{mod}8$;

$\left(  2\right)  $ $u_{T}=2$ if and only if $\left(  i\right)
\ v_{2}\left(  b^{2}d-a^{2}\right)  \geq8$ with $v_{2}\!\left(  a\right)
=v_{2}\!\left(  b\right)  =1$ and $a\equiv6\ \operatorname{mod}8$, $\left(
ii\right)  \ 4\leq v_{2}\left(  b^{2}d-a^{2}\right)  \leq7$ with
$v_{2}\!\left(  a\right)  =v_{2}\!\left(  b\right)  =1$, or $\left(
iii\right)  $ $v_{2}\!\left(  b\right)  \geq3$ and $a\equiv
3\ \operatorname{mod}4$;

$\left(  3\right)  $ $u_{T}=1$ if and only if the above conditions do not hold.
\end{claim}

\begin{proof}
By Lemma \ref{silverconverse}, $\Delta_{E_{T}}^{\text{min}}=u_{T}^{-12}%
\gamma_{T}$ for some positive integer $u_{T}$ and recall that
\[
\alpha_{T}=16\left(  3b^{2}d+a^{2}\right)  ,\qquad\beta_{T}=-64a\left(
9b^{2}d-a^{2}\right)  ,\qquad\gamma_{T}=64b^{2}d\left(  b^{2}d-a^{2}\right)
^{2}.
\]

First, suppose the $j$-invariant of $E_{T}$ is $0$. By Lemma
\ref{sselmmac2result}, we may assume that $\left(  a,b,d\right)  =\left(
3b,b,-3\right)  $ with $b$ a squarefree integer. Then $\gamma_{T}%
=-2^{10}3^{3}b^{6}$. In particular, if $b$ is odd, then $v_{p}\!\left(
\gamma_{T}\right)  <12$ for all primes $p$ and therefore $\gamma_{T}$ is the
minimal discriminant of $E_{T}$. Now suppose $b$ is even. The admissible
change of variables $x\longmapsto4x$ and $y\longmapsto8y$ gives a $%
%TCIMACRO{\U{211a} }%
%BeginExpansion
\mathbb{Q}
%EndExpansion
$-isomorphism from $E_{T}$ onto%
\[
E_{T}^{\prime}:y^{2}=x^{3}+\frac{3b}{2}x^{2}+\frac{3b^{2}}{4}x.
\]
Then $u_{T}^{-12}\gamma_{T}=-2^{-2}3^{3}b^{6}$ with $u_{T}=2$ and thus
$v_{p}\!\left(  u_{T}^{-12}\gamma_{T}\right)  <12$ for each prime $p$. Hence
$u_{T}^{-12}\gamma_{T}$ is the minimal discriminant of $E_{T}$ since
$E_{T}^{\prime}$ is given by an integral Weierstrass model. Note that this
agrees with the claim since the current assumptions imply that $v_{2}\!\left(
a\right)  =v_{2}\!\left(  b\right)  =1$ with $v_{2}\!\left(  b^{2}%
d-a^{2}\right)  =4$.

Next, suppose the $j$-invariant of $E_{T}$ is $1728$. By Lemma
\ref{sselmmac2result}, $E_{T}$ is $%
%TCIMACRO{\U{211a} }%
%BeginExpansion
\mathbb{Q}
%EndExpansion
$-isomorphic to $E_{T}\!\left(  0,b,d\right)  $ for squarefree integers $b$
and $d$. Then $\alpha_{T}=2^{4}3b^{2}d$ and $\gamma_{T}=2^{6}b^{6}d^{3}$. In
particular, $v_{p}\!\left(  \gamma_{T}\right)  \leq9$ for each odd prime $p$
and $v_{2}\!\left(  \gamma_{T}\right)  \leq15$. It follows that $u_{T}$
divides $2$. Towards a contradiction, suppose $u_{T}=2$ so that $v_{2}%
\!\left(  \gamma_{T}\right)  \geq12$. This is equivalent to $b$ being even and
so $v_{2}\!\left(  u_{T}^{-4}\alpha_{T}\right)  =2$. Since $\beta_{T}=0$,
Corollary \ref{CorKraus} asserts that $u_{T}^{-12}\gamma_{T}$ is not the
minimal discriminant of $E_{T}$, which is a contradiction. Consequently, the
minimal discriminant of $E_{T}$ is $\gamma_{T}$.

Now suppose that the $j$-invariant of $E_{T}$ is not equal to $0$ or $1728$.
Next, let $\gcd\!\left(  a,b,d\right)  =m$. Then 
$\gcd\!\left(  a,b\right)  =mn$ and $\gcd\!\left(  a,d\right)  =ml$ for some positive integers $n,l$. Note that the assumption that $\gcd\!\left(  a,b\right)$ and $d$ are squarefree implies that the $m,n$, and $l$ are squarefree. In particular,  $\gcd\!\left(  m,n\right)= \gcd\!\left(  m,l\right)=1$. Now observe that  $\gcd\!\left(  n,l\right)$ divides $a,b$ and $d$. Therefore $\gcd\!\left(  n,l\right)$ divides~$m$. Since $\gcd\!\left(  m,l\right)=1$, we conclude that $\gcd\!\left(  n,l\right)=1$. Also note that $\gcd\!\left(  b,l\right)  =1$. In sum, $m,n,$ and $l$ are pairwise relatively prime positive squarefree integers and
\[
a=mnl\tilde{a},\qquad b=mn\tilde{b},\qquad\text{and}\qquad d=ml\tilde{d}%
\]
for some integers $\tilde{a},\tilde{b},$ and $\tilde{d}$. By Lemma
\ref{polynomials},%
\begin{equation}%
\begin{tabular}
[c]{rrl}%
$\gcd\!\left(  \alpha_{T},\beta_{T}\right)  \qquad$divides$\qquad$ &
$2^{8}3^{2}\gcd\!\left(  b^{4}d^{2},a^{3}\right)  $ & \hspace{-1.1em} $=2^{8}3^{2}m^{3}%
n^{3}l^{2},$\\
$\gcd\!\left(  \alpha_{T},\gamma_{T}\right)  \qquad$divides$\qquad$ &
$2^{10}\gcd\!\left(  b^{6}d^{3},a^{6}\right)  $ & \hspace{-1.1em} $=2^{10}m^{6}n^{6}l^{3},$\\
$\gcd\!\left(  \beta_{T},\gamma_{T}\right)  \qquad$divides$\qquad$ &
$2^{12}\gcd\!\left(  b^{8}d^{4},a^{7}\right)  $ & \hspace{-1.1em} $=2^{12}m^{7}n^{7}l^{4}.$%
\end{tabular}
\ \ \label{c2thm1}%
\end{equation}

We claim that $u_{T}$ divides $4$. To this end, suppose $p$ is an odd prime
dividing $u_{T}$. If $p>3$, then $p^{4}$ divides $m^{3}n^{3}l^{2}$ by
(\ref{c2thm1}). But this is impossible since $m,n,$ and $l$ are relatively
prime squarefree integers. So suppose $p=3$. Then $3$ does not divide $l$ by
(\ref{c2thm1}) since this would imply that $3^{4}$ does not divide
$\gcd\!\left(  \alpha_{T},\gamma_{T}\right)  $. We may therefore assume that
$3$ divides either $m$ or $n$.

Suppose $3$ divides $u_{T}$ and $m$. Then $a=3\hat{a},\ b=3\hat{b},$ and
$d=3\hat{d}$ for some integers $\hat{a},\hat{b},\hat{d}$ with $3$ dividing at
most one of $\hat{a}$ and $\hat{b}$. Then$\ 4\leq v_{3}\!\left(  \alpha
_{T}\right)  =2+v_{3}\!\left(  \hat{a}^{2}+9\hat{b}^{2}\hat{d}\right)  $. Thus
$v_{3}\!\left(  \hat{a}\right)  >0$, which implies that $3$ does not divide
$\hat{b}$. These assumptions imply that%
\[
9=v_{3}\!\left(  \gamma_{T}\right)  =v_{3}\!\left(  27\hat{b}^{2}\hat
{d}\right)  +2v_{3}\!\left(  27\hat{b}^{2}\hat{d}-9\hat{a}^{2}\right)
\]
which contradicts the assumption that $3$ divides $u_{T}$.

Now suppose that $3$ divides $u_{T}$ and $n$. In particular, $d$ is not
divisible by $3$. Then $v_{3}\!\left(  \alpha_{T}\right)  \leq3$ which
contradicts the assumption that $u_{T}$ is divisible by $3$.

Since $u_{T}$ is not divisible by odd primes, we conclude that $u_{T}$ divides
$4$ by (\ref{c2thm1}). Indeed, $u_{T}^{4}$ divides $\gcd\!\left(  \alpha
_{T},\beta_{T}\right)  $ and $v_{2}\!\left(  \gcd\!\left(  \alpha_{T}%
,\beta_{T}\right)  \right)  \leq11$. Moreover, $v_{3}\!\left(  u_{T}^{-6}%
\beta_{T}\right)  =v_{3}\!\left(  \beta_{T}\right)  \neq2$ by
Theorem~\ref{kraus} since $E_{T}$ is given by an integral Weierstrass model.
We now show the theorem by considering the cases when $u_{T}$ is $4,2,$ or $1$.

\textbf{Case 1.} Suppose $u_{T}=4$ so that $v_{2}\!\left(  \alpha_{T}\right)
\geq8,\ v_{2}\!\left(  \beta_{T}\right)  \geq12$, and $v_{2}\!\left(
\gamma_{T}\right)  \geq24$. In particular, $v_{2}\!\left(  3b^{2}%
d+a^{2}\right)  \geq4$. This implies that either $\left(  i\right)  $ $abd$ is
odd or $\left(  ii\right)  $ $v_{2}\!\left(  a\right)  =v_{2}\!\left(
b\right)  =1$ with $d\equiv1\ \operatorname{mod}4$.

\qquad\textbf{Subcase 1a.} Assume that $abd$ is odd. Then the assumptions on
$v_{2}\!\left(  \beta_{T}\right)  $ and $v_{2}\!\left(  \gamma_{T}\right)  $
imply that $6\leq v_{2}\!\left(  9b^{2}d-a^{2}\right)  $ and $9\leq
v_{2}\!\left(  b^{2}d-a^{2}\right)  $, respectively. Thus $9b^{2}d-a^{2}%
\equiv0\ \operatorname{mod}64$. But this is a contradiction since
$9b^{2}d-a^{2}\equiv8b^{2}d\ \operatorname{mod}64$ is non-zero.

\qquad\textbf{Subcase 1b.} Assume that $v_{2}\!\left(  a\right)
=v_{2}\!\left(  b\right)  =1$ with $d\equiv1\ \operatorname{mod}4$. Then
$v_{2}\!\left(  \gamma_{T}\right)  \geq24$ implies that $v_{2}\!\left(
b^{2}d-a^{2}\right)  \geq8$. Consequently, $3b^{2}d+a^{2}\equiv2\left(
b^{2}d+a^{2}\right)  \ \operatorname{mod}32$. From this, we conclude that
$v_{2}\!\left(  \alpha_{T}\right)  =8$. Next, observe that $u_{T}^{-6}%
\beta_{T}=-2^{-6}a\left(  8b^{2}d+b^{2}d-a^{2}\right)  $. Therefore
$u_{T}^{-6}\beta_{T}\equiv\frac{-a}{2}\ \operatorname{mod}4$ since
$\frac{b^{2}d}{4}\equiv1\ \operatorname{mod}4$. By Theorem \ref{kraus}, there
is an integral Weierstrass model having $u_{T}^{-4}\alpha_{T}$ and $u_{T}%
^{-6}\beta_{T}$ as its invariants $c_{4}$ and $c_{6}$, respectively, if and
only if $\frac{a}{2}\equiv1\ \operatorname{mod}4$. This shows that $u_{T}=4$
if and only if $a\equiv2\ \operatorname{mod}8$, which completes the proof of
$\left(  1\right)  $.

\textbf{Case 2.} Suppose $u_{T}=2$ so that $v_{2}\!\left(  \alpha_{T}\right)
\geq4,\ v_{2}\!\left(  \beta_{T}\right)  \geq6,$ and $v_{2}\!\left(
\gamma_{T}\right)  \geq12$. In particular, $v_{2}\!\left(  b^{2}d\right)
+2v_{2}\!\left(  b^{2}d-a^{2}\right)  \geq6$. Observe that this is equivalent to $v_2(b)+v_2(d)/2+v_2(b^2d-a^2)\geq3$. Since $d$ is squarefree, we deduce that this inequality holds if and only if $v_{2}\!\left(  b\right)
+v_{2}\!\left(  b^{2}d-a^{2}\right)\geq 3$. We now proceed by cases, by conditioning upon the possible values of $v_2(b)$.

\qquad\textbf{Subcase 2a.} Suppose $v_2(b)=0$, so that $v_{2}\!\left(  b^{2}d-a^{2}\right)  \geq3$. Then the inequality holds if and only if $a$ is odd with
$d\equiv1\ \operatorname{mod}8$. Hence $u_{T}^{-4}\alpha_{T}\equiv
4\ \operatorname{mod}8$ since $u_{T}^{-4}\alpha_{T}=b^{2}d-a^{2}%
+2b^{2}d+2a^{2}$. Since $u_{T}^{-6}\beta_{T}$ is even, we have a contradiction
by Corollary \ref{CorKraus} which asserts that $u_{T}^{-12}\gamma_{T}$ is not
the minimal discriminant of $E_{T}$.

\qquad\textbf{Subcase 2b.} Suppose $v_{2}\!\left(  b\right)  =1$, so that $v_{2}\!\left(  b^{2}d-a^{2}\right)  \geq2$. Then $a$ is even, and thus $u_{T}%
^{-6}\beta_{T}$ is even. By Corollary \ref{CorKraus}, $v_{2}\!\left(
u_{T}^{-4}\alpha_{T}\right)  \geq4$. This inequality holds if and only if
$v_{2}\!\left(  a\right)  =1$ with $d\equiv1\ \operatorname{mod}4$. Under
these assumptions, it is easily verified that $v_{2}\!\left(  b^{2}%
d-a^{2}\right)  \geq4$. Consequently, $9b^{2}d-a^{2}\equiv
0\ \operatorname{mod}16$ and thus $u_{T}^{-6}\beta_{T}\equiv
0\ \operatorname{mod}32$. By Theorem \ref{kraus} we conclude that there is an
integral Weierstrass model having $u_{T}^{-4}\alpha_{T}$ and $u_{T}^{-6}%
\beta_{T}$ as its invariants $c_{4}$ and $c_{6}$, respectively. This and Case
1 above imply that under the assumptions of this subcase, $u_{T}=2$ if and
only if $\left(  i\right)  \ v_{2}\!\left(  b^{2}d-a^{2}\right)  \geq8$ with
$v_{2}\!\left(  a\right)  =v_{2}\!\left(  b\right)  =1$ and $a\equiv
6\ \operatorname{mod}8$ or $\left(  ii\right)  $ $4\leq v_{2}\!\left(
b^{2}d-a^{2}\right)  <8$ with $v_{2}\!\left(  a\right)  =v_{2}\!\left(
b\right)  =1$.

\qquad\textbf{Subcase 2c.} Suppose $v_{2}\!\left(  b\right)  =2$, so that
$v_{2}\!\left(  b^{2}d-a^{2}\right)  \geq1$. Then $v_{2}\!\left(  a\right)  =1$ since
$\gcd\!\left(  a,b\right)  $ is squarefree and thus $v_{2}\!\left(
b^{2}d-a^{2}\right)  =2$. Consequently, $u_{T}^{-6}\beta_{T}$ is even and
$v_{2}\!\left(  u_{T}^{-4}\alpha_{T}\right)  =2$. But this is a contradiction,
since Corollary \ref{CorKraus} implies that $u_{T}^{-12}\gamma_{T}$ is not the
minimal discriminant of $E_{T}$.

\qquad\textbf{Subcase 2d.} Suppose $v_{2}\!\left(  b\right)  \geq3$, so that $v_{2}\!\left(  b^{2}d-a^{2}\right)  \geq0$. If $a$ is even, then $v_{2}\!\left(
a\right)  =1$. Hence $v_{2}\!\left(  b^{2}d-a^{2}\right)  =2$. But this leads
to a contradiction by the argument given in Subcase~2c. So suppose $a$ is odd
so that $v_{2}\!\left(  b^{2}d-a^{2}\right)  =0$. Then by inspection,
$u_{T}^{-4}\alpha_{T}$ is an integer and $u_{T}^{-6}\beta_{T}\equiv
a^{3}\ \operatorname{mod}4$. By Theorem \ref{kraus} we conclude that there is
an integral Weierstrass model having $u_{T}^{-4}\alpha_{T}$ and $u_{T}%
^{-6}\beta_{T}$ as its invariants $c_{4}$ and $c_{6}$, respectively, if and
only if $a\equiv3\ \operatorname{mod}4$, which concludes the proof of $\left(
2\right)  $.

\textbf{Case 3.} Suppose $u_{T}=1$. Since $E_{T}$ is given by an integral
Weierstrass model, we have that $u_{T}=1$ if and only if $u_{T}\neq2,4$, which
concludes the proof.
\end{proof}

\subsection[Proof of Theorem \ref{semistablecondthm} for $T=C_{3}$]{Proof of
Theorem \ref{semistablecondthm} for \bm{$T=C_{3}$}}

\begin{claim}
[Theorem \ref{semistablecondthm} for \bm{$T=C_{3}$}.]Let $E_{T}=E_{C_{3}%
}\!\left(  a,b\right)  $ where $a$ and $b$ are relatively prime integers with
$a$ positive. Write $a=c^{3}d^{2}e$ with $d,e$ positive relatively prime
squarefree integers. Then the minimal discriminant of $E_{T}$ is $u_{T}%
^{-12}\gamma_{T}$ where $u_{T}=c^{2}d$.
\end{claim}

\begin{proof}
Let $u_{T}=c^{2}d$. Then the admissible change of variables $x\longmapsto
u_{T}^{2}x$ and $y\longmapsto u_{T}^{3}y$ results in a $%
%TCIMACRO{\U{211a} }%
%BeginExpansion
\mathbb{Q}
%EndExpansion
$-isomorphism between $E_{T}$ and the elliptic curve%
\[
E_{T}^{\prime}:y^{2}+cdexy+de^{2}by=x^{3}.
\]
In particular, $E_{T}^{\prime}$ is given by an integral Weierstrass model and%
\begin{align*}
u_{T}^{-4}\alpha_{T}  &  =cd^{2}e^{3}\left(  a-24b\right)  ,\qquad u_{T}%
^{-6}\beta_{T}=d^{2}e^{4}\left(  -a^{2}+36ab-216b^{2}\right)  , \\
u_{T}^{-12}\gamma_{T}  &  =d^{4}e^{8}b^{3}\left(  a-27b\right)  .
\end{align*}

First, suppose $E_{T}$ has $j$-invariant $0$. Then by Proposition
\ref{rationalmodels}, $\left(  c,d,e,b\right)  =\left(  2,1,3,1\right)  $. The
theorem now follows since $u_{T}^{-12}\gamma_{T}=-3^{9}$ is the minimal
discriminant of $E_{T}\!\left(  24,1\right)  $. Next, observe that $E_{T}$
does not have $j$-invariant $1728$ since $\beta_{T}=0$ does not have solutions
in the positive integers.

Now suppose the $j$-invariant of $E_{T}$ is not equal to $0$. We claim that
$E_{T}^{\prime}$ is a global minimal model for $E_{T}$. By Lemma
\ref{polynomials}, $\gcd\!\left(  \beta_{T},\gamma_{T}\right)  $ divides
$2^{6}3^{9}a^{4}$ since $\gcd\!\left(  a,b\right)  =1$. In particular,
$u_{T}^{-6}\gcd\!\left(  \beta_{T},\gamma_{T}\right)  $ divides $2^{6}%
3^{9}d^{2}e^{4}$. Thus $\gcd\!\left(  u_{T}^{-6}\beta_{T},u_{T}^{-12}%
\gamma_{T}\right)  $ divides $2^{6}3^{9}d^{2}e^{4}$. By Lemma
\ref{silverconverse}, there is a unique positive integer $w$ such that
$\Delta_{E_{T}}^{\text{min}}=w^{-12}u_{T}^{-12}\gamma_{T}$. It suffices to
show that $w=1$. To this end, observe that $w^{6}$ divides $2^{6}3^{9}%
d^{2}e^{4}$. In particular, $v_{p}\!\left(  w\right)  =0$ for all primes
$p\geq5$ since $d$ and $e$ are relatively prime squarefree integers.

If $v_{3}\!\left(  w\right)  >0$, then $12\leq v_{3}\!\left(  u_{T}%
^{-12}\gamma_{T}\right)  =v_{3}\!\left(  d^{4}e^{8}b^{3}\right)
+v_{3}\!\left(  a-27b\right)  $. First, suppose $v_{3}\!\left(  a\right)  >0$
with $v_{3}\!\left(  a\right)  \neq3$. Then $v_{3}\!\left(  a-27b\right)
\leq3$ and $v_{3}\!\left(  d^{4}e^{8}b^{3}\right)  \leq8$ since $a,b$ are
relatively prime and $d,e$ are relatively prime and squarefree. But this is a
contradiction. Suppose instead that $v_{3}\!\left(  a\right)  =3$. In
particular, $v_{3}\!\left(  c\right)  =1$ and $v_{3}\!\left(  de\right)  =0$.
Thus $v_{3}\!\left(  u_{T}^{-4}\alpha_{T}\right)  =2$, which is a
contradiction. Lastly, suppose $v_{3}\!\left(  a\right)  =0$. Then $u_{T}%
^{-4}\alpha_{T}\equiv1\ \operatorname{mod}3$, which contradicts the assumption
that $v_{3}\!\left(  u_{T}^{-4}\alpha_{T}\right)  \geq4$. We conclude that
$v_{3}\!\left(  w\right)  =0$.

If $v_{2}\!\left(  w\right)  >0$, then $v_{2}\!\left(  u_{T}^{-4}\alpha
_{T}\right)  \geq4$. Thus $a$ is even, and consequently, $a-27b$ is odd. But
then $v_{2}\!\left(  u_{T}^{-12}\gamma_{T}\right)  \leq8$, which is our
desired contradiction. Thus $v_{p}\!\left(  w\right)  =0$ for each prime $p$
and hence $w=1$, which shows that $u_{T}^{-12}\gamma_{T}$ is the minimal
discriminant of $E_{T}$.
\end{proof}

\subsection[Proof of Theorem \ref{semistablecondthm} for $T=C_{4}$]{Proof of
Theorem \ref{semistablecondthm} for \bm{$T=C_{4}$}}

\begin{claim}
[Theorem \ref{semistablecondthm} for \bm{$T=C_{4}$}.]Let $E_{T}=E_{C_{4}%
}\!\left(  a,b\right)  $ where $a$ and $b$ are relatively prime integers with
$a$ positive. Write $a=c^{2}d$ for $d$ a positive squarefree integer. Then the
minimal discriminant of $E_{T}$ is $u_{T}^{-12}\gamma_{T}$ where $u_{T}%
\in\left\{  c,2c\right\}  $. Moreover, $u_{T}=2c$ if and only if
$v_{2}\!\left(  a\right)  \geq8$ is even with $bd\equiv3\ \operatorname{mod}4$.
\end{claim}

\begin{proof}
Let $u_{T}=c$. Then the admissible change of variables $x\longmapsto u_{T}%
^{2}x$ and $y\longmapsto u_{T}^{3}y$ results in a $%
%TCIMACRO{\U{211a} }%
%BeginExpansion
\mathbb{Q}
%EndExpansion
$-isomorphism between $E_{T}$ and the elliptic curve%
\begin{equation}
E_{T}^{\prime}:y^{2}+cdxy-cd^{2}by=x^{3}-bdx^{2}. \label{ETprimeC4}%
\end{equation}
In particular, $E_{T}^{\prime}$ is given by an integral Weierstrass model and%
\begin{align*}
c^{-4}\alpha_{T}  &  =d^{2}\left(  a^{2}+16ab+16b^{2}\right)  ,\qquad
c^{-6}\beta_{T}=d^{3}\left(  a+8b\right)  \left(  -a^{2}-16ab+8b^{2}\right)
,\\
c^{-12}\gamma_{T}  &  =b^{4}c^{2}d^{7}\left(  a+16b\right)  .
\end{align*}

First, suppose $E_{T}$ has $j$-invariant $0$ or $1728$. By Lemma
\ref{sslemc4j0}, we may assume $\left(  c,d,b\right)  =\left(  2,2,-1\right)
$. The theorem now follows since $u_{T}^{-12}\gamma_{T}=-2^{12}$ is the
minimal discriminant of $E_{T}\!\left(  8,-1\right)  $.

Next, suppose that the $j$-invariant of $E_{T}$ is not $0$ or $1728$. By Lemma
\ref{polynomials}, $\gcd\!\left(  \beta_{T},\gamma_{T}\right)  $ divides
$2^{18}a^{3}$ since $a$ and $b$ are relatively prime integers. Since
$\gcd\!\left(  c^{-6}\beta_{T},c^{-12}\gamma_{T}\right)  $ divides $c^{-6}%
\gcd\!\left(  \beta_{T},\gamma_{T}\right)  $, we have that $\gcd\!\left(
c^{-6}\beta_{T},c^{-12}\gamma_{T}\right)  $ divides $2^{18}d^{3}$. By Lemma
\ref{silverconverse}, there is a unique positive integer $w$ such that
$\Delta_{E_{T}}^{\text{min}}=w^{-12}c^{-12}\gamma_{T}$. It follows that
$w^{6}$ divides $2^{18}d^{3}$ and thus $v_{p}\!\left(  w\right)  =0$ for all
odd primes $p$ since $d$ is squarefree. Therefore $w$ divides $8$. We claim
that $w$ divides $2$. Towards a contradiction, suppose $4$ divides $w$. Then
$v_{2}\!\left(  c^{-4}\alpha_{T}\right)  \geq8$ and $v_{2}\!\left(
c^{-6}\beta_{T}\right)  \geq12$. Note that $c^{-4}\alpha_{T}$ is even if and
only if $a$ is even. If $a$ is even with $v_{2}\!\left(  a\right)  \neq2$,
then $v_{2}\!\left(  c^{-4}\alpha_{T}\right)  \leq6$ which is a contradiction.
But if $v_{2}\!\left(  a\right)  =2$, then $v_{2}\!\left(  c^{-6}\beta
_{T}\right)  \leq8$ which shows that $4$ does not divide $w$. Consequently,
$w$ divides $2$ as claimed.

Now suppose $w=2$. Then $v_{2}\!\left(  c^{-4}\alpha_{T}\right)  \geq
4,\ v_{2}\!\left(  c^{-6}\beta_{T}\right)  \geq6,$ and $v_{2}\!\left(
c^{-12}\gamma_{T}\right)  \geq12$. Note that $c^{-4}\alpha_{T}$ is even if and
only if $a$ is even.

\textbf{Case 1.} Suppose $d$ is even. Then $v_{2}\!\left(  c^{-12}\gamma
_{T}\right)  \geq12$ is equivalent to $2v_{2}\!\left(  c\right)
+v_{2}\!\left(  c^{2}d+16b\right)  \geq5$. This inequality holds if and only
if $c$ is even. Then $w^{-6}c^{-6}\beta_{T}$ is even and $w^{-4}c^{-4}%
\alpha_{T}\equiv4b^{2}\ \operatorname{mod}8$. In particular, $v_{2}\!\left(
w^{-4}c^{-4}\alpha_{T}\right)  =2$, which is our desired contradiction since
Corollary \ref{CorKraus} implies that $w^{-12}c^{-12}\gamma_{T}$ is not the
minimal discriminant of $E_{T}$.

\textbf{Case 2.} Suppose $d$ is odd. Then $v_{2}\!\left(  c^{-12}\gamma
_{T}\right)  \geq12$ is equivalent to $\left(  i\right)  \ v_{2}\!\left(
a\right)  =4$ with $v_{2}\!\left(  c^{2}d+16b\right)  \geq8$ or $\left(
ii\right)  $ $v_{2}\!\left(  a\right)  \geq8$ is even. We consider these cases separately.

\qquad\textbf{Subcase 2a.} Suppose $v_{2}\!\left(  a\right)  =4$ with
$v_{2}\!\left(  c^{2}d+16b\right)  \geq8$. In particular, $\frac{c^{2}d}%
{16}+b\equiv0\ \operatorname{mod}4$. But this is equivalent to $bd\equiv
3\ \operatorname{mod}4$ since $\frac{c^{2}}{16}$ is an odd square. Then
\begin{equation}
w^{-6}c^{-6}\beta_{T}=d^{3}\left(  \frac{c^{2}d}{8}+b\right)  \left(
\frac{-c^{4}d^{2}}{8}-2c^{2}db+b^{2}\right)  \label{c6forC4}%
\end{equation}
implies that $w^{-6}c^{-6}\beta_{T}\equiv2+b^{3}d^{3}\ \operatorname{mod}4$,
which is $1\ \operatorname{mod}4$. This is a contradiction since Corollary
\ref{CorKraus} implies that $w^{-12}c^{-12}\gamma_{T}$ is not the minimal
discriminant of $E_{T}$.

\qquad\textbf{Subcase 2b.} Suppose $v_{2}\!\left(  a\right)  \geq8$ is even.
Then $d$ is odd and from (\ref{c6forC4}), we deduce that $w^{-6}c^{-6}%
\beta_{T}\equiv b^{3}d^{3}\ \operatorname{mod}4$. Moreover, $v_{3}\!\left(
w^{-6}c^{-6}\beta_{T}\right)  =v_{3}\!\left(  c^{-6}\beta_{T}\right)  \neq2$
by Theorem \ref{kraus} since $E_{T}^{\prime}$ is given by an integral
Weierstrass model. We conclude by Theorem \ref{kraus} that $w^{-12}%
c^{-12}\gamma_{T}$ is the minimal discriminant of $E_{T}$ if and only if
$v_{2}\!\left(  a\right)  \geq8$ is even and $bd\equiv3\ \operatorname{mod}4$.
Equivalently, $u_{T}^{-12}\gamma_{T}$ with $u_{T}=2c$ is the minimal
discriminant of $E_{T}$.

Since $E_{T}^{\prime}$ is given by an integral Weierstrass model, we conclude
by the above that $u_{T}^{-12}\gamma_{T}$ with $u_{T}=c$ is the minimal
discriminant of $E_{T}$ if and only if it is not the case that $v_{2}\!\left(
a\right)  \geq8$ is even with $bd\equiv3\ \operatorname{mod}4$.
\end{proof}

\subsection[Proof of Theorem \ref{semistablecondthm} for $T=C_{6}$]{Proof of
Theorem \ref{semistablecondthm} for \bm{$T=C_{6}$}}

\begin{claim}
[Theorem \ref{semistablecondthm} for \bm{$T=C_{6}$}.]Let $E_{T}=E_{C_{6}%
}\!\left(  a,b\right)  $ where $a$ and $b$ are relatively prime integers with
$a$ positive. Then the minimal discriminant of $E_{T}$ is $u_{T}^{-12}%
\gamma_{T}$ with $u_{T}\in\left\{  1,2\right\}  $. Moreover, $u_{T}=2$ if and
only if $v_{2}\!\left(  a+b\right)  \geq3$.
\end{claim}

\begin{proof}
First, suppose $E_{T}$ has $j$-invariant $0$ or $1728$. By Lemma
\ref{sslemc4j0}, $\left(  a,b\right)  =\left(  3,-1\right)  $. Then
$\gamma_{T}=-2^{4}3^{3}$ and therefore it is the minimal discriminant of
$E_{T}$.

Next, suppose the $j$-invariant of $E_{T}$ is not equal to $0$ or $1728$.
Since $E_{T}$ is given by an integral Weierstrass model, we have by Lemma
\ref{silverconverse} that there is a unique positive integer $u_{T}$ such that
$\Delta_{E_{T}}^{\text{min}}=u_{T}^{-12}\gamma_{T}$. In particular, $u_{T}%
^{6}$ divides $\gcd\!\left(  \beta_{T},\gamma_{T}\right)  $. By Lemma
\ref{polynomials}, $\gcd\!\left(  \beta_{T},\gamma_{T}\right)  $ divides
$2^{9}3^{3}$ since $\gcd\!\left(  a,b\right)  =1$. It follows that $u_{T}$
divides $2$.

Suppose $u_{T}=2$. It is then verified that $\alpha_{T}\equiv\left(
a+b\right)  ^{4}\ \operatorname{mod}2$ and thus $a+b$ is even. Since $a$ and
$b$ are relatively prime, we deduce that $a$ and $b$ are both odd. Moreover,%
\[
12\leq v_{2}\!\left(  \gamma_{T}\right)  =v_{2}\!\left(  a+9b\right)
+3v_{2}\!\left(  a+b\right)  .
\]
If $v_{2}\!\left(  a+b\right)  \leq2$, then $v_{2}\!\left(  a+9b\right)
\leq2$ since $a+9b\equiv a+b\ \operatorname{mod}8$. But this is a
contradiction since $v_{2}\!\left(  \gamma_{T}\right)  \leq8$.

Now suppose $v_{2}\!\left(  a+b\right)  \geq3$ and consider the admissible
change of variables $x\longmapsto u_{T}^{2}x$ and $y\longmapsto u_{T}^{3}y$
from $E_{T}$ onto%
\[
E_{T}^{\prime}:y^{2}+\frac{a-b}{2}xy-\frac{ab\left(  a+b\right)  }{8}%
y=x^{3}-\frac{b\left(  a+b\right)  }{4}x^{2}.
\]
Since $v_{2}\!\left(  a+b\right)  \geq3$, it follows that $E_{T}^{\prime}$ is
given by an integral Weierstrass model with discriminant $2^{-12}\gamma_{T}$.
This shows that $u_{T}=2$ if and only if $v_{2}\!\left(  a+b\right)  \geq3$.
Consequently, $E_{T}$ is a global minimal model for $E_{T}$ if and only if
$v_{2}\!\left(  a+b\right)  <3$ since $u_{T}$ divides $2$.
\end{proof}

\subsection[Proof of Theorem \ref{semistablecondthm} for $T=C_{8}$]{Proof of
Theorem \ref{semistablecondthm} for \bm{$T=C_{8}$}}

\begin{claim}
[Theorem \ref{semistablecondthm} for \bm{$T=C_{8}$}.]Let $E_{T}=E_{C_{8}%
}\!\left(  a,b\right)  $ where $a$ and $b$ are relatively prime integers with
$a$ positive. Then the minimal discriminant of $E_{T}$ is $u_{T}^{-12}%
\gamma_{T}$ with $u_{T}\in\left\{  1,2\right\}  $. Moreover, $u_{T}=2$ if and
only if $v_{2}\!\left(  a\right)  =1$.
\end{claim}

\begin{proof}
By Lemma \ref{silverconverse}, there is a unique positive integer $u_{T}$ such
that $\Delta_{E_{T}}^{\text{min}}=u_{T}^{-12}\gamma_{T}$ since $E_{T}$ is
given by an integral Weierstrass model. In particular, $u_{T}^{6}$ divides
$\gcd\!\left(  \beta_{T},\gamma_{T}\right)  $ and by Lemma~\ref{polynomials},
$\gcd\!\left(  \beta_{T},\gamma_{T}\right)  $ divides $2^{9}$. It follows that
$u_{T}$ divides $2$.

Suppose $u_{T}=2$. Then $v_{2}\!\left(  \alpha_{T}\right)  \geq4$. Since $\alpha_{T}\equiv a^{8}\ \operatorname{mod}2$, we have that $a$ is even. We now proceed by cases and observe that
\begin{equation}
12\leq v_{2}\!\left(  \gamma_{T}\right)  =2v_{2}\!\left(  a\right)
+4v_{2}\!\left(  a-2b\right)  +v_{2}\!\left(  a^{2}-8ab+8b^{2}\right)  .
\label{ch:ss:c8in1}%
\end{equation}

\textbf{Case 1.} Suppose $v_{2}\!\left(  a\right)  \geq2$. Then $v_{2}%
\!\left(  a^{2}-8ab+8b^{2}\right)  =3$ and by (\ref{ch:ss:c8in1}), we deduce
that $v_{2}\!\left(  a\right)  \geq3$. By inspection, $2^{-6}\beta_{T}\equiv
b^{12}\ \operatorname{mod}4=1\ \operatorname{mod}4$ since $b$ is odd. By
Corollary \ref{CorKraus}, we have that $u_{T}^{-12}\gamma_{T}$ is not the
minimal discriminant of $E_{T}$, which is a contradiction.

\textbf{Case 2.} Suppose $v_{2}\!\left(  a\right)  =1$. Then the admissible
change of variables $x\longmapsto u_{T}^{2}x$ and $y\longmapsto u_{T}^{3}y$
from $E_{T}$ onto%
\[
E_{T}^{\prime}:y^{2}-\frac{a^{2}-4ab+2b^{2}}{2}xy-\frac{ab^{3}\left(
a-2b\right)  \left(  a-b\right)  }{8}y=x^{3}-\frac{\left(  a-2b\right)
\left(  a-b\right)  b^{2}}{4}x^{2}%
\]
gives an integral Weierstrass model since $v_{2}\!\left(  a-2b\right)  \geq2$.
This shows that $u_{T}=2$ if and only if $v_{2}\!\left(  a\right)  =1$.
Consequently, $E_{T}$ is a global minimal model for $E_{T}$ if and only if
$v_{2}\!\left(  a\right)  \neq1$.
\end{proof}

\subsection[Proof of Theorem \ref{semistablecondthm} for $T=C_{10}$]{Proof of
Theorem \ref{semistablecondthm} for \bm{$T=C_{10}$}}

\begin{claim}
[Theorem \ref{semistablecondthm} for \bm{$T=C_{10}$}.]Let $E_{T}=E_{C_{10}%
}\!\left(  a,b\right)  $ where $a$ and $b$ are relatively prime integers with
$a$ positive. Then the minimal discriminant of $E_{T}$ is $u_{T}^{-12}%
\gamma_{T}$ with $u_{T}\in\left\{  1,2\right\}  $. Moreover, $u_{T}=2$ if and
only if $a$ is even.
\end{claim}

\begin{proof}
Since $E_{T}$ is given by an integral Weierstrass model, Lemma
\ref{silverconverse} implies that there is a unique positive integer $u_{T}$
such that $\Delta_{E_{T}}^{\text{min}}=u_{T}^{-12}\gamma_{T}$. In particular,
$u_{T}^{6}$ divides $\gcd\!\left(  \beta_{T},\gamma_{T}\right)  $ and by Lemma
\ref{polynomials}, $\gcd\!\left(  \beta_{T},\gamma_{T}\right)  $ divides
$2^{8}5$. Therefore $u_{T}$ divides $2$.

Suppose $u_{T}=2$. Then $v_{2}\!\left(  \alpha_{T}\right)  \geq4$. Since
$\alpha_{T}\equiv a^{12}\ \operatorname{mod}2$, we deduce that $a$ is even.
Next, consider the admissible change of variables $x\longmapsto u_{T}^{2}x$
and $y\longmapsto u_{T}^{3}y$ from $E_{T}$ onto%
\begin{align*}
E_{T}^{\prime}:y^{2}  &  +\frac{a^{3}-2a^{2}b-2ab^{2}+2b^{3}}{2}xy-\frac
{a^{2}b^{3}\left(  a-2b\right)  \left(  a-b\right)  \left(  a^{2}%
-3ab+b^{2}\right)  }{8}\\
&  =x^{3}-\frac{a\left(  a-2b\right)  \left(  a-b\right)  b^{3}}{4}x^{2}.
\end{align*}
It follows that $E_{T}^{\prime}$ is given by 
an integral Weierstrass model since $a$
is even. Therefore $E_{T}^{\prime}$ is a global minimal model for $E_{T}$ if
and only if $a$ is even. Lastly, if $a$ is odd, then $E_{T}$ is a global
minimal model.
\end{proof}

\subsection[Proof of Theorem \ref{semistablecondthm} for $T=C_{12}$]{Proof of
Theorem \ref{semistablecondthm} for \bm{$T=C_{12}$}}

\begin{claim}
[Theorem \ref{semistablecondthm} for \bm{$T=C_{12}$}.]Let $E_{T}=E_{C_{12}%
}\!\left(  a,b\right)  $ where $a$ and $b$ are relatively prime integers with
$a$ positive. Then the minimal discriminant of $E_{T}$ is $u_{T}^{-12}%
\gamma_{T}$ with $u_{T}\in\left\{  1,2\right\}  $. Moreover, $u_{T}=2$ if and
only if $a$ is even.
\end{claim}

\begin{proof}
By Lemma \ref{silverconverse}, there is a unique positive integer $u_{T}$ such
that $\Delta_{E_{T}}^{\text{min}}=u_{T}^{-12}\gamma_{T}$ since $E_{T}$ is
given by an integral Weierstrass model. In particular, $u_{T}^{6}$ divides
$\gcd\!\left(  \beta_{T},\gamma_{T}\right)  $. Then $\gcd\!\left(  \beta
_{T},\gamma_{T}\right)  $ divides $2^{9}3^{3}$ by Lemma \ref{polynomials}.
Thus $u_{T}$ divides $2$.

Suppose $u_{T}=2$. Then $v_{2}\!\left(  \alpha_{T}\right)  \geq4$ and it
follows that $a$ is even since $\alpha_{T}\equiv a^{16}\ \operatorname{mod}2$.
The admissible change of variables $x\longmapsto u_{T}^{2}x$ and $y\longmapsto
u_{T}^{3}y$ from $E_{T}$ onto $E_{T}^{\prime}:y^{2}+a_{1}xy+a_{3}y=x^{3}%
+a_{2}x^{2}$ where%
\begin{align*}
a_{1}  &  =-\frac{1}{2}\left(  a^{4}-2a^{3}b-2a^{2}b^{2}+8ab^{3}%
-6b^{4}\right)  ,\\
a_{2}  &  =\frac{1}{4}b\left(  a-2b\right)  \left(  a-b\right)  ^{2}\left(
a^{2}-3ab+3b^{2}\right)  \left(  a^{2}-2ab+2b^{2}\right)  ,\\
a_{3}  &  =-\frac{1}{8}ab\left(  a-2b\right)  \left(  a-b\right)  ^{5}\left(
a^{2}-3ab+3b^{2}\right)  \left(  a^{2}-2ab+2b^{2}\right)  .
\end{align*}
Since $a$ is even, $E_{T}^{\prime}$ is given by an integral Weierstrass model.
Therefore $E_{T}^{\prime}$ is a global minimal model for $E_{T}$ if and only
if $a$ is even. Lastly, if $a$ is odd, then $E_{T}$ is a global minimal model.
\end{proof}

\subsection[Proof of Theorem \ref{semistablecondthm} for $T=C_{2}\times C_{2}%
$]{Proof of Theorem \ref{semistablecondthm} for \bm{$T=C_{2}\times C_{2}$}}

\begin{claim}
[Theorem \ref{semistablecondthm} for \bm{$T=C_{2}\times C_{2}$}.]Let
$E_{T}=E_{C_{2}\times C_{2}}\!\left(  a,b,d\right)  $ with $\gcd\!\left(
a,b\right)  =1$, $d$ squarefree, and $a$ even. Then the minimal discriminant
of $E_{T}$ is $u_{T}^{-12}\gamma_{T}$ with $u_{T}\in\left\{  1,2\right\}  $.
Moreover, $u_{T}=2$ if and only if $v_{2}\!\left(  a\right)  \geq4$ and
$bd\equiv1\ \operatorname{mod}4$.
\end{claim}

\begin{proof}
Since $E_{T}$ is given by an integral Weierstrass model, Lemma
\ref{silverconverse} implies that there is a unique positive integer $u_{T}$
such that $\Delta_{E_{T}}^{\text{min}}=u_{T}^{-12}\gamma_{T}$. Moreover,%
\[
\alpha_{T}=16d^{2}\left(  a^{2}-ab+b^{2}\right)  ,\quad\beta_{T}%
=-32d^{3}\left(  a+b\right)  \left(  a-2b\right)  \left(  2a-b\right)
,\quad\gamma_{T}=16a^{2}b^{2}d^{6}\left(  a-b\right)  ^{2}.
\]
Now suppose $E_{T}$ has $j$-invariant equal to $0$ or $1728$. By Lemma
\ref{sselmmac2c2result}, we may assume that $\left(  a,b,d\right)  =\left(
2,1,d\right)  $. Since $\gamma_{T}=64d^{6}$ and $d$ is squarefree, it follows
that $u_{T}$ divides $2$. If $u_{T}=2$, then $v_{2}\!\left(  u_{T}^{-4}%
\alpha_{T}\right)  =2$ and $u_{T}^{-6}\beta_{T}=0$. This is a contradiction
since Corollary \ref{CorKraus} implies that $u_{T}^{-12}\gamma_{T}$ is not the
minimal discriminant of $E_{T}$. Thus $u_{T}=1$ which concludes this case.

Next, suppose the $j$-invariant of $E_{T}$ is not $0$ or $1728$. In
particular, $u_{T}^{6}$ divides $\gcd\!\left(  \beta_{T},\gamma_{T}\right)  $
and by Lemma \ref{polynomials}, $\gcd\!\left(  \beta_{T},\gamma_{T}\right)  $
divides $2^{7}d^{8}$. Therefore $u_{T}$ divides $2d$ since $d$ is a squarefree
integer. We claim that $v_{p}\!\left(  u_{T}\right)  =0$ for all odd primes
$p$. Towards a contradiction, suppose an odd prime $p$ divides $u_{T}$. Then $v_{p}\!\left(  \gamma_{T}\right)  \geq12$ which implies that
$v_{p}\!\left(  ab\left(  a-b\right)  \right)  \geq3$. Since $\gcd\!\left(
a,b\right)  =1$, it follows that $p$ divides exactly one of $a,b,$ or $a-b$.
If $p$ divides one of $a$ or $b$, then $p$ does not divide $a^{2}-ab+b^{2}$
which contradicts the assumption that $p^{4}$ divides $\alpha_{T}$. Therefore
$p$ divides $a-b$ and $a^{2}-ab+b^{2}$. But then $p$ divides $a^{2}%
-ab+b^{2}-\left(  a-b\right)  ^{2}=ab$, which is a contradiction. Hence
$v_{p}\!\left(  u_{T}\right)  =0$ for all odd primes $p$.

It follows that $u_{T}$ divides $4$. Now observe that if $u_{T}=4$, then $d$
is even and $v_{2}\!\left(  \alpha_{T}\right)  \geq8$. But this is a
contradiction since $v_{2}\!\left(  \alpha_{T}\right)  =6$. Hence $u_{T}$
divides $2$.

So suppose $u_{T}=2$ and note that the assumption that $a$ is even implies
that%
\begin{equation}
12\leq v_{2}\!\left(  \gamma_{T}\right)  =4+6v_{2}\!\left(  d\right)
+2v_{2}\!\left(  a\right)  . \label{ch:ss:c2c2ine2}%
\end{equation}

\textbf{Case 1.} Suppose $d$ is even. Then, by inspection, $u_{T}^{-6}%
\beta_{T}$ is even and $v_{2}\!\left(  u_{T}^{-4}\alpha_{T}\right)  =2$. This
is a contradiction, since Corollary \ref{CorKraus} implies that $u_{T}%
^{-12}\gamma_{T}$ is not the minimal discriminant of $E_{T}$.

\textbf{Case 2.} Suppose $d$ is odd. Then inequality (\ref{ch:ss:c2c2ine2})
holds if and only if $v_{2}\!\left(  a\right)  \geq4$. Next, observe that
Theorem \ref{kraus} implies that $v_{3}\!\left(  u_{T}^{-6}\beta_{T}\right)
=v_{3}\!\left(  \beta_{T}\right)  \neq2$ since $E_{T}$ is given by an integral
Weierstrass model. Since $u_{T}^{-6}\beta_{T}\equiv-b^{3}d^{3}%
\ \operatorname{mod}4$, we have by Theorem \ref{kraus} that $u_{T}^{-12}%
\gamma_{T}$ is the minimal discriminant of $E_{T}$ if and only if
$v_{2}\!\left(  a\right)  \geq4$ and $bd\equiv1\ \operatorname{mod}4$.
Consequently, $\gamma_{T}$ is the minimal discriminant of $E_T$ if and only if $v_{2}\!\left(  a\right)  \leq3$ or $bd \not \equiv
1\ \operatorname{mod}4$.
\end{proof}

\subsection[Proof of Theorem \ref{semistablecondthm} for $T=C_{2}\times C_{4}%
$]{Proof of Theorem \ref{semistablecondthm} for \bm{$T=C_{2}\times C_{4}$}}

\begin{claim}
[Theorem \ref{semistablecondthm} for \bm{$T=C_{2}\times C_{4}$}.]Let
$E_{T}=E_{C_{2}\times C_{4}}\!\left(  a,b\right)  $ where $a$ and $b$ are
relatively prime integers with $a$ positive. Then the minimal discriminant of
$E_{T}$ is $u_{T}^{-12}\gamma_{T}$ with $u_{T}\in\left\{  1,2,4\right\}  $.

$\left(  1\right)  $ $u_{T}=4$ if and only if $v_{2}\!\left(  a\right)  =2$
and $v_{2}\!\left(  a+4b\right)  \geq4$;

$\left(  2\right)  \ u_{T}=2$ if and only if $v_{2}\!\left(  a\right)  \geq2$
with $v_{2}\!\left(  a+4b\right)  \leq3$;

$\left(  3\right)  \ u_{T}=1$ if and only if $v_{2}\!\left(  a\right)  \leq1$.
\end{claim}

\begin{proof}
By Lemma \ref{silverconverse}, there is a unique positive integer $u_{T}$ such
that $\Delta_{E_{T}}^{\text{min}}=u_{T}^{-12}\gamma_{T}$ since $E_{T}$ is
given by an integral Weierstrass model. In particular, $u_{T}^{4}$ divides
$\gcd\!\left(  \alpha_{T},\beta_{T}\right)  $ and by Lemma~\ref{polynomials},
$\gcd\!\left(  \alpha_{T},\beta_{T}\right)  $ divides $2^{14}3^{2}$. Therefore
$u_{T}$ divides $8$. Note that
\[
v_{2}\!\left(  \gamma_{T}\right)  =2v_{2}\!\left(  a\right) +4v_{2}\!\left(  b\right)  +2v_{2}\!\left(
a+8b\right)  +4v_{2}\!\left(  a+4b\right)
\]
and%
\begin{equation}
v_{2}\!\left(  \alpha_{T}\right)  =v_{2}\!\left(  a^{4}+16a^{3}b+80a^{2}%
b^{2}+128ab^{3}+256b^{4}\right)  =\left\{
\begin{array}
[c]{ll}%
0 & \text{if }v_{2}\!\left(  a\right)  =0,\\
4 & \text{if }v_{2}\!\left(  a\right)  =1,\\
8 & \text{if }v_{2}\!\left(  a\right)  \geq2.
\end{array}
\right.  \label{ch:ss:c2c4ine}%
\end{equation}
Consequently, $u_{T}$ divides $4$. For the cases below, we will consider the
admissible change of variables $x\longmapsto u_{T}^{2}x$ and $y\longmapsto
u_{T}^{3}y$ from $E_{T}$ onto%
\begin{equation}
E_{T,u_{T}}^{\prime}:y^{2}+\frac{a}{u_{T}}xy-\frac{ab\left(  a+4b\right)
}{u_{T}^{3}}y=x^{3}-\frac{b\left(  a+4b\right)  }{u_{T}^{2}}x^{2}.
\label{WeierforC2C4}%
\end{equation}

\textbf{Case 1.} Suppose $u_{T}=4$. Then $v_{2}\!\left(  \alpha_{T}\right)
\geq8$ and $v_{2}\!\left(  \gamma_{T}\right)  \geq24$. By (\ref{ch:ss:c2c4ine}%
), $v_{2}\!\left(  \alpha_{T}\right)  \geq8$ if and only if $v_{2}\!\left(
a\right)  \geq2$. We proceed by cases.

\qquad\textbf{Subcase 1a.} Suppose $v_{2}\!\left(  a\right)  \geq4$. Then
$v_{2}\!\left(  \gamma_{T}\right)  =2v_{2}\!\left(  a\right)  +14$ and
$v_{2}\!\left(  \gamma_{T}\right)  \geq24$ holds if and only if $v_{2}%
\!\left(  a\right)  \geq5$. By inspection, $u_{T}^{-6}\beta_{T}\equiv
b^{6}\ \operatorname{mod}4$ which is $1\ \operatorname{mod}4$. This is a
contradiction since Corollary \ref{CorKraus} implies that $u_{T}^{-12}%
\gamma_{T}$ is not the minimal discriminant of $E_{T}$.

\qquad\textbf{Subcase 1b.} Suppose $v_{2}\!\left(  a\right)  =3$. Then
$v_{2}\!\left(  \gamma_{T}\right)  =20+2v_{2}\!\left(  \frac{a}{8}+b\right)  $
and therefore $v_{2}\!\left(  \gamma_{T}\right)  \geq24$ holds if and only if
$v_{2}\!\left(  \frac{a}{8}+b\right)  \geq2$. It is then verified that
$u_{T}^{-6}\beta_{T}\equiv3+\frac{ab}{4}\ \operatorname{mod}4$, which is
$1\ \operatorname{mod}4$. By Corollary \ref{CorKraus}, $u_{T}^{-12}\gamma_{T}$
is not the minimal discriminant of $E_{T}$ which is our desired contradiction.

\qquad\textbf{Subcase 1c.} Suppose $v_{2}\!\left(  a\right)  =2$. Then
$v_{2}\!\left(  \gamma_{T}\right)  =8+4v_{2}\!\left(  a+4b\right)  $ and so
$v_{2}\!\left(  \gamma_{T}\right)  \geq24$ holds if and only if $v_{2}%
\!\left(  a+4b\right)  \geq4$. Under this assumption, we observe by
(\ref{WeierforC2C4}) that $E_{T,u_{T}}^{\prime}$ is given by an integral
Weierstrass model and thus $u_{T}^{-12}\gamma_{T}$ is the minimal discriminant
of $E_{T}$ if and only if $v_{2}\!\left(  a\right)  =2$ with $v_{2}\!\left(
a+4b\right)  \geq4$, which concludes the proof of $\left(  1\right)  $.

\textbf{Case 2.} Suppose $u_{T}=2$. Then $v_{2}\!\left(  \alpha_{T}\right)
\geq4$ and $v_{2}\!\left(  \gamma_{T}\right)  \geq12$. By Case 1 and
(\ref{ch:ss:c2c4ine}), we have that $v_{2}\!\left(  a\right)  \geq1$ such that
$v_{2}\!\left(  a+4b\right)  <4$. We proceed by cases.

\qquad\textbf{Subcase 2a.} Suppose $v_{2}\!\left(  a\right)  =1$. Then
$v_{2}\!\left(  \gamma_{T}\right)  =8$, which is a contradiction.

\qquad\textbf{Subcase 2b.} Suppose $v_{2}\!\left(  a\right)  \geq2$ with
$v_{2}\!\left(  a+4b\right)  <4$. By (\ref{WeierforC2C4}) and Case 1, we have
that $u_{T}^{-12}\gamma_{T}$ is the minimal discriminant of $E_{T}$ since
$E_{T,u_{T}}^{\prime}$ is given by an integral Weierstrass model. This
concludes the proof of $\left(  2\right)  $.

Lastly, by Cases 1 and 2 and the fact that $E_{T}$ is given by an integral
Weierstrass model, we conclude that $E_{T}$ is a global minimal model if and
only if $v_{2}\!\left(  a\right)  \leq1$.
\end{proof}

\subsection[Proof of Theorem \ref{semistablecondthm} for $T=C_{2}\times C_{6}%
$]{Proof of Theorem \ref{semistablecondthm} for \bm{$T=C_{2}\times C_{6}$}}

\begin{claim}
[Theorem \ref{semistablecondthm} for \bm{$T=C_{2}\times C_{6}$}.]Let
$E_{T}=E_{C_{2}\times C_{6}}\!\left(  a,b\right)  $ where $a$ and $b$ are
relatively prime integers with $a$ positive. Then the minimal discriminant of
$E_{T}$ is $u_{T}^{-12}\gamma_{T}$ with $u_{T}\in\left\{  1,4,16\right\}  $. Moreover,

$\left(  1\right)  $ $u_{T}=16$ if and only if $v_{2}\!\left(  a+b\right)  =1$;

$\left(  2\right)  $ $u_{T}=4$ if and only if $v_{2}\!\left(  a+b\right)
\geq2$;

$\left(  3\right)  $ $u_{T}=1$ if and only if $v_{2}\!\left(  a+b\right)  =0$.
\end{claim}

\begin{proof}
By Lemma \ref{silverconverse}, there is a unique positive integer $u_{T}$ such
that $\Delta_{E_{T}}^{\text{min}}=u_{T}^{-12}\gamma_{T}$ since $E_{T}$ is
given by an integral Weierstrass model. In particular, $u_{T}^{4}$ divides
$\gcd\!\left(  \alpha_{T},\beta_{T}\right)  $ and by Lemma~\ref{polynomials},
$\gcd\!\left(  \alpha_{T},\beta_{T}\right)  $ divides $2^{31}3^{4}$ and
$\gcd\!\left(  \beta_{T},\gamma_{T}\right)  $ divides $2^{56}3^{3}$. Therefore
$u_{T}$ divides $2^{7}$. Since $\alpha_{T}\equiv a^{8}+b^{8}%
\ \operatorname{mod}2$, we deduce that $\alpha_{T}$ is even if and only if
$a+b$ is even. In particular, $a$ and $b$ are both odd since $a$ and $b$ are
relatively prime. Hence $b=2k-a$ for some integer $k$. In fact, $k$ is odd if
and only if $v_{2}\!\left(  a+b\right)  =1$. With this substitution, it is
checked that%
\begin{equation}
v_{2}\!\left(  \alpha_{T}\right)  =\left\{
\begin{array}
[c]{ll}%
0 & \text{if }v_{2}\!\left(  a+b\right)  =0,\\
16 & \text{if }v_{2}\!\left(  a+b\right)  =1,\\
8 & \text{if }v_{2}\!\left(  a+b\right)  \geq2.
\end{array}
\right.  \label{c2c6inealpt}%
\end{equation}
It follows that $u_{T}$ divides $16$. For the cases below, we will consider
the admissible change of variables $x\longmapsto u_{T}^{2}x$ and $y\longmapsto
u_{T}^{3}y$ from $E_{T}$ onto $E_{T,u_{T}}^{\prime}:y^{2}-\frac{a_{1}}{u_{T}}xy+\frac{2a_{3}}{u_{T}^{3}%
}y=x^{3}+\frac{2a_{2}}{u_{T}^{2}}x^{2}$ where
\begin{align}
% E_{T,u_{T}}^{\prime}  &  :y^{2}-\frac{a_{1}}{u_{T}}xy+\frac{2a_{3}}{u_{T}^{3}%
% }y=x^{3}+\frac{2a_{2}}{u_{T}^{2}}x^{2}\text{ where}\\
a_{1}  &  =20a^{2}-\left(  a+b\right)  ^{2},\qquad a_{2}=a\left(  b-a\right)
^{2}\left(  b-5a\right)  , \label{ch:ss:c2c6Et}\\
a_{3}  &  =a\left(  b-3a\right)  \left(  3a+b\right)  \left(  b-a\right)
^{2}\left(  b-5a\right)  .\nonumber
\end{align}

\textbf{Case 1.} Suppose $u_{T}=16$. By (\ref{c2c6inealpt}), $v_{2}\!\left(
a+b\right)  =1$. Note that this implies that $v_{2}\!\left(  b-a\right)
\geq2$. We claim that $E_{T,u_{T}}^{\prime}$ is given by an integral
Weierstrass model. To this end, observe that $v_{2}\!\left(  a_{1}\right)
\geq4$ since $\frac{a_{1}}{4}\equiv0\ \operatorname{mod}4$. For $a_{2}$ and
$a_{3}$, set $b=2k-a$ for some odd integer $k$. Then $\left(  b-a\right)
^{2}\left(  b-5a\right)  =8\left(  a-k\right)  ^{2}\left(  k-3a\right)  $.
Since $\left(  a-k\right)  ^{2}\left(  k-3a\right)  \equiv
0\ \operatorname{mod}16$, we deduce that $v_{2}(  \left(  b-a\right)
^{2}\left(  b-5a\right)  )  \geq7$. Thus $\frac{2a_{2}}{u_{T}^{2}}$ is
an integer.

Next, observe that $b-3a\equiv b+a\ \operatorname{mod}4$ and thus
$v_{2}\!\left(  b-3a\right)  =1$. Then $\left(  b-3a\right)  \left(
3a+b\right)  $ is divisible by $8$ as it is a difference of odd squares.
Therefore $v_{2} \!\left(  3a+b\right)  \geq2$. Now suppose $v_{2}\!\left(
3a+b\right)  =2$ so that $b=4l-3a$ for some odd integer $l$. Then $\left(
b-a\right)  ^{2}\left(  b-5a\right)  =-64\left(  a-l\right)  ^{2}\left(
2a-l\right)  $ and thus $v_{2}(  \left(  b-a\right)  ^{2}\left(
b-5a\right)  )  \geq8$. Next, suppose $v_{2}\!\left(
3a+b\right)  >2$. A similar argument to the above gives that  $v_{2}(  \left(  b-a\right)  ^{2}\left(
b-5a\right)  )  \geq7$. It follows in both cases that 
$v_{2}\!\left(  a_{3}\right)
\geq11$ and hence $\frac{2a_{3}}{u_{T}^{3}}$ is an integer which shows that
$E_{T,u_{T}}^{\prime}$ is given by an integral Weierstrass model.
This shows (1).

%concludes the proof of $\left(  1\right)  $.

\textbf{Case 2.} By Case 1 and (\ref{c2c6inealpt}), we have that if
$v_{2}\!\left(  a+b\right)  \neq1$, then $u_{T}$ divides $4$. So suppose
$u_{T}=4$. By (\ref{c2c6inealpt}), $v_{2}\!\left(  a+b\right)  \geq2$. Then
$v_{2}\!\left(  b-a\right)  =1$ and by inspection we have that $E_{T,u_{T}%
}^{\prime}$ is given by an integral Weierstrass model. In particular, the
minimal discriminant of $E_{T}$ is $u_{T}^{-12}\gamma_{T}$ if and only only if
$v_{2}\!\left(  a+b\right)  \geq2$.

It remains to consider the case when $v_{2}\!\left(  a+b\right)  =0$. By the
cases above, (\ref{c2c6inealpt}), and the fact that $E_{T}$ is given by an
integral Weierstrass model, we conclude that the minimal discriminant of
$E_{T}$ is $\gamma_{T}$ if and only if $v_{2}\!\left(  a+b\right)  =0$.
\end{proof}

\subsection[Proof of Theorem \ref{semistablecondthm} for $T=C_{2}\times C_{8}%
$]{Proof of Theorem \ref{semistablecondthm} for \bm{$T=C_{2}\times C_{8}$}}

\begin{claim}
[Theorem \ref{semistablecondthm} for \bm{$T=C_{2}\times C_{8}$}.]Let
$E_{T}=E_{C_{2}\times C_{8}}\!\left(  a,b\right)  $ where $a$ and $b$ are
relatively prime integers with $a$ positive. Then the minimal discriminant of
$E_{T}$ is $u_{T}^{-12}\gamma_{T}$ with $u_{T}\in\left\{  1,16,64\right\}  $. Moreover,

$\left(  1\right)  $ $u_{T}=64$ if and only if $v_{2}\!\left(  a\right)
\geq2$;

$\left(  2\right)  $ $u_{T}=16$ if and only if $v_{2}\!\left(  a\right)  =1$;

$\left(  3\right)  $ $u_{T}=1$ if and only if $a$ is odd.
\end{claim}

\begin{proof}
By Lemma \ref{silverconverse}, there is a unique positive integer $u_{T}$ such
that $\Delta_{E_{T}}^{\text{min}}=u_{T}^{-12}\gamma_{T}$ since $E_{T}$ is
given by an integral Weierstrass model. In particular, $u_{T}^{4}$ divides
$\gcd\!\left(  \alpha_{T},\beta_{T}\right)  $. By Lemma~\ref{polynomials},
$\gcd\!\left(  \alpha_{T},\beta_{T}\right)  $ divides $2^{49}3^{2}$ and thus
$u_{T}$ divides $2^{12}$. Next, observe that $\alpha_{T}\equiv a^{16}%
\ \operatorname{mod}2$ and therefore $\alpha_{T}$ is even if and only if $a$
is even. Now set $a=2k$ for some integer $k$. With this substitution, it is
verified that
\begin{equation}
v_{2}\!\left(  \alpha_{T}\right)  =\left\{
\begin{array}
[c]{ll}%
0 & \text{if }v_{2}\!\left(  a\right)  =0,\\
16 & \text{if }v_{2}\!\left(  a\right)  =1,\\
24 & \text{if }v_{2}\!\left(  a\right)  \geq2.
\end{array}
\right.  \label{valalpt2c2xc8}%
\end{equation}
It follows that $u_{T}$ divides $64$. For the cases below, we will consider
the admissible change of variables $x\longmapsto u_{T}^{2}x$ and $y\longmapsto
u_{T}^{3}y$ from $E_{T}$ onto%
\begin{align*}
E_{T,u_{T}}^{\prime}  &  :y^{2}-\frac{a_{1}}{u_{T}}xy+\frac{8a_{3}}{u_{T}^{3}%
}y=x^{3}-\frac{4a_{2}}{u_{T}^{2}}x^{2}\text{ where}\\
a_{1}  &  =a^{4}+8a^{3}b+24a^{2}b^{2}-64b^{4},\qquad a_{2}=ab^{2}\left(
a+2b\right)  \left(  a+4b\right)  ^{2}\left(  a^{2}+4ab+8b^{2}\right)  ,\\
a_{3}  &  =ab^{3}\left(  a+2b\right)  \left(  a+4b\right)  ^{3}\left(
a^{2}-8b^{2}\right)  \left(  a^{2}+4ab+8b^{2}\right)  .
\end{align*}

\textbf{Case 1.} Suppose $u_{T}=64$. By (\ref{valalpt2c2xc8}), $v_{2}\!\left(
a\right)  \geq2$. By inspection, the Weierstrass coefficients of $E_{T,u_{T}%
}^{\prime}$ satisfy $v_{2}\!\left(  a_{1}\right)  \geq6,\ v_{2}\!\left(
a_{2}\right)  \geq10,$ and $v_{2}\!\left(  a_{3}\right)  \geq15$.
Consequently, $E_{T,u_{T}}^{\prime}$ is given by an integral Weierstrass model
which shows that $u_{T}^{-12}\gamma_{T}$ is the minimal discriminant of
$E_{T}$ if and only if $v_{2}\!\left(  a\right)  \geq2$, which shows $\left(
1\right)  $.

\textbf{Case 2.} By Case 1 and (\ref{valalpt2c2xc8}), we have that if
$v_{2}\!\left(  a\right)  \leq1$, then $u_{T}$ divides $16$. So suppose
$u_{T}=16$. By (\ref{valalpt2c2xc8}), $v_{2}\!\left(  a\right)  =1$. By
inspection, the Weierstrass coefficients of $E_{T,u_{T}}^{\prime}$ satisfy
$v_{2}\!\left(  a_{1}\right)  \geq4,\ v_{2}\!\left(  a_{2}\right)  \geq7,$ and
$v_{2}\!\left(  a_{3}\right)  \geq10$. Consequently, $E_{T,u_{T}}^{\prime}$ is
given by an integral Weierstrass model which shows that $u_{T}^{-12}\gamma
_{T}$ is the minimal discriminant of $E_{T}$ if and only if $v_{2}\!\left(
a\right)  =1$, which shows $\left(  2\right)  $.

It remains to consider the case when $v_{2}\!\left(  a\right)  =0$. By
(\ref{valalpt2c2xc8}) and the fact that $E_{T}$ is given by an integral
Weierstrass model, we conclude that the minimal discriminant of $E_{T}$ is
$\gamma_{T}$ if and only if $v_{2}\!\left(  a\right)  =0$.
\end{proof}

\section{Global Minimal Models\label{sec:GMM}}
By Theorem \ref{semistablecondthm}, there are necessary and sufficient conditions on the parameters of $E_T$ to determine its minimal discriminant. However, the proof of this result did not exhibit a global minimal model for each of the families $E_T$. Indeed, some of the cases relied on Theorem \ref{kraus} to conclude the minimal discriminant in several cases. The following theorem gives sufficient conditions on the parameters of $E_T$ to determine a global minimal model for $E_T$.

\begin{theorem}
\label{GlobalMinModel}Let $u_{T}$ be as given in Theorem
\ref{semistablecondthm} and let $E_{T}^{\prime}$ be the elliptic curve
attained from $E_{T}$ via the admissible change of variables $x\longmapsto
u_{T}^{2}x+r_{T}$ and $y\longmapsto u_{T}^{3}y+u_{T}^{2}s_{T}x$ where
\begin{align*}
r_{T}  &  =\left\{
\begin{array}
[c]{cl}%
1 & \text{if }T=C_{2}\text{ with }v_{2}\!\left(  b\right)  \geq3\text{ and
}a\equiv7\ \operatorname{mod}8,\\
-3 & \text{if }T=C_{2}\text{ with }v_{2}\!\left(  b\right)  \geq3\text{ and
}a\equiv3\ \operatorname{mod}8,\\
0 & \text{otherwise,}%
\end{array}
\right. \\
s_{T}  &  =\left\{
\begin{array}
[c]{cl}%
0 & \text{if }\left(  i\right)  \ T\neq C_{2},C_{4},C_{2}\times C_{2},\text{
}\left(  ii\right)  \ T=C_{4}\ \text{with }u_{T}=c,\ \text{or }\\
& \hspace{1.1em} \left(  iii\right)  \text{\ }T=C_{2},C_{2}\times
C_{2}\ \text{with }u_{T}=1,\\
1 & \text{if }T=C_{2}\times C_{2}\text{ with }u_{T}=2\text{ or }T=C_{2}\text{
with }v_{2}\!\left(  b\right)  \geq3\text{ and }a\equiv3\ \operatorname{mod}%
4,\\
2 & \text{if }T=C_{2}\text{ with }u_{T}\neq1\text{ and }v_{2}\!\left(
b\right)  =1,\\
2c & \text{if }T=C_{4}\text{ with }u_{T}=2c.
\end{array}
\right.
\end{align*}
Then $E_{T}^{\prime}$ is a global minimal model for $E_{T}$.
\end{theorem}

\begin{proof}
In the proof of Theorem \ref{semistablecondthm}, it was shown that
$E_{T}^{\prime}$ is a global minimal model for $E_{T}$ if $T\not =C_{2}%
,C_{4},C_{2}\times C_{2}$.  In what follows, we consider these cases separately.

\textbf{Case 1.} Let $T=C_{2}$. If $u_{T}=1$, then $E_{T}^{\prime}=E_{T}$ is a
global minimal model. So it remains to show the cases corresponding to
$u_{T}=2,4$.

\qquad\textbf{Subcase 1a.} Suppose $v_{2}\!\left(  b\right)  \geq3$ and
$a\equiv7\ \operatorname{mod}8$. Then $u_{T}=2$ by Theorem
\ref{semistablecondthm}. Now write $b=8\hat{b}$ and $a=8k+7$ for some integers
$\hat{b}$ and $k$ and observe that%
\[
E_{T}^{\prime}:y^{2}+xy=x^{3}+\left(  4k+4\right)  x^{2}+\left(  4k^{2}%
-4\hat{b}^{2}d+9k+5\right)  x+k^{2}+2k+1-b^{2}%
\]
is an integral Weierstrass model. Thus $E_{T}^{\prime}$ is a global minimal
model for $E_{T}$.

\qquad\textbf{Subcase 1b.} Suppose $v_{2}\!\left(  b\right)  \geq3$ and
$a\equiv3\ \operatorname{mod}8$. Then $u_{T}=2$ by Theorem
\ref{semistablecondthm}. Now write $b=8\hat{b}$ and $a=8k+3$ for some integers
$\hat{b}$ and $k$ and note that%
\[
E_{T}^{\prime}:y^{2}+xy=x^{3}+\left(  4k-1\right)  x^{2}+\left(  4k^{2}%
-4b^{2}d-3k\right)  x+3b^{2}d-3k^{2}%
\]
is an integral Weierstrass model. Hence $E_{T}^{\prime}$ is a global minimal
model for $E_{T}$.

\qquad\textbf{Subcase 1c.} Suppose $v_{2}\!\left(  a\right)  =v_{2}\!\left(
b\right)  =1$ with either $\left(  i\right)  $ $v_{2}\!\left(  b^{2}%
d-a^{2}\right)  \geq8$ with $a\equiv6\ \operatorname{mod}8$ or $\left(
ii\right)  $ $4\leq v_{2}\!\left(  b^{2}d-a^{2}\right)  \leq7$. Then $u_{T}=2$
by Theorem \ref{semistablecondthm} and observe that
\[
E_{T}^{\prime}:y^{2}+2xy=x^{3}+\frac{2a-4}{4}x^{2}-\frac{b^{2}d-a^{2}}{16}x
\]
is an integral Weierstrass model. In particular, $E_{T}^{\prime}$ is a global
minimal model for $E_{T}$.

\qquad\textbf{Subcase 1d.} Suppose $v_{2}\!\left(  b^{2}d-a^{2}\right)  \geq8$
with $v_{2}\!\left(  a\right)  =v_{2}\!\left(  b\right)  =1$ and
$a\equiv2\ \operatorname{mod}8$. Then $u_{T}=4$ by Theorem
\ref{semistablecondthm} and
\[
E_{T}^{\prime}:y^{2}+xy=x^{3}+\frac{a-2}{8}x^{2}-\frac{b^{2}d-a^{2}}{256}x
\]
is an integral Weierstrass model which shows that $E_{T}^{\prime}$ is a global
minimal model for $E_{T}$. This concludes the proof for $T=C_{2}$.

\textbf{Case 2.} Let $T=C_{4}$ and let $a=c^{2}d$ for $d$ a positive
squarefree integer. By the proof of Theorem \ref{semistablecondthm}, we have
that $E_{T}^{\prime}$ is a global minimal model for $E_{T}$ if $u_{T}=c$. So
it remains to show the case when $u_{T}=2c$. By Theorem
\ref{semistablecondthm}, this is equivalent to $v_{2}\!\left(  a\right)
\geq8$ is even with $bd\equiv3\ \operatorname{mod}4$. It follows that%
\[
E_{T}^{\prime}:y^{2}+\left(  \frac{cd}{2}+1\right)  xy-\frac{bcd^{2}}%
{8}y=x^{3}-\frac{bd+cd+1}{4}x^{2}+\frac{bcd^{2}}{16}x
\]
is an integral Weierstrass model since $v_{2}\!\left(  c\right)  \geq4$ and
$bd\equiv3\ \operatorname{mod}4$. Therefore $E_{T}^{\prime}$ is a global
minimal model for $E_{T}$.

\textbf{Case 3.} Let $T=C_{2}\times C_{2}$. If $u_{T}=1$, then $E_{T}^{\prime
}=E_{T}$ is a global minimal model. So it remains to show the case
corresponding to $u_{T}=2$. By Theorem \ref{semistablecondthm}, this is
equivalent to $v_{2}\!\left(  a\right)  \geq4$ and $bd\equiv
1\ \operatorname{mod}4$. Then%
\[
E_{T}^{\prime}:y^{2}+xy=x^{3}+\frac{ad+bd-1}{4}x^{2}+\frac{abd^{2}}{16}x
\]
is an integral Weierstrass model and thus a global minimal model for $E_{T}$.
\end{proof}

\section{Necessary and Sufficient Conditions for
Semistability\label{secaddred}}

In this section, we use Theorems \ref{Thm1} and \ref{semistablecondthm} to
determine the primes at which a rational elliptic curve has additive reduction. Specifically, we find necessary and sufficient conditions on the parameters of $E_{T}$ to
determine the primes at which additive reduction occurs. Consequently, we will
be able to find necessary and sufficient conditions on the parameters of
$E_{T}$ to determine when $E_{T}$ is semistable.

\begin{theorem}
\label{semis}Let $E_{T}$ be as given in Table \ref{ta:ETmodel} and suppose
that the parameters of $E_{T}$ satisfy the conclusion of Proposition
\ref{rationalmodels}. Then $E_{C_{2}\times C_{8}}$ is semistable and for
$T\neq C_{2}\times C_{8}$, $E_{T}$ has additive reduction at a prime $p$ if
and only if $p$ is listed below and the corresponding condition on the
parameters of $E_{T}$ is satisfied.
\end{theorem}
{\renewcommand*{\arraystretch}{1.15} \begin{longtable*}{ccc}
% \caption{Semistablity of $E_{T}$}\\
\hline
$T$ & $p$ & Conditions on $a,b,d$\\
\hline
\endfirsthead
%\caption[]{\emph{continued}}\\
\hline
$T$ & $p$ & Conditions on $a,b,d$\\
\hline
\endhead
\hline
\multicolumn{3}{r}{\emph{continued on next page}}
\endfoot
\hline
\endlastfoot
	
$C_{2}$ & $\neq2$ & $v_{p}\!\left(  \gcd\!\left(  a,bd\right)  \right)  \geq
1$\\\cmidrule(lr){2-3}
& $2$ & $v_{2}\!\left(  b^{2}d-a^{2}\right)  \geq8$ with $a\equiv
6\ \operatorname{mod}8$\\\cmidrule(lr){3-3}
&  & $v_{2}\!\left(  b^{2}d-a^{2}\right)  \leq7$ with $v_{2}\!\left(
a\right)  =v_{2}\!\left(  b\right)  =1$\\\cmidrule(lr){3-3}
&  & $v_{2}\!\left(  b\right)  \geq3$ with $a\not \equiv 3\ \operatorname{mod}%
4$\\\cmidrule(lr){3-3}
& & $v_{2}\!\left(  b\right)  =1$ with
$v_{2}\!\left(  a\right)  \neq1$ \\\cmidrule(lr){3-3}
&  & $v_{2}\!\left(  b\right)  =0,2$\\\hline
$C_{3}^{0}$ & $\geq2$ & $v_{p}\!\left(  a\right)=1,2$\\\cmidrule(lr){2-3}
& $3$ & $v_{3}\!\left(  a\right)=0$\\\hline
$C_{3}$ & $\neq3$ & $v_{p}\!\left(  a\right)  \not \equiv
0\ \operatorname{mod}3$\\\cmidrule(lr){2-3}
& $3$ & $v_{3}\!\left(  a\right)  \geq1$\\\hline
$C_{4}$ & $\geq2$ & $v_{p}\!\left(  a\right)  \equiv1\ \operatorname{mod}2$\\\cmidrule(lr){2-3}
& $2$ & $v_{2}\!\left(  a\right)  =2,4,6$\\\cmidrule(lr){3-3}
& & $v_{2}\!\left(  a\right)  \geq8$ is even with $bd\equiv1\ \operatorname{mod}%
4$\\\hline
$C_{5}$ & $5$ & $v_{5}\!\left(  a+3b\right)  \geq1$\\\hline
$C_{6}$ & $2$ & $v_{2}\!\left(  a+b\right)  =1,2$\\\cmidrule(lr){2-3}
& $3$ & $v_{3}\!\left(  a\right)  \geq1$\\\hline
$C_{7}$ & $7$ & $v_{7}\!\left(  a+4b\right)  \geq1$\\\hline
$C_{8}$ & $2$ & $v_{2}\!\left(  a\right)  \geq2$\\\hline
$C_{9}$ & $3$ & $v_{3}\!\left(  a+b\right)  \geq1$\\\hline
$C_{10}$ & $5$ & $v_{5}\!\left(  a+b\right)  \geq1$\\\hline
$C_{12}$ & $3$ & $v_{3}\!\left(  a\right)  \geq1$\\\hline
$C_{2}\times C_{2}$ & $\neq2$ & $v_{p}\!\left(  d\right)  =1$\\\cmidrule(lr){2-3}
& $2$ & $v_{2}\!\left(  a\right)  =1,2,3$ or $bd\not \equiv
1\ \operatorname{mod}4$\\\hline
$C_{2}\times C_{4}$ & $2$ & $v_{2}\!\left(  a\right)  \geq1$ with $v_{2}\!\left(  a+4b\right)
\leq3$\\\hline
$C_{2}\times C_{6}$ & $3$ & $v_{3}\!\left(  b\right)  \geq1$
\end{longtable*}}%\addtocounter{table}{-1}
\begin{proof}
By Theorem \ref{semistablecondthm}, there are necessary and sufficient
conditions on the parameters of $E_{T}$ to determine its minimal discriminant
$\Delta_{E_{T}}^{\text{min}}$. In particular, $\Delta_{E_{T}}^{\text{min}%
}=u_{T}^{-12}\gamma_{T}$ where $u_{T}$ and $\gamma_{T}$ are as given in
Theorem \ref{semistablecondthm} and Table \ref{ta:gamT}, respectively. In
particular, the invariant $c_{4}$ associated to a global minimal model of
$E_{T}$ is $u_{T}^{-4}\alpha_{T}$ where $\alpha_{T}$ is as given in Table
\ref{ta:alpT}. Consequently, $E_{T}$ has additive reduction at a prime $p$ if
and only if $p$ divides $\gcd\!\left(  u_{T}^{-4}\alpha_{T},u_{T}^{-12}%
\gamma_{T}\right)  $. In particular, if $E_{T}$ has $j$-invariant $0$ or
$1728$, then $E_{T}$ has additive reduction at each prime dividing the minimal
discriminant $u_{T}^{-12}\gamma_{T}$. Moreover, if $E_{T}$ has a point of
order $N\geq4$, then $T$ is either $C_{4}$ or $C_{6}$ by Lemma \ref{sslemc4j0}%
. If $|T|\geq5$, then Theorem \ref{Thm1} implies
that if $E_{T}$ has additive reduction at a prime $p$, then $p$ is in the set $S$
given below:
\begin{equation}%
\begin{tabular}
[c]{ccccccccccc}\hline
$T$ & $C_{5}$ &$C_{6}$& $C_{7}$ & $C_{8}$ & $C_{9}$ & $C_{10}$ & $C_{12}$ &
$C_{2}\times C_{4}$ & $C_{2}\times C_{6}$ & $C_{2}\times C_{8}$\\\hline
$S$ & $\left\{  5\right\}$& $\left\{  2,3\right\}  $ & $\left\{  7\right\}  $ & $\left\{  2\right\}  $
& $\left\{  3\right\}  $ & $\left\{  5\right\}  $ & $\left\{  2,3\right\}  $ &
$\left\{  2\right\}  $ & $\left\{  2,3\right\}  $ & $\left\{  2\right\}
$\\\hline
\end{tabular}
\ \label{set}%
\end{equation}
We now show the theorem by considering each $T$ separately.

\textbf{Case 1.} Let $T=C_{2}$. By Theorem \ref{semistablecondthm}, the
minimal discriminant of $E_{T}$ is $u_{T}^{-12}\gamma_{T}$ where $u_{T}%
\in\left\{  1,2,4\right\}  $.

First, suppose $E_{T}$ has $j$-invariant $0$. By Lemma \ref{sselmmac2result},
we may assume that $\left(  a,b,d\right)  =\left(  3b,b,-3\right)  $ for $b$
a squarefree integer. By Theorem \ref{semistablecondthm}, $\Delta_{E_{T}%
}^{\text{min}}=-u_{T}^{-12}2^{10}3^{3}b^{6}$ where $u_{T}\in\left\{
1,2\right\}  $. In particular, $2$ divides $\Delta_{E_{T}}^{\text{min}}$ and
therefore $E_{T}$ has additive reduction at $2$. Note that the assumptions on
$a,b,d$ imply that either $v_{2}\!\left(  b\right)  =0$ or $v_{2}\!\left(
a\right)  =v_{2}\!\left(  b\right)  =1$ with $v_{2}\!\left(  b^{2}%
d-a^{2}\right)  =4$ which agrees with the claimed conditions. Next, observe
that $E_{T}$ has additive reduction at an odd prime $p$ if and only if $p$
divides $3b$. This shows the claim since $\gcd\!\left(  a,bd\right)  =3b$.

Now suppose $E_{T}$ has $j$-invariant $1728$. By Lemma \ref{sselmmac2result},
we may assume that $\left(  a,b,d\right)  =\left(  0,b,d\right)  $ with $b$
squarefree. By Theorem \ref{semistablecondthm}, $\Delta_{E_{T}}^{\text{min}%
}=2^{6}b^{6}d^{3}$. It follows that $E_{T}$ has additive reduction at each
prime dividing $2bd$. This agrees with the claim since $\gcd\!\left(
a,bd\right)  =bd$ and $v_{2}\!\left(  b\right)  =0,1$ with $v_{2}\!\left(
a\right)  =\infty$.

It remains to consider the case when the $j$-invariant of $E_{T}$ is not $0$
or $1728$. By Lemma~\ref{polynomials}, $\gcd\!\left(  u_{T}^{-4}\alpha
_{T},u_{T}^{-12}\gamma_{T}\right)  $ divides $2^{10}\gcd\!\left(  a^{6}%
,b^{6}d^{3}\right)  $. In particular, if $E_{T}$ has additive reduction at a
prime $p$, then $p=2$ or $p$ is an odd prime dividing $\gcd\!\left(
a,bd\right)  $. It remains to show the converse. If $p$ is an odd prime
dividing $\gcd\!\left(  a,bd\right)  $, then by inspection $p$ divides
$\gcd\!\left(  \alpha_{T},\gamma_{T}\right)  $. Since $v_{p}\!\left(
\Delta_{E_{T}}^{\text{min}}\right)  =v_{p}\!\left(  \gamma_{T}\right)  $ and
$v_{p}\!\left(  u_{T}^{-4}\alpha_{T}\right)  =v_{p}\!\left(  \alpha
_{T}\right)  $ for $p$ odd, we conclude that $E_{T}$ has additive reduction at
an odd prime $p$ if and only if $p$ divides $\gcd\!\left(  a,bd\right)  $. It
remains to consider the case when $E_{T}$ has additive reduction at $2$. We do
this by considering the cases corresponding to $u_{T}=1,2,$ or $4$.

\qquad\textbf{Subcase 1a.} Suppose $u_{T}=4$. Then $v_{2}\left(  b^{2}%
d-a^{2}\right)  \geq8$ with $v_{2}\!\left(  a\right)  =v_{2}\!\left(
b\right)  =1$ and $a\equiv2\ \operatorname{mod}8$. Write $b=2\hat{b}$ for some
odd integer $\hat{b}$. Then $a^{2}=4\hat{b}^{2}d-2^{8}k$ for some integer $k$
and observe that $u_{T}^{-4}\alpha_{T}=\hat{b}^{2}d-16k$ is odd. Thus $E_{T}$
is semistable at $2$.

\qquad\textbf{Subcase 1b.} Suppose $u_{T}=2$. First, suppose that
$v_{2}\!\left(  a\right)  =v_{2}\!\left(  b\right)  =1$ with $\left(
i\right)  \ 4\leq v_{2}\!\left(  b^{2}d-a^{2}\right)  \leq7$ or $\left(
ii\right)  $ $v_{2}\!\left(  b^{2}d-a^{2}\right)  \geq8$ and $a\equiv
6\ \operatorname{mod}8$. Then $a^{2}=b^{2}d-16k$ for some integer $k$. Hence
$u_{T}^{-4}\alpha_{T}=4b^{2}d-16k$ and $u_{T}^{-12}\gamma_{T}=4b^{2}dk^{2}$.
It follows that $E_{T}$ has additive reduction at $2$.

Now suppose $v_{2}\!\left(  b\right)  \geq3$ and $a\equiv3\ \operatorname{mod}%
4$. Write $b=8\hat{b}$ for some integer $\hat{b}$ and observe that $u_{T}%
^{-4}\alpha_{T}=192\hat{b}^{2}d+a^{2}$ is odd. Hence $E_{T}$ is semistable at
$2$.

\qquad\textbf{Subcase 1c.} Suppose $u_{T}=1$. Then $\alpha_{T}$ and
$\gamma_{T}$ are even and it follows that $E_{T}$ has additive reduction at
$2$. Lastly, we note that by Theorem \ref{semistablecondthm}, $u_{T}=1$ if and
only if $\left(  i\right)  $ $v_{2}\!\left(  b^{2}d-a^{2}\right)  \leq3$ with
$v_{2}\!\left(  a\right)  =v_{2}\!\left(  b\right)  =1$, $\left(  ii\right)  $
$v_{2}\!\left(  b\right)  \geq3$ with $a\not \equiv 3\ \operatorname{mod}4$,
$\left(  iii\right)  $ $v_{2}\!\left(  b\right)  =0,2$, or $\left(  iv\right)
$ $v_{2}\!\left(  b\right)  =1$ with $v_{2}\!\left(  a\right)  \neq1$. This
concludes the proof for $T=C_{2}$.

\textbf{Case 2.} Let $T=C_{3}^{0}$. By Theorem \ref{semistablecondthm},
$\Delta_{E_{T}}^{\text{min}}=-27a^{4}$. Moreover, $E_{T}$ has $j$-invariant
$0$ since $\alpha_{T}=0$. It follows that $E_{T}$ has additive reduction at
each prime dividing $3a$. The claim now follows since $a$ is assumed to be cubefree.

\textbf{Case 3.} Let $T=C_{3}$ and write $a=c^{3}d^{2}e$ with $d$ and $e$
relatively prime positive squarefree integers. By Theorem
\ref{semistablecondthm}, the minimal discriminant of $E_{T}$ is $u_{T}%
^{-12}\gamma_{T}$ with $u_{T}=c^{2}d$. In particular,%
\begin{equation}
u_{T}^{-4}\alpha_{T}=cd^{2}e^{3}\left(  a-24b\right)  \qquad\text{and}\qquad
u_{T}^{-12}\gamma_{T}=d^{4}e^{8}b^{3}\left(  a-27b\right)  . \label{qussC3}%
\end{equation}
Now suppose $E_{T}$ has $j$-invariant $0$ or $1728$. Then the $j$-invariant of
$E_{T}$ is $0$ by the proof of Theorem \ref{semistablecondthm}. By Proposition
\ref{rationalmodels}, it follows that $\left(  a,b\right)  =\left(
24,1\right)  $. Then $E_{T}$ has additive reduction at $3$ since $u_{T}%
^{-12}\gamma_{T}=-3^{9}$. The claim now follows that since $v_{3}\!\left(
24\right)  =1$. Now suppose that the $j$-invariant of $E_{T}$ is not $0$ or
$1728$. By Lemma \ref{polynomials} and the proof of Theorem
\ref{semistablecondthm}, $\gcd\!\left(  u_{T}^{-4}\alpha_{T},u_{T}^{-12}%
\gamma_{T}\right)  $ divides $2^{15}3^{6}cde$. It follows that if $E_{T}$ has
additive reduction at a prime $p$, then $p=2,3$ or $p$ divides $cde$.

Now suppose $p\geq5$ divides $cde$. If $p|c$ and $p\nmid de$, then $u_{T}%
^{-4}\alpha_{T}\equiv-24b\ \operatorname{mod}p$ is non-zero since $a$ and $b$
are relatively prime. Consequently, if $E_{T}$ has additive reduction at a
prime $p\geq5$, then $p$ divides $de$. Now observe that if a prime $p$ divides
$de$, then $v_{p}\!\left(  a\right)  \not \equiv 0\ \operatorname{mod}3$.
Moreover, $p$ divides $u_{T}^{-4}\alpha_{T}$ and $u_{T}^{-6}\gamma_{T}$ by
(\ref{qussC3}). In particular, $E_{T}$ has additive reduction at $p$.

It remains to consider the case when $p$ is $2$ or $3$. Since $u_{T}%
^{-4}\alpha_{T}\equiv c^{4}d^{4}e^{4}\ \operatorname{mod}3$ and $u_{T}%
^{-12}\gamma_{T}\equiv c^{3}d^{6}e^{9}b^{3}\ \operatorname{mod}3$, we conclude
that $E_{T}$ has additive reduction at $3$ if and only if $3$ divides $a$.
Similarly, $E_{T}$ has additive reduction at $2$ if and only if $2$ divides
$de$ since $u_{T}^{-4}\alpha_{T}\equiv c^{4}d^{4}e^{4}\ \operatorname{mod}2$
and $u_{T}^{-12}\gamma_{T}\equiv d^{4}e^{8}b^{3}\left(  a+b\right)
\ \operatorname{mod}2$.

\textbf{Case 4.} Let $T=C_{4}$ and write $a=c^{2}d$ for $d$ a positive
squarefree integer. Now suppose $E_{T}$ has $j$-invariant $0$ or $1728$. Then
$\left(  a,b\right)  =\left(  8,-1\right)  $ by Lemma \ref{sslemc4j0}. In
particular, $E_{T}$ has additive reduction at $2$ since $\Delta_{E_{T}%
}^{\text{min}}=-4096$. This agrees with the claim since $v_{2}\!\left(
a\right)  =3$. Now suppose that the $j$-invariant of $E_{T}$ is not $0$ or
$1728$. By Theorem \ref{semistablecondthm}, the minimal discriminant of
$E_{T}$ is $u_{T}^{-12}\gamma_{T}$ where $u_{T}\in\left\{  c,2c\right\}  $. By
Lemma \ref{polynomials} and the proof of Theorem \ref{semistablecondthm},
$\gcd\!\left(  u_{T}^{-4}\alpha_{T},u_{T}^{-12}\gamma_{T}\right)  $ divides
$2^{12}d^{2}$. Consequently, if $E_{T}$ has additive reduction at a prime $p$,
then $p$ divides $2d$. Now observe that if $p$ is an odd prime, then
$v_{p}\!\left(  c^{-4}\alpha_{T}\right)  =v_{p}\!\left(  u_{T}^{-4}\alpha
_{T}\right)  $ and $v_{p}\!\left(  c^{-12}\gamma_{T}\right)  =v_{p}\!\left(
u_{T}^{-12}\gamma_{T}\right)  $. Since
\begin{equation}
c^{-4}\alpha_{T}=d^{2}\left(  a^{2}+16ab+16b^{2}\right)  \qquad\text{and}%
\qquad c^{-12}\gamma_{T}=b^{4}c^{2}d^{7}\left(  a+16b\right)  ,
\label{c4addred}%
\end{equation}
it follows that $E_{T}$ has additive reduction at an odd prime $p$ if and only
if $p$ divides $d$. Note that this is equivalent to $v_{p}\!\left(  a\right)
\equiv1\ \operatorname{mod}2$. For $p=2$, we consider the cases corresponding
to $u_{T}=c$ or $2c$ separately.

\qquad\textbf{Subcase 4a.} Suppose $u_{T}=2c$. Then $v_{2}\!\left(  a\right)
\geq8$ is even with $bd\equiv3\ \operatorname{mod}4$. Now write $c=2^{4}k$ for
some integer $k$ and observe that $E_{T}$ is semistable at $2$ since
\[
u_{T}^{-4}\alpha_{T}=b^{2}d^{2}+256bd^{3}k^{2}+4096d^{4}k^{4}\equiv
1\ \operatorname{mod}4.
\]

\qquad\textbf{Subcase 4b.} Suppose $u_{T}=c$. Then $v_{2}\!\left(  a\right)
\leq7$ or $bd\not \equiv 3\ \operatorname{mod}4$. By (\ref{c4addred}),
$u_{T}^{-4}\alpha_{T}$ is even if and only if $a$ is even. But if $a$ is even,
then $u_{T}^{-12}\gamma_{T}$ is even. Thus $E_{T}$ has additive reduction at
$2$ if and only if $\left(  i\right)  $ $v_{2}\!\left(  a\right)
\equiv1\ \operatorname{mod}2$, $\left(  ii\right)  $ $v_{2}\!\left(  a\right)
=2,4,6$, or $\left(  iii\right)  $ $v_{2}\!\left(  a\right)  \geq8$ with
$bd\equiv1\ \operatorname{mod}4$. This concludes the proof for $T=C_{4}$.

\textbf{Case 5.} Let $T=C_{5}$. By Theorem \ref{semistablecondthm}, the
minimal discriminant of $E_{T}$ is $\gamma_{T}$. By (\ref{set}), if $E_{T}$ has
additive reduction at a prime $p$, then $p=5$. Since $\alpha
_{T}\equiv\left(  a+3b\right)  ^{4}\ \operatorname{mod}5$ and $\gamma
_{T}\equiv4a^{5}b^{5}\left(  a+3b\right)  ^{2}\ \operatorname{mod}5$, we
conclude that $E_{T}$ has additive reduction at $5$ if and only if
$v_{5}\!\left(  a+3b\right)  \geq1$.

\textbf{Case 6.} Let $T=C_{6}$. Now suppose $E_{T}$ has $j$-invariant $0$ or
$1728$. Then $\left(  a,b\right)  =\left(  3,-1\right)  $ by Lemma
\ref{sslemc4j0}. It follows that $E_{T}$ has additive reduction at $2$ and $3$
since $\Delta_{E_{T}}^{\text{min}}=-432$. This agrees with the claim since
$v_{3}\!\left(  a+b\right)  =v_{3}\!\left(  a\right)  =1$. Now suppose that
the $j$-invariant of $E_{T}$ is not $0$ or $1728$. By Theorem
\ref{semistablecondthm}, the minimal discriminant of $E_{T}$ is $u_{T}%
^{-12}\gamma_{T}$ where $u_{T}\in\left\{  1,2\right\}  $. By (\ref{set}), if
$E_{T}$ has additive reduction at a prime $p$, then $p=2,3$. Next,
observe that $v_{3}\!\left(  \alpha_{T}\right)  =v_{3}\!\left(  u_{T}%
^{-4}\alpha_{T}\right)  $ and $v_{3}\!\left(  \gamma_{T}\right)
=v_{3}\!\left(  u_{T}^{-12}\gamma_{T}\right)  $. Since $\alpha_{T}\equiv
a^{4}\ \operatorname{mod}3$ and $\gamma_{T}\equiv a^{3}b^{6}\left(
a+b\right)  \left(  a^{2}-ab+b^{2}\right)  \ \operatorname{mod}3$, we conclude
that $E_{T}$ has additive reduction at $3$ if and only if $3$ divides $a$. For
$p=2$, we proceed by cases.

\qquad\textbf{Subcase 6a.} Suppose $u_{T}=2$. Then $v_{2}\!\left(  a+b\right)
\geq3$ and write $b=8k-a$ for some integer $k$. It is then checked with this
substitution that $u_{T}^{-4}\alpha_{T}\equiv a^{4}\ \operatorname{mod}2$. By
assumption, $a$ is odd, and hence $E_{T}$ is semistable at $2$.

\qquad\textbf{Subcase 6b.} Suppose $u_{T}=1$. Then $v_{2}\!\left(  a+b\right)
\leq2$. Since $\alpha_{T}\equiv\left(  a+b\right)  ^{4}\ \operatorname{mod}2$
and $\gamma_{T}\equiv a^{2}b^{6}\left(  a+b\right)  ^{6}\ \operatorname{mod}%
2$, we conclude that $E_{T}$ has additive reduction at $2$ if and only if
$v_{2}\!\left(  a+b\right)  =1,2$.

\textbf{Case 7.} Let $T=C_{7}$. By Theorem \ref{semistablecondthm}, the
minimal discriminant of $E_{T}$ is $\gamma_{T}$. By (\ref{set}), if $E_{T}$ has
additive reduction at a prime $p$, then $p=7$. Since $\alpha
_{T}\equiv\left(  a+4b\right)  \left(  a+2b\right)  ^{7}\ \operatorname{mod}7$
and $\gamma_{T}\equiv6a^{7}b^{7}\left(  a+4b\right)  ^{3}\left(  a-b\right)
^{7}\ \operatorname{mod}7$, we conclude that $E_{T}$ has additive reduction at
$7$ if and only if $v_{7}\!\left(  a+4b\right)  \geq1$.

\textbf{Case 8.} Let $T=C_{8}$. By Theorem \ref{semistablecondthm}, the
minimal discriminant of $E_{T}$ is $u_{T}^{-12}\gamma_{T}$ where $u_{T}%
\in\left\{  1,2\right\}  $. By (\ref{set}), if $E_{T}$ has additive reduction at
a prime $p$, then $p=2$. We now proceed by cases.

\qquad\textbf{Subcase 8a.} Suppose $u_{T}=2$. Then $v_{2}\!\left(  a\right)
=1$. It follows that $u_{T}^{-4}\alpha_{T}\equiv b^{8}\ \operatorname{mod}2$
and thus $E_{T}$ is semistable at $2$ since $b$ is odd.

\qquad\textbf{Subcase 8b.} Suppose $u_{T}=1$. Then $v_{2}\!\left(  a\right)
\neq1$. Since $\alpha_{T}\equiv a^{8}\ \operatorname{mod}2$ and $\gamma
_{T}\equiv a^{8}b^{8}\left(  a+b\right)  ^{8}\ \operatorname{mod}2$, we
conclude that $E_{T}$ has additive reduction at $2$ if and only if
$v_{2}\!\left(  a\right)  \geq2$.

\textbf{Case 9.} Let $T=C_{9}$. By Theorem \ref{semistablecondthm}, the
minimal discriminant of $E_{T}$ is $\gamma_{T}$. By (\ref{set}), if $E_{T}$ has
additive reduction at a prime $p$, then $p=3$. Since $\alpha
_{T}=\left(  a+b\right)  ^{12}\ \operatorname{mod}3$ and $\gamma_{T}%
=2a^{9}b^{9}\left(  a^{2}-b^{2}\right)  ^{9}\ \operatorname{mod}3$, we
conclude that $E_{T}$ has additive reduction at $3$ if and only if
$v_{3}\!\left(  a+b\right)  \geq1$.

\textbf{Case 10.} Let $T=C_{10}$. By Theorem \ref{semistablecondthm}, the
minimal discriminant of $E_{T}$ is $u_{T}^{-12}\gamma_{T}$ where $u_{T}%
\in\left\{  1,2\right\}  $. By (\ref{set}), if $E_{T}$ has additive reduction at
a prime $p$, then $p=5$. Next, observe that $v_{5}\!\left(
u_{T}^{-4}\alpha_{T}\right)  =v_{5}\!\left(  \alpha_{T}\right)  $ and
$v_{5}\!\left(  u_{T}^{-12}\gamma_{T}\right)  =v_{5}\!\left(  \gamma
_{T}\right)  $. We conclude that $E_{T}$ has additive reduction at $5$ if and
only if $v_{5}\!\left(  a+b\right)  \geq1$ since $\alpha_{T}\equiv\left(
a+b\right)  ^{12}\ \operatorname{mod}5$ and $\gamma_{T}\equiv a^{5}%
b^{10}\left(  a+b\right)  ^{6}\left(  a^{15}+a^{10}b^{5}+3b^{15}\right)
\ \operatorname{mod}5$.

\textbf{Case 11.} Let $T=C_{12}$. By Theorem \ref{semistablecondthm}, the
minimal discriminant of $E_{T}$ is $u_{T}^{-12}\gamma_{T}$ where $u_{T}%
\in\left\{  1,2\right\}  $. By (\ref{set}), if $E_{T}$ has additive reduction at
a prime $p$, then $p=2,3$. Now observe that $\alpha_{T}\equiv
a^{16}\ \operatorname{mod}3$ and $a$ is a factor of $\gamma_{T}$. Thus $E_{T}$
has additive reduction at $3$ if and only if $v_{3}\!\left(  a\right)  \geq1$
since $v_{3}\!\left(  u_{T}^{-4}\alpha_{T}\right)  =v_{3}\!\left(  \alpha
_{T}\right)  $ and $v_{3}\!\left(  u_{T}^{-12}\gamma_{T}\right)
=v_{3}\!\left(  \gamma_{T}\right)  $. For $p=2$, we proceed by cases.

\qquad\textbf{Subcase 11a.} Suppose $u_{T}=2$. Then $a$ is even and
$u_{T}^{-4}\alpha_{T}\equiv b^{16}\ \operatorname{mod}2$. Since $b$ is odd, we
deduce that $E_{T}$ is semistable at $2$.

\qquad\textbf{Subcase 11b.} Suppose $u_{T}=1$. Then $a$ is odd, and
$\alpha_{T}$ is odd since $\alpha_{T}\equiv a^{16}\ \operatorname{mod}2$. Thus
$E_{T}$ is semistable at $2$.

\textbf{Case 12.} Let $T=C_{2}\times C_{2}$. By Theorem
\ref{semistablecondthm}, the minimal discriminant of $E_{T}$ is $u_{T}%
^{-12}\gamma_{T}$ where $u_{T}\in\left\{  1,2\right\}  $. Now suppose $E_{T}$
has $j$-invariant $0$ or $1728$. Then by Lemma \ref{sselmmac2c2result},
$\left(  a,b,d\right)  =\left(  2,1,d\right)  $ for $d$ a squarefree integer.
By Theorem \ref{semistablecondthm}, the minimal discriminant of $E_{T}$ is
$64d^{6}$. Consequently, $E_{T}$ has additive reduction at each prime dividing
$2d$ which verifies with the claim. Now suppose that the $j$-invariant of
$E_{T}$ is not $0$ or $1728$. By Lemma \ref{polynomials}, $\gcd\!\left(
u_{T}^{-4}\alpha_{T},u_{T}^{-12}\gamma_{T}\right)  $ divides $2^{4}d^{6}$. It
follows that if $E_{T}$ has additive reduction at a prime $p$, then $p$
divides $2d$. Now observe that if $p$ is an odd prime, then $v_{p}\!\left(
\alpha_{T}\right)  =v_{p}\!\left(  u_{T}^{-4}\alpha_{T}\right)  $ and
$v_{p}\!\left(  \gamma_{T}\right)  =v_{p}\!\left(  u_{T}^{-12}\gamma
_{T}\right)  $. Since $d$ divides both $\alpha_{T}$ and $\gamma_{T}$, we
conclude that $E_{T}$ has additive reduction at an odd prime $p$ if and only
if $p$ divides $d$. For $p=2$, we proceed by cases.

\qquad\textbf{Subcase 12a.} Suppose $u_{T}=2$. Then $v_{2}\!\left(  a\right)
\geq4$ and $bd\equiv1\ \operatorname{mod}4$. It follows that $E_{T}$ is
semistable at $2$ since $u_{T}^{-4}\alpha_{T}=d^{2}\left(  a^{2}%
-ab+b^{2}\right)  \equiv1\ \operatorname{mod}2$.

\qquad\textbf{Subcase 12b.} Suppose $u_{T}=1$. Then $v_{2}\!\left(  a\right)
\leq3$ or $bd\not \equiv 1\ \operatorname{mod}4$. Since $\alpha_{T}$ and
$\gamma_{T}$ are even, we conclude that $E_{T}$ has additive reduction at $2$.

\textbf{Case 13.} Let $T=C_{2}\times C_{4}$. By Theorem
\ref{semistablecondthm}, the minimal discriminant of $E_{T}$ is $u_{T}%
^{-12}\gamma_{T}$ where $u_{T}\in\left\{  1,2,4\right\}  $. By (\ref{set}),
if $E_{T}$ has additive reduction at a prime $p$, then $p=2$. We now
proceed by cases.

\qquad\textbf{Subcase 13a.} Suppose $u_{T}=4$. Then $v_{2}\!\left(  a\right)
=2$ and $v_{2}\!\left(  a+4b\right)  \geq4$. Write $a=16k-4b$ for some integer
$k$ and observe that $u_{T}^{-4}\alpha_{T}\equiv b^{4}\ \operatorname{mod}2$.
Since $b$ is odd, it follows that $E_{T}$ is semistable at $2$.

\qquad\textbf{Subcase 13b.} Suppose $u_{T}=2$. Then $v_{2}\!\left(  a\right)
\geq2$ and $v_{2}\!\left(  a+4b\right)  \leq3$. Write $a=4k$ for some integer
$k$ and observe that $u_{T}^{-4}\alpha_{T}$ and $u_{T}^{-12}\gamma_{T}$ are
even. Thus $E_{T}$ has additive reduction at $2$.

\qquad\textbf{Subcase 13c.} Suppose $u_{T}=1$. Then $v_{2}\!\left(  a\right)
\leq1$. Observe that $\alpha_{T}$ is even if and only if $a$ is even. But if
$a$ is even, then $\gamma_{T}$ is even. Consequently, $E_{T}$ has additive
reduction at $2$ if and only if $v_{2}\!\left(  a\right)  =1$.

\textbf{Case 14. }Let $T=C_{2}\times C_{6}$. By Theorem
\ref{semistablecondthm}, the minimal discriminant of $E_{T}$ is $u_{T}%
^{-12}\gamma_{T}$ where $u_{T}\in\left\{  1,4,16\right\}  $. By (\ref{set}), if
$E_{T}$ has additive reduction at a prime $p$, then $p=2,3$. Now
observe that $v_{3}\!\left(  \alpha_{T}\right)  =v_{3}\!\left(  u_{T}%
^{-4}\alpha_{T}\right)  $ and $v_{3}\!\left(  \gamma_{T}\right)
=v_{3}\!\left(  u_{T}^{-12}\gamma_{T}\right)  $. Hence $E_{T}$ has additive
reduction at $3$ if and only if $3$ divides $b$ since $\alpha_{T}\equiv
b^{8}\ \operatorname{mod}3$ and $b$ is a factor of $\gamma_{T}$. For $p=2$, we
observe that by the proof of Theorem \ref{semistablecondthm}, $v_{2}\!\left(
u_{T}^{-4}\alpha_{T}\right)  =0$. Therefore, $E_{T}$ is semistable at $2$.

\textbf{Case 15. }Let $T=C_{2}\times C_{8}$. By Theorem
\ref{semistablecondthm}, the minimal discriminant of $E_{T}$ is $u_{T}%
^{-12}\gamma_{T}$ where $u_{T}\in\left\{  1,16,64\right\}  $. By (\ref{set}),
if $E_{T}$ has additive reduction at a prime $p$, then $p=2$. By the
proof of Theorem \ref{semistablecondthm}, $v_{2}\!\left(  u_{T}^{-4}\alpha
_{T}\right)  =0$. Therefore, $E_{T}$ is semistable.
\end{proof}

\begin{corollary}\label{semistableconditions}
Suppose that the parameters of $E_{T}$ satisfy the conclusion of Proposition
\ref{rationalmodels}. Then $E_{T}$ is semistable if and only if one of the
conditions on the parameters of $E_{T}$ below is satisfied. In particular,
$E_{C_{2}\times C_{8}}$ is semistable.
\end{corollary}
{\renewcommand*{\arraystretch}{1.15}%
\begin{longtable*}{cc}
%\caption{Semistablity of $E_{T}$}\\
\hline
$T$ & Conditions on $a,b,d$ \\
\hline
\endfirsthead
{\emph{continued}}\\
\hline
$T$ & Conditions on $a,b,d$\\
\hline
\endhead
\hline
\multicolumn{2}{r}{\emph{continued on next page}}
\endfoot
\hline
\endlastfoot
	
$C_{2}$ & $\gcd\!\left(  a,bd\right)  =2$, $v_{2}\left(  b^{2}d-a^{2}\right)  \geq8$ with
$v_{2}\!\left(  a\right)  =v_{2}\!\left(  b\right)  =1$ and $a\equiv2\ \operatorname{mod}8$
\\\cmidrule(lr){2-2}
& $\gcd\!\left(  a,bd\right)  =1$ and $v_{2}\!\left(  b\right)  \geq3$ with $a\equiv3\ \operatorname{mod}4$\\\hline
$C_{3}$ & $a$ is a cube with  $v_{3}\!\left(  a\right) =0$\\\hline
$C_{4}$ & $a$ is an odd square \\\cmidrule(lr){2-2}
& $a$ is a square and $v_{2}\!\left(  a\right)  \geq8$ with $b\equiv
3\ \operatorname{mod}4$\\\hline
$C_{5}$ & $v_{5}\!\left(  a+3b\right)  =0$ \\\hline
$C_{6}$ & $v_{3}\!\left(  a\right)  =0$ and $v_{2}\!\left(  a+b\right)  \neq1,2$\\\hline
$C_{7}$ & $v_{7}\!\left(  a+4b\right)  =0$ \\\hline
$C_{8}$ & $v_{2}\!\left(  a\right)  \leq1$ \\\hline
$C_{9}$ & $v_{3}\!\left(  a+b\right)  =0$ \\\hline
$C_{10}$ & $v_{5}\!\left(  a+b\right)  =0$ \\\hline
$C_{12}$ & $v_{3}\!\left(  a\right)  =0$ \\\hline
$C_{2}\times C_{2}$ & $d=1$ and $v_{2}\!\left(  a\right)  \geq4$ with $b\equiv1\ \operatorname{mod}4$\\\hline
$C_{2}\times C_{4}$ & $a$ is odd \\\cmidrule(lr){2-2}
& $v_{2}\!\left(  a\right)  =2$ with $v_{2}\!\left(  a+4b\right)  \geq4$\\\hline
$C_{2}\times C_{6}$ & $v_{3}\!\left(  b\right)  =0$\\\hline
$C_{2}\times C_{8}$ & $v_{2}\!\left(  a\right)  \geq 0$
\label{ta:ssET}	
\end{longtable*}}%\addtocounter{table}{-1}

\section{Corollaries and Examples\label{sec:corexa}}

The following statement is an automatic consequence of the proof of Theorem
\ref{semistablecondthm}.

\begin{corollary}
Let $E$ be a rational elliptic curve with a rational point of order $3,5,$ or
$7$. Then the discriminant $\Delta$ of $E$ is minimal if and only if
$v_{p}\!\left(  \Delta\right)  <12$ or $v_{p}\!\left(  c_{4}\right)  <4$ for
all primes $p$ where $c_{4}$ is the invariant of the Weierstrass model of $E$.
\end{corollary}

This is only true for an arbitrary elliptic curve for primes $p\geq5$
\cite[Remark VII.1.1]{MR2514094}. Next, we recall that if $E$ is an elliptic
curve over a number field $K$ which has additive reduction at two primes with
distinct residue characteristics, then $E\!\left(  K\right)  _{\text{tors}}$
divides $12$ by Flexor and Oesterl\'{e}'s Theorem \cite{MR1065153}. Moreover,
Flexor and Oesterl\'{e} showed that this divisibility condition is sharp since
the elliptic curve $E:y^{2}-2y=x^{3}$ over $K=\mathbb{Q}\!\left(  \sqrt{-3}\right)  $ has $E\!\left(  K\right)
_{\text{tors}}\cong C_{2}\times C_{6}$ with additive reduction at two primes and their
residue characteristic are $2$ and $3$. Our next result, shows that over $\mathbb{Q}$ the divisibility condition is not sharp.

\begin{corollary}
Let $E$ be a rational elliptic curve. If $E$ has additive reduction at three
or more primes, then $E\!\left(
%TCIMACRO{\U{211a} }%
%BeginExpansion
\mathbb{Q}
%EndExpansion
\right)  _{\text{tors}}\cong C_{N}$ for $N\leq4$ or $E\!\left(
%TCIMACRO{\U{211a} }%
%BeginExpansion
\mathbb{Q}
%EndExpansion
\right)  _{\text{tors}}\cong C_{2}\times C_{2}$. If $E$ has additive reduction
at two primes, then  $E\!\left(
%TCIMACRO{\U{211a} }%
%BeginExpansion
\mathbb{Q}
%EndExpansion
\right)  _{\text{tors}}$ divides $4$ or $6$.
\end{corollary}

\begin{proof}
The elliptic curve $y^{2}=x^{3}+30$ has additive reduction at the primes
$2,3,$ and $5$ and has trivial torsion subgroup. Consequently, the case of trivial torsion does occur over $\mathbb{Q}$. The Corollary now holds for
the remaining $T$'s by Proposition~\ref{rationalmodels} and Theorem~\ref{semis}.
\end{proof}

\begin{corollary}
\label{oddc4c6}Let $T=C_{7},C_{9},C_{10},C_{12},C_{2}\times C_{6},$ or
$C_{2}\times C_{8}$. If $E$ is a rational elliptic curve with
$T\hookrightarrow E\!\left(
%TCIMACRO{\U{211a} }%
%BeginExpansion
\mathbb{Q}
%EndExpansion
\right)  $, then the invariants $c_{4}$ and $c_{6}$ associated to a global
minimal model of $E$ are odd.
\end{corollary}

\begin{proof}
For the given $T$, let $u_{T}$ be as given in Theorem \ref{semistablecondthm}.
Then the invariants $c_{4}$ and $c_{6}$ associated to a global minimal model
of $E_{T}$ are $u_{T}^{-4}\alpha_{T}$ and $u_{T}^{-6}\beta_{T}$, respectively.
It is then verified that $v_{2}\!\left(  u_{T}^{-4}\alpha_{T}\right)  =0$ and
by Lemma \ref{ch:ss:elemlemm} we conclude that $v_{2}\!\left(  u_{T}^{-6}%
\beta_{T}\right)  =0$. The result now follows by Proposition
\ref{rationalmodels}.
\end{proof}

\begin{corollary}
Let $T=C_{12}$, $C_{2}\times C_{6}$, or $C_{2}\times C_{8}$. If $E$ is a rational elliptic
curve with $T\hookrightarrow E\!\left(
%TCIMACRO{\U{211a} }%
%BeginExpansion
\mathbb{Q}
%EndExpansion
\right)  $, then the minimal discriminant $\Delta_{E}^{\text{min}}$ of $E$ is
divisible by $2,3,$ and $5$. Moreover, if $T=C_{2}\times C_{8}$, then
$\Delta_{E}^{\text{min}}$ is divisible by $7$.
\end{corollary}

\begin{proof}
By Theorem \ref{semistablecondthm}, $\Delta_{E_{T}}^{\text{min}}=u_{T}%
^{-12}\gamma_{T}$ and $v_{p}\!\left(  \Delta_{E_{T}}^{\text{min}}\right)
=v_{p}\!\left(  \gamma_{T}\right)  $ for odd primes $p$. It is then verified
for the given $T$ that it is always the case that $u_{T}^{-12}\gamma_{T}%
\equiv0\ \operatorname{mod}2$ and $\gamma_{T}\equiv0\ \operatorname{mod}p$ for
$p=2,3,5$. Moreover, $\gamma_{C_{2}\times C_{8}}\equiv
0\ \operatorname{mod}7$. The result now follows by Proposition~\ref{rationalmodels}.
\end{proof}

\begin{example}
\label{example1}The elliptic curve%
\[
E:y^{2}=x^{3}-1900650154752x+990015042347311104
\]
has torsion subgroup $E\!\left(
%TCIMACRO{\U{211a} }%
%BeginExpansion
\mathbb{Q}
%EndExpansion
\right)  _{\text{tors}}\cong C_{2}\times C_{4}$. The point $P=\left(
222288,760596480\right)  $ has order $4$, and placing $E$ in Tate normal form
with respect to $P$ results in the elliptic curve%
\[
E_{\text{TNF}}:y^{2}+xy-\frac{4585}{36864}y=x^{3}-\frac{4585}{36864}x^{2}.
\]
Now consider the elliptic curve $\mathcal{X}_{t}\!\left(  C_{4}\right)  .$ It
is clear that if $t=\frac{4585}{36864},$ then $\mathcal{X}_{t}\!\left(
C_{4}\right)  $ is $E_{\text{TNF}}$. Therefore $E$ is $%
%TCIMACRO{\U{211a} }%
%BeginExpansion
\mathbb{Q}
%EndExpansion
$-isomorphic to $E_{C_{4}}\!\left(  36864,4585\right)  $. Moreover,
$36864=2^{12}\cdot3^{2}$ and $4585\equiv1\ \operatorname{mod}4$. By Theorem
\ref{semis}, $E$ has additive reduction at $2$ and is semistable at each odd
prime $p$ since $c=2^{6}\cdot3$ with the notation of Theorem
\ref{semistablecondthm}. We also have that the minimal discriminant
$\Delta_{E}^{\text{min}}$ of $E$ and associated invariants $c_{4}$ and $c_{6}$
are%
\begin{align*}
\Delta_{E}^{\text{min}} &  =\left(  2^{6}\cdot3\right)  ^{-12}\gamma_{C_{4}%
}\!\left(  36864,4585\right)  =2^{16}\cdot3^{2}\cdot5^{4}\cdot7^{4}\cdot
83^{2}\cdot131^{4}\\
c_{4} &  =\left(  2^{6}\cdot3\right)  ^{-4}\alpha_{C_{4}}\!\left(
36864,4585\right)  =2^{4}\cdot274978321\\
c_{6} &  =\left(  2^{6}\cdot3\right)  ^{-4}\beta_{C_{4}}\!\left(
36864,4585\right)  =-2^{6}\cdot23\cdot29\cdot47\cdot313\cdot317\cdot1439.
\end{align*}
By Theorem \ref{GlobalMinModel}, $E^{\text{min}}:y^{2}+192xy-880320y=x^{3}-4585x^{2}$
is a global minimal model for $E$.
\end{example}

\begin{example}
\label{ex:ss1}Consider the elliptic curve $E$ given by the Weierstrass
equation%
\[
E:y^{2}=x^{3}-19057987954261048752x+31955359661403338940204703104.
\]
The point $P=\left(  2365794828,10458914400000\right)  $ is a torsion point of
order $12$ on $E$. Placing $E$ in Tate normal form with respect to $P$ yields
the Weierstrass equation%
\[
E_{\text{TNF}}:y^{2}+\frac{6021}{125}xy-\frac{430408}{1875}y=x^{3}%
-\frac{430408}{1875}x^{2}.
\]
In particular, $E_{\text{TNF}}$ is equal to $\mathcal{X}_{t}\!\left(
C_{12}\right)  $ for some $t$. Therefore, we solve for $t$ and attain%
\[
\frac{12t^{6}-30t^{5}+34t^{4}-21t^{3}+7t^{2}-t}{\left(  t-1\right)  ^{4}%
}=\frac{430408}{1875}\text{ and }1-\frac{-6t^{4}+9t^{3}-5t^{2}+t}{\left(
t-1\right)  ^{3}}=\frac{6021}{125}.
\]
Observe that the common rational solution to both equations is $t=\frac{11}%
{6}$ and therefore $E$ is isomorphic over $%
%TCIMACRO{\U{211a} }%
%BeginExpansion
\mathbb{Q}
%EndExpansion
$ to $E_{C_{12}}\!\left(  6,11\right)  $. Since $v_{3}\!\left(  6\right)  >0$,
we have by Theorem \ref{semis} that $E$ only has additive reduction at $3$.
Moreover, by Theorem \ref{semistablecondthm} its minimal discriminant and
invariants $c_{4}$ and $c_{6}$ associated to a global minimal model are
\begin{align*}
\Delta_{C_{12}}^{\text{min}} &  =2^{-12}\gamma_{C_{12}}\!\left(  6,11\right)
=2^{18}\cdot3^{7}\cdot5^{12}\cdot11^{12}\cdot61\cdot67^{4}\cdot73^{3}\\
c_{4} &  =2^{-4}\alpha_{C_{12}}\!\left(  6,11\right)  =3^{2}\cdot
23\cdot107\cdot227\cdot27361\cdot320687\\
c_{6} &  =2^{-6}\beta_{C_{12}}\!\left(  6,11\right)  =-3^{3}\cdot
503\cdot769\cdot47221\cdot18748939480561.
\end{align*}
Lastly,
\[
E^{\text{min}}:y^{2}+18063xy-12105225000y=x^{3}-32280600x^{2}%
\]
is a global minimal model for $E$ by\ Theorem \ref{GlobalMinModel}.
\end{example}

\begin{remark}
In Examples \ref{example1} and \ref{ex:ss1}, we considered rational elliptic curves with a torsion point of order at least $4$. This allowed us to consider the Tate normal form and thereby attain the parameters $a$ and $b$, which related these curves to our parameterized family $E_T(a,b)$. If instead, we had started with an arbitrary rational elliptic curve $E$, the algorithm in \cite{MR1931194} can be used to determine whether $E$ has a point of order at least $4$, and if it does, the algorithm returns the Tate normal form of $E$.
\end{remark}

\noindent \textbf{Acknowledgments.} The author is grateful to Edray Goins, Manami Roy, and the referees for their helpful comments and suggestions.

%\newpage

\section{\texorpdfstring{$E_{T}$}{} and its Associated Invariants}\label{SectionTables}
{\renewcommand*{\arraystretch}{1.2} 
\begin{longtable}{C{0.6in}C{1in}C{1.5in}C{1.5in}C{0.5in}}
    \caption[Weierstrass Model of $E_{T}$]{The Weierstrass Model $E_{T}:y^{2}+a_{1}xy+a_{3}y=x^{3}+a_{2}x^{2}+a_{4}x$} \label{ta:ETmodel}	\\
    \hline
    $T$ & $a_{1}$ & $a_{2}$ & $a_{3}$ & $a_{4}$ \\
    \hline
  \endfirsthead
    \caption[]{\emph{continued}}\\
    \hline
    $T$ & $a_{1}$ & $a_{2}$ & $a_{3}$ & $a_{4}$\\
    \hline
  \endhead
    \hline
    \multicolumn{4}{r}{\emph{continued on next page}}
  \endfoot
    \hline
  \endlastfoot
  
$C_{2}$ & $0$ & $2a$ & $0$ & $a^{2}-b^{2}d$ \\\hline
$C_{3}^{0}$ & $0$ & $0$ & $a$ & $0$ \\\hline
$C_{3}$ & $a$ & $0$ & $a^{2}b$ & $0$ \\\hline
$C_{4}$ & $a$ & $-ab$ & $-a^{2}b$ & $0$ \\\hline
$C_{5}$ & $a-b$ & $-ab$ & $-a^{2}b$ & $0$ \\\hline
$C_{6}$ & $a-b$ & $-ab-b^{2}$ & $-a^{2}b-ab^{2}$ & $0$\\\hline
$C_{7}$ & $a^{2}+ab-b^{2}$ & $a^{2}b^{2}-ab^{3}$ & $a^{4}b^{2}-a^{3}b^{3}$ & $0$ \\\hline
$C_{8}$ & $-a^{2}+4ab-2b^{2}$ & $-a^{2}b^{2}+3ab^{3}-2b^{4}$ & $-a^{3}b^{3}+3a^{2}
b^{4}-2ab^{5}$ & $0$ \\\hline
$C_{9}$ & $a^{3}+ab^{2}-b^{3}$ & $
a^{4}b^{2}-2a^{3}b^{3}+
2a^{2}b^{4}-ab^{5}
$ & $a^{3}\cdot a_{2}$
& $0$ \\\hline
$C_{10}$ &$
a^{3}-2a^{2}b-
2ab^{2}+2b^{3}
$ & $-a^{3}b^{3}+3a^{2}b^{4}-2ab^{5}$ & $(a^{3}-3a^{2}b+ab^{2})\cdot a_{2}$ & $0$\\\hline
$C_{12}$ & $
-a^{4}+2a^{3}b+2a^{2}b^{2}-
8ab^{3}+6b^{4}
$ & $b(a-2b)(a-b)^{2}(a^{2}-3ab+3b^{2})(a^{2}-2ab+2b^{2})
$ & $a(b-a)^3 \cdot a_{2} $& $0$ \\\hline
$C_{2}\times C_{2}$ & $0$ & $ad+bd$ & $0$ & $abd^{2}$ \\\hline
$C_{2}\times C_{4}$ & $a$ & $-ab-4b^{2}$ & $-a^{2}b-4ab^{2}$ & $0$ \\\hline
$C_{2}\times C_{6}$ & $-19a^{2}+2ab+b^{2}$ & $
-10a^{4}+22a^{3}b-
14a^{2}b^{2}+2ab^{3}
$ & $
90a^{6}-198a^{5}b+116a^{4}b^{2}+
4a^{3}b^{3}-14a^{2}b^{4}+2ab^{5}
$ & $0$ \\\hline
$C_{2}\times C_{8}$ & $
-a^{4}-8a^{3}b-
24a^{2}b^{2}+64b^{4}$ & $-4ab^{2}(a+2b)(a+4b)^{2}(a^{2}+4ab+8b^{2}) $ & $ -2b(a+4b)(a^{2}-8b^{2}) \cdot a_{2}
$ & $0$
\end{longtable}}

{\renewcommand*{\arraystretch}{1.2} 
\begin{longtable}{C{0.6in}C{4.9in}}
    \caption{The Polynomials $\alpha_{T}$}\\
    \hline
    $T$ &$\alpha_{T}$\\
    \hline
  \endfirsthead
    \caption[]{\emph{continued}}\\
    \hline
    $T$ & $\alpha_{T}$\\
    \hline
  \endhead
    \hline
    \multicolumn{2}{r}{\emph{continued on next page}}
  \endfoot
    \hline
  \endlastfoot

$C_{2}$ & $16(  3b^{2}d+a^{2})  $ \\\hline
$C_{3}^{0}$ &  $0$  \\\hline
$C_{3}$ & $a^{3}(  a-24b)  $ \\\hline
$C_{4}$ & $a^{2}(  a^{2}+16ab+16b^{2})  $ \\\hline
$C_{5}$ & $  a^{4}+12a^{3}b+14a^{2}b^{2}-12ab^{3}%
+b^{4}  $ \\\hline
$C_{6}$ & $(  a+3b)  (  a^{3}+9a^{2}b+3ab^{2}%
+3b^{3})  $ \\\hline
$C_{7}$ & $(  a^{2}-ab+b^{2})  (  a^{6}%
+5a^{5}b-10a^{4}b^{2}-15a^{3}b^{3}+30a^{2}b^{4}-11ab^{5}+b^{6})  $ \\\hline
$C_{8}$ & $  a^{8}-16a^{7}b+96a^{6}b^{2}-288a^{5}%
b^{3}+480a^{4}b^{4}-448a^{3}b^{5}+224a^{2}b^{6}-64ab^{7}+16b^{8}  $ \\\hline
$C_{9}$ & $(  a^{3}-3ab^{2}+b^{3})  (  a^{9}%
-9a^{7}b^{2}+27a^{6}b^{3}-45a^{5}b^{4}+54a^{4}b^{5}-48a^{3}b^{6}+27a^{2}%
b^{7}-9ab^{8}+b^{9})  $ \\\hline
$C_{10}$ & $a^{12}-8a^{11}b+16a^{10}b^{2}+40a^{9}b^{3}-240a^{8}b^{4}+432a^{7}%
b^{5}-256a^{6}b^{6}-288a^{5}b^{7}+720a^{4}b^{8}-720a^{3}b^{9}+416a^{2}b^{10}-
128ab^{11}+16b^{12}$ 
\\\hline
$C_{12}$ & $(  a^{4}-6a^{3}b+12a^{2}b^{2}-12ab^{3}+6b^{4})  (a^{12}
-18a^{11}b+144a^{10}b^{2}-684a^{9}b^{3}+2154a^{8}b^{4}-4728a^{7}b^{5}+7368a^{6}b^{6}-
8112a^{5}b^{7}+6132a^{4}b^{8}-3000a^{3}
b^{9}+864a^{2}b^{10}-144ab^{11}+24b^{12})$ \\\hline
$C_{2}\times
C_{2}$ & $16d^{2}(  a^{2}-ab+b^{2})  $ \\\hline
$C_{2}\times C_{4}$ & $a^{4}+16a^{3}b+80a^{2}b^{2}+128ab^{3}+256b^{4}$ \\\hline
$C_{2}\times C_{6}$ & $(  21a^{2}-6ab+b^{2})  (  6861a^{6}
-2178a^{5}b-825a^{4}b^{2}+180a^{3}b^{3}+75a^{2}b^{4}-18ab^{5}+b^{6})  $
\\\hline
$C_{2}\times C_{8}$ & $a^{16}+32a^{15}b+448a^{14}b^{2}+3584a^{13}b^{3}+17664a^{12}b^{4}
+51200a^{11}b^{5}+51200a^{10}b^{6}-237568a^{9}b^{7}- 1183744a^{8}b^{8}-1900544a^{7}b^{9}+3276800a^{6}
b^{10}+26214400a^{5}b^{11}+
72351744a^{4}b^{12}+117440512a^{3}b^{13}+ 117440512a^{2}b^{14}+67108864ab^{15}+16777216b^{16}$
\label{ta:alpT}	
\end{longtable}}

{\renewcommand*{\arraystretch}{1.2} 
\begin{longtable}{C{0.6in}C{4.9in}}
    \caption{The Polynomials $\beta_{T}$}\\
    \hline
    $T$ & $\beta_{T}$\\
    \hline
  \endfirsthead
    \caption[]{\emph{continued}}\\
    \hline
    $T$ & $\beta_{T}$ \\
    \hline
  \endhead
    \hline
    \multicolumn{2}{r}{\emph{continued on next page}}
  \endfoot
    \hline
  \endlastfoot
	
$C_{2}$ & $-64a(  9b^{2}d-a^{2})  $ \\\hline
$C_{3}^{0}$ &  $-216a^{2}$  \\\hline
$C_{3}$ & $a^{4}(  -a^{2}+36ab-216b^{2})  $ \\\hline
$C_{4}$ & $a^{3}(  a+8b)  (  -a^{2}-16ab+8b^{2})  $\\\hline
$C_{5}$ & $-(  a^{2}+b^{2})  (  a^{4}%
+18a^{3}b+74a^{2}b^{2}-18ab^{3}+b^{4})  $ \\\hline
$C_{6}$ & $-(  a^{2}+6ab-3b^{2}) (  a^{4}%
+12a^{3}b+30a^{2}b^{2}+36ab^{3}+9b^{4})  $ \\\hline

$C_{7}$ & $-(a^{12}+6a^{11}b-15a^{10}b^{2}-46a^{9}b^{3}+174a^{8}b^{4}-222a^{7}
b^{5}+273a^{6}b^{6}-486a^{5}b^{7}+570a^{4}b^{8}-354a^{3}b^{9}+
117a^{2}b^{10}-18ab^{11}+b^{12})$
\\\hline
$C_{8}$ & $
-(  a^{4}-8a^{3}b+16a^{2}b^{2}-16ab^{3}+8b^{4})  (a^{8}%
-16a^{7}b+96a^{6}b^{2}-288a^{5}b^{3}+456a^{4}b^{4}-352a^{3}b^{5}+
80a^{2}b^{6}+32ab^{7}-8b^{8})$ \\\hline
$C_{9}$ & $
-(a^{18}-18a^{16}b^{2}+42a^{15}b^{3}+27a^{14}b^{4}-306a^{13}b^{5}%
+735a^{12}b^{6}-1080a^{11}b^{7}+1359a^{10}b^{8}-
2032a^{9}b^{9}+ 3240a^{8}b^{10}- 4230a^{7}b^{11}+4128a^{6}b^{12}-2970a^{5}%
b^{13}1359a^{10}b^{8}-570a^{3}b^{15}+
135a^{2}b^{16}-18ab^{17}+b^{18})
$ \\\hline
$C_{10}$ & $-(  a^{2}-2ab+2b^{2})  (  a^{4}-2a^{3}b+2b^{4})
(  a^{4}-2a^{3}b-6a^{2}b^{2}+12ab^{3}-4b^{4})  (a^{8}-6a^{7}b
+4a^{6}b^{2}+48a^{5}b^{3}- 146a^{4}b^{4}+176a^{3}b^{5}-104a^{2}b^{6}
+32ab^{7}-4b^{8})$ \\\hline
$C_{12}$ & $-(  a^{8}-12a^{7}b+60a^{6}b^{2}-168a^{5}b^{3}+288a^{4}b^{4}-312a^{3}
b^{5}+216a^{2}b^{6}-96ab^{7}+24b^{8})  (a^{16}-
24a^{15}b+264a^{14}b^{2}+ 8208a^{12}b^{4}-27696a^{11}b^{5}+70632a^{10}
b^{6}-138720a^{9}b^{7}+211296a^{8}b^{8}-
248688a^{7}b^{9}+222552a^{6}b^{10}- 146304a^{5}b^{11}+65880a^{4}b^{12}
-17136a^{3}b^{13}+1008a^{2}b^{14}+
576ab^{15}-72b^{16})$ \\\hline
$C_{2}\times C_{2}$ & $-32d^{3}(  a+b)  (  a-2b) (
2a-b)  $ \\\hline
$C_{2}\times
C_{4}$ & $-(  a^{2}+8ab-16b^{2})  (
a^{2}+8ab+8b^{2}) (  a^{2}+8ab+32b^{2})  $ \\\hline
$C_{2}\times C_{6}$ & $-(  183a^{4}-36a^{3}b-30a^{2}b^{2}+12ab^{3}-b^{4})  (
393a^{4}-156a^{3}b+30a^{2}b^{2}-12ab^{3}+b^{4})  (759a^{4}-
228a^{3}b-30a^{2}b^{2}+ 12ab^{3}-b^{4}$ \\\hline
$C_{2}\times C_{8}$ & $-(  a^{8}+16a^{7}b+96a^{6}b^{2}+256a^{5}b^{3}-256a^{4}b^{4}%
-4096a^{3}b^{5}-12288a^{2}b^{6}-16384ab^{7}-8192b^{8})
( a^{8}+16a^{7}b+ 96a^{6}b^{2}+256a^{5}b^{3}+128a^{4}b^{4}-1024a^{3}%
b^{5}-3072a^{2}b^{6}-4096ab^{7}-2048b^{8})  (a^{8}+
16a^{7}b+96a^{6}b^{2}+ 256a^{5}b^{3}+ 512a^{4}b^{4}+2048a^{3}b^{5}%
+6144a^{2}b^{6}+8192ab^{7}+4096b^{8})$
\label{ta:betT}	
\end{longtable}}

{\renewcommand*{\arraystretch}{1.2} 
\begin{longtable}{C{0.6in}C{4.9in}}
    \caption{The Polynomials $\gamma_{T}$}\\
    \hline
    $T$ & $\gamma_{T} $\\
    \hline
  \endfirsthead
    \caption[]{\emph{continued}}\\
    \hline
 $T$ &  $\gamma_{T}$ \\
    \hline
  \endhead
    \hline
    \multicolumn{2}{r}{\emph{continued on next page}}
  \endfoot
    \hline
  \endlastfoot
	
$C_{2}$ & $64b^{2}d(  b^{2}d-a^{2})  ^{2}$\\\hline
$C_{3}^{0}$ & $-27a^{4}$   \\\hline
$C_{3}$ & $a^{8}b^{3}(  a-27b)  $ \\\hline
$C_{4}$ & $a^{7}b^{4}(  a+16b)  $\\\hline
$C_{5}$ & $(  ab)  ^{5}(  -a^{2}-11ab+b^{2})
$ \\\hline
$C_{6}$ & $a^{2}b^{6}(  a+9b) (  a+b)  ^{3}%
$ \\\hline
$C_{7}$ & $(  ab)  ^{7}(  -a+b)  ^{7}(
a^{3}+5a^{2}b-8ab^{2}+b^{3})  $\\\hline
$C_{8}$ & $a^{2}b^{8}(  a-2b)  ^{4}(  a-b)
^{8}(  a^{2}-8ab+8b^{2})  $\\\hline
$C_{9}$ & $(  ab)  ^{9}(  -a+b)  ^{9}(
a^{2}-ab+b^{2})  ^{3}(  a^{3}+3a^{2}b-6ab^{2}+b^{3})  $\\\hline
$C_{10}$ & $a^{5}b^{10}(  a-2b)  ^{5}(
a-b)  ^{10}(  a^{2}+2ab-4b^{2}) (  a^{2}-3ab+b^{2})  ^{2}$ \\\hline
$C_{12}$ & $a^{2}b^{12}(  a-2b)  ^{6}(  a-b)
^{12}(  a^{2}-6ab+6b^{2}) (  a^{2}-2ab+2b^{2})
^{3}(  a^{2}-3ab+3b^{2})  ^{4}$ \\\hline
$C_{2}\times C_{2}$ & $16a^{2}b^{2}d^{6}(  a-b)  ^{2}$ \\\hline
$C_{2}\times C_{4}$ & $a^{2}b^{4}(  a+8b)  ^{2}(  a+4b)
^{4}$ \\\hline
$C_{2}\times C_{6}$ & $(  2a)  ^{6}(  -9a+b)  ^{2}(
-3a+b)  ^{2}(  3a+b)  ^{2}(  -5a+b)  ^{6}(
-a+b)  ^{6}$ \\\hline
$C_{2}\times C_{8}$ & $(  2ab)  ^{8}(  a+2b)  ^{8}(
a+4b)  ^{8}(  a^{2}-8b^{2})  ^{2}(  a^{2}+8ab+8b^{2}%
)  ^{2}(  a^{2}+4ab+8b^{2})  ^{4}$ 
\label{ta:gamT}	
\end{longtable}}

\bibliographystyle{amsplain}
\bibliography{bibliography}

\end{document}